\documentclass{svjour3}                     
\smartqed  
\usepackage{color}
\usepackage{amsmath,amsfonts,amssymb,bm,mathtools}
\usepackage{graphicx,psfrag,subfigure}
\usepackage{enumitem}
\usepackage{epstopdf}
\usepackage{latexsym}
\usepackage{mathrsfs}
\usepackage{empheq}
\usepackage{fullpage}
\numberwithin{equation}{section}

\renewcommand{\d}{\mathrm{d}}
\newcommand{\e}{\mathrm{e}}
\newcommand{\R}{\mathbb{R}}

\newcommand{\Z}{\mathbb{Z}}

\newcommand{\bv}{\bm{v}}
\newcommand{\bw}{\bm{w}}
\newcommand{\bx}{\bm{x}}

\newcommand{\hE}{\mathcal{E}}

\newcommand{\Mh}{\mathcal{M}_h}
\newcommand{\<}{\langle}
\renewcommand{\>}{\rangle}

\newcommand{\uinit}{u_{\text{\rm init}}}
\newcommand{\kp}{\kappa}

\newcommand{\nn}{\nonumber}
\newcommand{\dt}{{\tau}}
\newcommand{\eps}{\varepsilon}
\newcommand{\daoshu}[2]{\frac{\d #1}{\d #2}}

\newcommand{\expp}[1]{\exp\{#1\}}

\newcommand{\Cp}{C_\text{\rm per}}

\begin{document}

\title{Stabilized exponential-SAV schemes  preserving energy dissipation law and maximum bound principle
for the Allen--Cahn type equations
\thanks{This work is supported by the CAS AMSS-PolyU Joint Laboratory of Applied Mathematics.
L. Ju's work is partially supported by US National Science Foundation grant DMS-2109633
and US Department of Energy grant DE-SC0020270.
X. Li's work is partially supported by National Natural Science Foundation of China grant 11801024
and the Hong Kong Polytechnic University grants 4-ZZMK and 1-BD8N.
Z. Qiao's work is partially supported by the Hong Kong Research Council RFS grant RFS2021-5S03
and GRF grants 15300417 and 15302919.}
}

\titlerunning{sESAV schemes for Allen--Cahn type equations}        

\author{Lili Ju \and Xiao Li \and Zhonghua Qiao
}


\institute{L. Ju \at
              Department of Mathematics, University of South Carolina, Columbia, SC 29208, USA\\
              \email{ju@math.sc.edu}           
           \and
           X. Li \at
              Corresponding author. Department of Applied Mathematics, The Hong Kong Polytechnic University, Hung Hom, Kowloon, Hong Kong\\
              \email{xiao1li@polyu.edu.hk}
           \and
           Z. Qiao \at
              Department of Applied Mathematics \& Research Institute for Smart Energy, The Hong Kong Polytechnic University, Hung Hom, Kowloon, Hong Kong\\
              \email{zqiao@polyu.edu.hk}
}

\date{Received: date / Accepted: date}

\maketitle

\begin{abstract}
It is well-known that the Allen--Cahn equation not only satisfies the energy dissipation law but also possesses the maximum bound principle (MBP) in the sense that the absolute value of its solution is pointwise bounded for all time by some specific constant under appropriate initial/boundary conditions. In recent years, the scalar auxiliary variable (SAV) method and many of its variants have attracted much attention in numerical solution for gradient flow problems due to their inherent advantage of  preserving certain discrete analogues of the energy dissipation law. However, existing SAV schemes usually fail to preserve the MBP when applied to the Allen--Cahn equation. In this paper, we develop and analyze new first- and second-order stabilized exponential-SAV schemes for a class of Allen--Cahn type equations, which are shown to simultaneously preserve the energy dissipation law and MBP in discrete settings. In addition, optimal error estimates for the numerical solutions are rigorously obtained for both schemes. Extensive numerical tests and comparisons are also conducted to demonstrate the performance of the proposed schemes.
\keywords{maximum bound principle \and
energy dissipation \and
stabilized method \and
exponential scalar auxiliary variable}
\subclass{35K55 \and 65M12 \and 65M15 \and 65F30}
\end{abstract}

\section{Introduction}

Let us consider a class of reaction-diffusion equations taking the following form
\begin{equation}
\label{AllenCahn}
u_t = \eps^2\Delta u + f(u), \quad t > 0, \ \bx \in \Omega,
\end{equation}
where $\Omega\subset\R^d$ is a spatial domain,
$u=u(t,\bx):[0,\infty)\times\overline{\Omega}\to\R$ is the unknown function,
$\eps>0$ denotes an interfacial parameter,
and $f(u)$ is a nonlinear reaction term with $f$ being continuously differentiable.
We also impose the initial condition
\[
u(0,\cdot)=\uinit \quad \text{on } \overline{\Omega}
\]
and the periodic or homogeneous Neumann boundary conditions.
The equation \eqref{AllenCahn} usually can be regarded as
the $L^2$ gradient flow with respect to the energy functional
\begin{equation}
\label{energy}
E(u) = \int_\Omega \Big( \frac{\eps^2}{2}|\nabla u(\bx)|^2 + F(u(\bx)) \Big) \, \d\bx,
\end{equation}
where $F$ is a smooth potential function satisfying $F'=-f$,
and thus, the solution to the equation \eqref{AllenCahn} decreases the energy \eqref{energy} along with the time,
i.e., $\daoshu{}{t}E(u(t))\le0$,
which is often called the {\em energy dissipation law}.
In addition, we also assume that
\begin{equation}
\label{assump}
\text{there exists a constant $\beta>0$ such that $f(\beta) \le 0\le f(-\beta)$}.
\end{equation}
It has been proved in \cite{DuJuLiQi21} that
the equation \eqref{AllenCahn} satisfies the {\em maximum bound principle} (MBP) in the sense that
if the absolute value of the initial data is bounded pointwise by $\beta$,
then the absolute value of the solution is also bounded by $\beta$ pointwise for all time, i.e.,
\begin{equation}
\label{mbp}
\max_{\bx\in\overline{\Omega}} |\uinit(\bx)| \le \beta \quad \Longrightarrow \quad
\max_{\bx\in\overline{\Omega}} |u(t,\bx)| \le \beta, \quad \forall \, t > 0.
\end{equation}

An important and special case of \eqref{AllenCahn} is the Allen--Cahn equation with $f(u)=u-u^3$,
which was originally introduced in \cite{AlCa79} to model the motion of anti-phase boundaries in crystalline solids.
The solution represents the difference between
the concentrations of two components of the alloy
and thus should be evaluated between $-1$ and $1$,
which is  guaranteed by the MBP.
With the corresponding double-well potential $F(u)=\frac{1}{4}(u^2-1)^2$,
the associated energy functional \eqref{energy} decays in time,
which reflects the energy dissipation of the phase transition process.
Both the MBP and the energy stability are also satisfied by some variants of \eqref{AllenCahn},
such as the nonlocal Allen--Cahn equation for phase separations within long-range interactions \cite{Bates06,DuJuLiQi19}
and the fractional Allen--Cahn equation used to describe some anomalous diffusion processes \cite{DuYaZh20,GuiZh15}.
To obtain stable numerical simulations and avoid nonphysical solutions for these models,
it is highly desirable to design  numerical schemes preserving effectively  these two basic physical properties, the MBP and the energy dissipation law in time discrete settings.

In the past decades,
there has been a large amount of research denoted to energy-stable numerical schemes for time discretization of gradient flow equations,
such as convex splitting schemes \cite{GuWaWi14,ShWaWaWi12,WiWaLo09},
stabilized semi-implicit schemes \cite{FeTaYa13,ShYa10b,XuTa06},
and exponential time differencing (ETD) schemes \cite{DuJuLiQi19,JuLiQiZh18,JuZhDu15}.
More recently, invariant energy quadratization (IEQ) schemes \cite{XuYaZhXi19,Yang16,YangZh20}
and scalar auxiliary variable (SAV) schemes \cite{ShXu18,ShXuYa18,ShXuYa19}
were proposed to  naturally provide energy-stable and linear algorithms with second-order temporal accuracy.
While the main idea for both methods is  to reformulate and split the energy functional \eqref{energy} in the quadratic form by introducing extra variables, the SAV approach is usually more efficient in terms of computations.
Many variants of SAV schemes were developed later;
see \cite{AkrivisLiLi19,ChenYa19,ChengLiSh20,ChengLiSh21,HouAzXu19,HuangShYa20,LiuLi20} and the references therein.
In practice, a suitable stabilization term is also introduced in such  splitting in order to maintain numerical stability
for highly stiff problems. On the other hand, existing SAV-type schemes usually fail to preserve the MBP, and a special case is the auxiliary variable proposed in \cite{HuangShYa20} which is shown to be positivity-preserving.

The MBP preservation recently has also attracted increasingly attention in the field of numerical methods for the Allen--Cahn type equations of the form \eqref{AllenCahn}.
The semi-implicit schemes were extensively studied in,
e.g. \cite{HoLe20,HoTaYa17,LiaoTaZh20,ShTaYa16,TaYa16,XiFeYu17},
for the classic, fractional, or surface Allen--Cahn equations.
The first- and second-order stabilized ETD schemes were shown to
preserve the MBP unconditionally for the nonlocal Allen--Cahn equation \cite{DuJuLiQi19}
and the conservative Allen--Cahn equation \cite{LiJuCaFe21}.
An abstract framework on MBP preservation of the ETD schemes
for a class of semilinear parabolic equations was established in \cite{DuJuLiQi21},
where sufficient conditions for the linear and nonlinear operators are presented in order to guarantee the MBP.
A family of  stabilized integrating factor Runge-Kutta (IFRK) schemes, up to third order, were developed in \cite{LiLiJuFe21}, which can
unconditionally preserve the MBP. In addition, a fourth-order (conditionally) MBP-preserving  IFRK scheme was presented in \cite{JuLiQiYa21}.
So far, as one of the very popular methods, there is still not much systematical study on MBP-preserving SAV schemes.

The main goal of this paper is to develop
first- and second-order energy dissipative and MBP-preserving SAV-type schemes for the Allen--Cahn type equation \eqref{AllenCahn}
by using an appropriate stabilization technique.
Specifically,
we propose new stabilized  exponential-SAV (ESAV) schemes
by introducing an artificial stabilization term rather than basing on the splitting of the energy functional suggested in \cite{ShXuYa19}.
With the effect of such stabilization, we show that
the proposed first-order scheme  preserves the MBP unconditionally with an appropriate stabilizing parameter
and the second-order one does under a time step size constraint.
A main difficulty for numerical analysis of the two schemes lies in that
the coefficients of the nonlinear term and the stabilization term are varying rather than constant
due to the use of the SAV approach.
With the help of the energy dissipation and MBP,  we are able to show that such variable coefficients are bounded from above and below by certain positive constants, and consequently, optimal error estimate are successfully obtained for the proposed schemes.
To the best of our knowledge, this is the first work in the direction of designing such  SAV-type methods. More importantly, the proposed stabilizing approaches can be easily generalized to deal with many other type of gradient flow problems where the SAV methods apply.

The rest of this paper is organized as follows.
Section \ref{sect_spacedis} is devoted to spatial discretization of the equation \eqref{AllenCahn}
and a brief summary of the classic SAV and ESAV schemes for the time integration. Then, our first- and second-order stabilized ESAV schemes  are presented in Section \ref{sect_sESAVsch},
together with the energy dissipation law, MBP preservation, and convergence analysis of the resulting fully discrete systems.
In Section \ref{sect_experiment},
extensive numerical tests and comparisons are carried out to demonstrate the performance of the proposed schemes.
Some concluding remarks are finally given in Section \ref{sect_conclusion}.

\section{Spatial discretization and SAV schemes for time integration}
\label{sect_spacedis}

For simplicity, throughout this paper we consider the two-dimensional square domain $\Omega=(0,L)\times(0,L)$
for the equation \eqref{AllenCahn} equipped with periodic boundary conditions. Note that the extensions to three-dimension problems and
homogeneous Neumann boundary condition are straightforward. For other feasible spatial discretization, we refer to \cite{DuJuLiQi21} for more details.
In this section, we first present some notations related to the spatial discretization by central finite difference, then
briefly review the classic SAV schemes  for time integration.

\subsection{Spatial discretization and the space-discrete problem}

Given a positive integer $M$,
we set $h=L/M$ to be the size of the uniform mesh partitioning $\overline{\Omega}$.
Denote by $\Omega_h$ the set of mesh points $(x_i,y_j)=(ih,jh)$, $1\le i,j\le M$.
For a grid function $v$ defined on $\Omega_h$, we write $v_{ij}=v(x_i,y_j)$ for simplicity.
Let $\Mh$ be the set of all periodic grid functions on $\Omega_h$, i.e.,
\[
\Mh = \{v:\Omega_h\to\R \,|\, v_{i+kM,j+lM}=v_{ij}, \ k,l\in\Z, \ 1 \le i,j\le M\}.
\]
The discrete inner product $\<\cdot,\cdot\>$, discrete $L^2$ norm $\|\cdot\|$,
and discrete $L^\infty$ norm $\|\cdot\|_\infty$ can be defined as usual, namely,
\[
\<v,w\> = h^2 \sum_{i,j=1}^M v_{ij} w_{ij}, \qquad
\|v\| = \sqrt{\<v,v\>}, \qquad
\|v\|_\infty = \max_{1\le i,j\le M} |v_{ij}|
\]
for any $v,w\in\Mh$, and
\[
\<\bv,\bw\> = \<v^1,w^1\> + \<v^2,w^2\>, \qquad
\|\bv\| = \sqrt{\<\bv,\bv\>}
\]
for any $\bv=(v^1,v^2)^T, \bw=(w^1,w^2)^T\in\Mh\times\Mh$.
We apply the second-order central finite difference  to approximate spatial differentiation operators.
For any $v\in\Mh$, the discrete Laplace operator $\Delta_h$ is defined by
\[
\Delta_h v_{ij} = \frac{1}{h^2} (v_{i+1,j}+v_{i-1,j}+v_{i,j+1}+v_{i,j-1}-4v_{ij}), \quad 1 \le i,j \le M,
\]
and the discrete gradient operator $\nabla_h$ is defined by
\[
\nabla_h v_{ij} = \Big( \frac{v_{i+1,j}-v_{ij}}{h}, \frac{v_{i,j+1}-v_{ij}}{h}\Big)^T, \quad 1 \le i,j \le M.
\]
By periodic boundary conditions,
the summation-by-parts formula is easy to verify:
\[
\<v,\Delta_hw\> = - \<\nabla_hv, \nabla_hw\> = \<\Delta_hv,w\>, \quad \forall\,v,w\in\Mh.
\]
Obviously, $\Delta_h$ is self-adjoint and negative semi-definite.
For any function $\varphi:\overline{\Omega}\to\R$,
we denote by $I_h$ the operator projecting $\varphi$ on the mesh as  $(I_h\varphi)_{ij}=\varphi(x_i,y_j)$ for $1\le i,j\le M$.
For example, we have
\[
\max_{1\le i,j\le N} |\Delta_h(I_h\varphi)_{ij} - \Delta\varphi(x_i,y_j)| \le C_\varphi h^2, \quad
\forall \, \varphi \in \Cp^4(\overline{\Omega}).
\]
For simplicity, we may directly omit the notation $I_h$ when there is no ambiguity.

Since $\Mh$ is a finite-dimensional linear space,
any grid function $v\in\Mh$ and any linear operator $Q:\Mh\to\Mh$
can be regarded as a vector in $\R^{M^2}$ and a matrix in $\R^{M^2\times M^2}$, respectively.
We still use the notations $\|\cdot\|$ and $\|\cdot\|_\infty$
to denote the matrix induced-norms consistent with $\|\cdot\|$ and $\|\cdot\|_\infty$ defined for vectors before, respectively.
By regarding $\Delta_h$ as a linear operator,
we know that $\Delta_h$ is the generator of a contraction semigroup on $\Mh$ \cite{DuJuLiQi21}.
Instead, by viewing $\Delta_h$ as a matrix,
it is weakly diagonally dominant with all diagonal entries negative.
Moreover, we have the following useful estimate
and the proof can be found in \cite{TaYa16}.

\begin{lemma}
\label{lem_lapdiff}
For any $a>0$, we have $\|(aI-\Delta_h)^{-1}\|_\infty \le a^{-1}$, where $I$ represents the identity matrix.
\end{lemma}

We have assumed that $f$ is continuously differentiable,
so $\|f'\|_{C[-\beta,\beta]}$ is always finite and then the following result is valid \cite{DuJuLiQi21}.

\begin{lemma}
\label{lem_nonlinear}
Under the assumption \eqref{assump},
if $\kp\ge \|f'\|_{C[-\beta,\beta]}$ holds for some positive constant $\kp$,
then we have $|f(\xi)+\kp\xi|\le\kp\beta$ for any $\xi\in[-\beta,\beta]$.
\end{lemma}

Next, let us introduce the space-discrete version of \eqref{AllenCahn}.
The space-discrete problem is to find a function $u_h:[0,\infty)\to\Mh$ satisfies
\begin{equation}
\label{semidis}
\daoshu{u_h}{t} = \eps^2\Delta_h u_h + f(u_h)
\end{equation}
with $u_h(0) = \uinit$.
It is easy to verify the energy dissipation law for \eqref{semidis} in the sense that
\begin{equation*}
\daoshu{}{t} E_h(u_h(t)) \le 0,
\end{equation*}
where $E_h$ is the spatially-discretized energy functional defined as
\begin{equation}\label{egydis}
E_h(v) := \frac{\eps^2}{2} \|\nabla_h v\|^2 + \<F(v),1\>, \quad \forall \, v \in \Mh.
\end{equation}
According to \cite{DuJuLiQi21}, the MBP also holds for $u_h$, i.e.,
$\|u_h(t)\|_\infty\le\beta$ for any $t>0$  if $\|\uinit\|_\infty\le\beta$.

Let us partition the time interval into $\{t_n=n\dt\}_{n\ge0}$ with $\dt>0$ being a uniform time step size.
In the remaining part of the paper, we will study time integration schemes for the space-discrete system \eqref{semidis}.
For simplicity of representation, we denote by $u^n$ the fully discrete approximate value of $u_e(t_n)$ or $u_{h,e}(t_n)$
with $u_e$ and $u_{h,e}$ denoting the exact solutions to
the original continuous problem \eqref{AllenCahn} and the space-discrete problem \eqref{semidis}, respectively.
In general, for a sequence $\{v^n\}$, we define the following notations:
\[
\delta_t v^{n+1} = \frac{v^{n+1}-v^n}{\dt}, \qquad
v^{n+\frac{1}{2}} = \frac{v^{n+1}+v^n}{2}.
\]

\subsection{Classic SAV schemes and stabilization}
\label{sect_classicSAV}

Here we give a brief summary of the classic SAV schemes.
The main idea of SAV is to reformulate the energy functional \eqref{energy} in the quadratic form
by introducing an appropriate SAV.
The framework of the classic SAV schemes is based on a linear splitting of the energy functional
and, as shown in \cite{ShXuYa19}, a suitable stabilization term is usually also introduced in such  splitting
so that the numerical simulations can provide satisfactory results for highly stiff problems in practice.

Denoting by $\kp\ge0$ the stabilizing constant,
the energy functional \eqref{energy} can be rewritten with a stabilization term as
\begin{align}
\label{sav_splitting}
E(u) &= \int_\Omega \Big( \frac{\eps^2}{2}|\nabla u|^2 + \frac{\kp}{2}u^2 + F(u) - \frac{\kp}{2}u^2 \Big) \, \d\bx\nonumber\\
&= \frac{\eps^2}{2}\|\nabla u\|_{L^2}^2 + \frac{\kp}{2}\|u\|_{L^2}^2
+ \int_\Omega \Big( F(u) - \frac{\kp}{2}u^2 \Big) \, \d\bx.
\end{align}
Suppose the last term in \eqref{sav_splitting} is bounded from below, that is,
\begin{equation}
\label{sav_e2}
E_2(u) := \int_\Omega \Big( F(u) - \frac{\kp}{2}u^2 \Big) \, \d\bx \ge -C_0
\end{equation}
for some constant $C_0\geq 0$.
Choosing $\delta>C_0$,
let us  define an auxiliary variable $r(t)=\sqrt{E_2(u(t))+\delta}$, and reformulate the original problem \eqref{AllenCahn} to the following equivalent system:
\begin{subequations}
\begin{align*}
u_t & = \eps^2\Delta u - \kp u + \frac{r}{\sqrt{E_2(u)+\delta}} (f(u)+\kp u), \\
r_t & = - \frac{1}{2\sqrt{E_2(u)+\delta}} (f(u)+\kp u,u_t).
\end{align*}
\end{subequations}
Then the first-order SAV scheme (SAV1) is given by \cite{ShXuYa19}
\begin{subequations}
\label{eq_sav1shen}
\begin{align}
\delta_t u^{n+1} & = \eps^2\Delta_h u^{n+1} - \kp u^{n+1} + \frac{r^{n+1}}{\sqrt{E_{2h}(u^n)+\delta}}(f(u^n)+\kp u^n),
\label{eq_sav1shena} \\
\delta_t r^{n+1} & = - \frac{1}{2\sqrt{E_{2h}(u^n)+\delta}} \<f(u^n)+\kp u^n,\delta_t u^{n+1}\>,
\end{align}
\end{subequations}
and the Crank--Nicolson type second-order SAV scheme (SAV2) reads as \cite{ShXuYa19}
\begin{subequations}
\label{eq_sav2shen}
\begin{align}
\delta_t u^{n+1}
& = \eps^2\Delta_h u^{n+\frac{1}{2}} - \kp u^{n+\frac{1}{2}}
+ \frac{r^{n+1}+r^n}{2\sqrt{E_{2h}(\widehat{u}^{n+\frac{1}{2}})+\delta}}
(f(\widehat{u}^{n+\frac{1}{2}})+\kp \widehat{u}^{n+\frac{1}{2}}), \\
\delta_t r^{n+1}
& = - \frac{1}{2\sqrt{E_{2h}(\widehat{u}^{n+\frac{1}{2}})+\delta}}
\<f(\widehat{u}^{n+\frac{1}{2}})+\kp \widehat{u}^{n+\frac{1}{2}},\delta_t u^{n+1}\>,
\end{align}
\end{subequations}
where 
$\widehat{u}^{n+\frac{1}{2}}$ is generated by solving the system
\[
\frac{\widehat{u}^{n+\frac{1}{2}}-u^n}{\dt/2}
= \eps^2\Delta_h \widehat{u}^{n+\frac{1}{2}} + f(u^n) - \kp (\widehat{u}^{n+\frac{1}{2}}-u^n).
\]
Both \eqref{eq_sav1shen} and \eqref{eq_sav2shen} are linear schemes and energy dissipative
in the sense that $\overline{E}_h(u^{n+1},r^{n+1})\le \overline{E}_h(u^n,r^n)$
with respect to the following modified energy
\[
\overline{E}_h(u^n,r^n) := \frac{\eps^2}{2}\|\nabla_h u^n\|^2 + \frac{\kp}{2}\|u^n\|^2 + (r^n)^2 - \delta.
\]
Note that in the discrete settings, $\overline{E}_h(u^n,r^n)$ is only an approximation of the original discrete energy $E_h(u^n)$ defined in \eqref{egydis} and they are not equal in general
since $r^n\not=\sqrt{E_{2h}(u^n)+\delta}$ for $n\ge 1$.
However, the MBP cannot be theoretically preserved by the above classic SAV schemes \eqref{eq_sav1shen} and \eqref{eq_sav2shen}
(See the discussion in Remark \ref{rmk_sav1_nonmbp}).

\subsection{Exponential-SAV schemes}
\label{sect_ESAV}

A variant of the classic SAV approach, called the exponential-SAV (ESAV) scheme, was studied in \cite{LiuLi20}.
We below  summarize the ESAV method, also with a stabilization term based on the energy splitting \eqref{sav_splitting}.
Define the auxiliary variable by $r(t)=\expp{E_2(u(t))}$
and reformulate \eqref{AllenCahn} as
\begin{subequations}
\begin{align*}
& u_t = \eps^2\Delta u - \kp u + \frac{r}{\expp{E_2(u)}} (f(u)+\kp u), \\
& (\ln r)_t = - \frac{r}{\expp{E_2(u)}} (f(u)+\kp u,u_t).
\end{align*}
\end{subequations}
Then the first-order ESAV scheme (ESAV1) reads as
\begin{subequations}
\label{eq_esav1}
\begin{align}
& \delta_t u^{n+1} = \eps^2\Delta_h u^{n+1} - \kp u^{n+1} + \frac{r^n}{\expp{E_{2h}(u^n)}}(f(u^n)+\kp u^n),
\label{eq_esav1a} \\
& \frac{\ln r^{n+1} - \ln r^n}{\dt} = - \frac{r^n}{\expp{E_{2h}(u^n)}} \<f(u^n)+\kp u^n,\delta_t u^{n+1}\>.
\end{align}
\end{subequations}
Setting $\kp=0$, the scheme \eqref{eq_esav1} reduces exactly to the original ESAV scheme (without stabilization) presented in \cite{LiuLi20}.
The Crank--Nicolson type ESAV scheme (ESAV2) is given by
\begin{subequations}
\label{eq_esav2}
\begin{align}
& \delta_t u^{n+1} = \eps^2\Delta_h u^{n+\frac{1}{2}} - \kp u^{n+\frac{1}{2}}
+ \frac{\widehat{r}^{n+\frac{1}{2}}}{\expp{E_{2h}(\widehat{u}^{n+\frac{1}{2}})}}(f(\widehat{u}^{n+\frac{1}{2}})+\kp \widehat{u}^{n+\frac{1}{2}}),
\label{eq_esav2a} \\
& \frac{\ln r^{n+1} - \ln r^n}{\dt} = - \frac{\widehat{r}^{n+\frac{1}{2}}}{\expp{E_{2h}(\widehat{u}^{n+\frac{1}{2}})}}
\<f(\widehat{u}^{n+\frac{1}{2}})+\kp \widehat{u}^{n+\frac{1}{2}},\delta_t u^{n+1}\>,
\end{align}
\end{subequations}
where the value $(\widehat{u}^{n+\frac{1}{2}},\widehat{r}^{n+\frac{1}{2}})$
can be generated by an extrapolation as suggested in \cite{LiuLi20}
or predicted by the first-order scheme \eqref{eq_esav1} with half of the time step size:
\begin{subequations}
\label{eq_esav21}
\begin{align}
& \frac{\widehat{u}^{n+\frac{1}{2}}-u^n}{\dt/2}
= \eps^2\Delta_h \widehat{u}^{n+\frac{1}{2}} - \kp \widehat{u}^{n+\frac{1}{2}} + \frac{r^n}{\expp{E_{2h}(u^n)}}(f(u^n)+\kp u^n), \\
& \ln \widehat{r}^{n+\frac{1}{2}} - \ln r^n
= - \frac{r^n}{\expp{E_{2h}(u^n)}} \<f(u^n)+\kp u^n,\widehat{u}^{n+\frac{1}{2}}-u^n\>.
\end{align}
\end{subequations}
We will adopt \eqref{eq_esav21} in the numerical experiments for the comparison.
Both \eqref{eq_esav1} and \eqref{eq_esav2} are energy dissipative in the sense that $\widetilde{E}_h(u^{n+1},r^{n+1})\le \widetilde{E}_h(u^n,r^n)$ with respect to the following modified energy
\[
\widetilde{E}_h(u^n,r^n) := \frac{\eps^2}{2}\|\nabla_h u^n\|^2 + \frac{\kp}{2}\|u^n\|^2 + \ln r^n.
\]
Similar to the classic SAV schemes,
the above ESAV schemes \eqref{eq_esav1} and \eqref{eq_esav2} also cannot preserve the MBP (See the discussion in Remark \ref{rmk_sav1_nonmbp}).

\section{New stabilized  exponential-SAV schemes}
\label{sect_sESAVsch}

From now on, we always assume the initial value $\uinit$ has the enough regularity as needed.
By spatial discretization, there is a constant $h_0>0$,
depending on $\uinit$, $F$, $\eps$, such that
\begin{equation}
\label{uinit_approx}
E_h(\uinit) \le E(\uinit) + 1, \quad
\|\Delta_h\uinit\| \le \|\Delta\uinit\|_{L^2} + 1,
\qquad \forall \, h \in (0,h_0].
\end{equation}
The continuity of $F$ implies that $F$ is bounded from below on $[-\beta,\beta]$.
Therefore, according to the MBP \eqref{mbp},
it holds that
$$E_1(u):=\int_\Omega F(u) \,\d\bx\ge -C_*$$
for some constant $C_*\geq 0$.
Introducing $s(t) = E_1(u(t))$,
we then have the following energy which is equivalent to $E(u)$:
\[
\hE(u,s) = \frac{\eps^2}{2} \|\nabla u\|_{L^2}^2 + s.
\]
Partially inspired by the idea of ESAV method \cite{LiuLi20},
we rewrite the equation \eqref{AllenCahn} as the following equivalent system:
\begin{subequations}
\begin{align*}
u_t & = \eps^2\Delta u + \frac{\expp{s}}{\expp{E_1(u)}} f(u), \\
s_t & = - \frac{\expp{s}}{\expp{E_1(u)}} (f(u),u_t).
\end{align*}
\end{subequations}
The corresponding space-discrete problem is
to find $u_h(t)\in\Mh$ and $s_h(t)$ for $t>0$ satisfies
\begin{subequations}
\label{eq_semidis}
\begin{align}
\daoshu{u_h}{t} & = \eps^2\Delta_h u_h + g(u_h,s_h) f(u_h), \label{eq_semidisa} \\
\daoshu{s_h}{t} & = - g(u_h,s_h) \Big\<f(u_h),\daoshu{u_h}{t}\Big\>, \label{eq_semidisb}
\end{align}
\end{subequations}
where
\begin{equation}
\label{gg}
g(u_h,s_h) := \frac{\expp{s_h}}{\expp{E_{1h}(u_h)}} > 0,
\end{equation}
and  $E_{1h}$ denotes the space-discrete version of $E_1$, i.e.,
$E_{1h}(v) := \<F(v),1\>$ for any $v \in \Mh$.
Based on such an equivalent form, we will give the stabilized ESAV schemes in the fully discrete version.
This section is devoted to the first-order scheme
and the second-order one will be discussed in the next section.
Recall that we use $u^n$ to represent the fully discrete approximate value of $u_e(t_n)$,
the exact solution to the problem  \eqref{AllenCahn}.

\subsection{First-order sESAV scheme}
\label{sect_sav1}

The first-order stabilized ESAV fully-discrete scheme (sESAV1) is given by
\begin{subequations}
\label{eq_sav1stab}
\begin{align}
\delta_t u^{n+1} & = \eps^2\Delta_h u^{n+1} +  g(u^n,s^n)f(u^n) - \kp  g(u^n,s^n)(u^{n+1}-u^n), \label{eq_sav1staba} \\
\delta_t s^{n+1} & = -  g(u^n,s^n)\<f(u^n),\delta_t u^{n+1}\>, \label{eq_sav1stabb}
\end{align}
\end{subequations}
where  $\kp\ge0$ is a stabilizing constant and $g(u^n,s^n)>0.$
The scheme \eqref{eq_sav1stab} is started by $u^0=\uinit$ and $s^0=E_{1h}(u^0)$.
We can rewrite \eqref{eq_sav1stab} equivalently as follows:
\begin{subequations}
\label{eq_sav1var}
\begin{align}
\Big[\Big(\frac{1}{\dt}+\kp  g(u^n,s^n)\Big)I-\eps^2\Delta_h\Big] u^{n+1}
& = \frac{u^n}{\dt} +g(u^n,s^n)f(u^n) + \kp  g(u^n,s^n)u^n,  \label{lem_sav1_pf1} \\
s^{n+1} & = s^n -  g(u^n,s^n)\<f(u^n),u^{n+1}-u^n\>. \label{lem_sav1_pf2}
\end{align}
\end{subequations}
Obviously, \eqref{eq_sav1var} is uniquely solvable for any $\dt>0$
since $(\frac{1}{\dt}+\kp  g(u^n,s^n))I-\eps^2\Delta_h$ is self-adjoint and positive definite,
which makes $u^{n+1}$ linearly determined from \eqref{lem_sav1_pf1}
and then $s^{n+1}$ computed explicitly  by \eqref{lem_sav1_pf2}.
If we take $\kp=0$ and $r^n=\expp{s^n}$, i.e., $s^n=\ln r^n$,
it is easy to verify that the scheme \eqref{eq_sav1stab}  gives us exactly the ESAV scheme \eqref{eq_esav1} with $\kappa=0$.
However, they differ when $\kappa>0$.

\subsubsection{Energy dissipation and MBP}
Now let us define a discrete energy as follows
\begin{equation}
\label{modified_energy}
\hE_h(u^n,s^n) := \frac{\eps^2}{2} \|\nabla_h u^n\|^2 + s^n,
\end{equation}
which is clearly again an approximation of the original discrete energy $E_h(u^n)$.
We first show that the  sESAV1 scheme \eqref{eq_sav1stab}  preserves the energy dissipation law and the MBP unccondtionally.
Then, as an application of both properties, we also prove the uniform boundedness of the variable coefficient $  g(u^n,s^n)$.

\begin{theorem}[Energy dissipation of sESAV1]
\label{thm_sav1_es}
For any $\kp\ge0$ and $\dt>0$, the sESAV1 scheme \eqref{eq_sav1stab} is energy dissipative
in the sense that $\hE_h(u^{n+1},s^{n+1})\le \hE_h(u^n,s^n)$.
\end{theorem}

\begin{proof}
Taking the inner product with \eqref{eq_sav1stab} by $u^{n+1}-u^n$ yields
\begin{equation}
\label{sav1_es_pf1}
\Big(\frac{1}{\dt}+\kp  g(u^n,s^n)\Big)\|u^{n+1}-u^n\|^2
= \eps^2\<\Delta_h u^{n+1},u^{n+1}-u^n\> +  g(u^n,s^n)\<f(u^n),u^{n+1}-u^n\>.
\end{equation}
Combining \eqref{sav1_es_pf1}, \eqref{eq_sav1stabb}, and the identity
\[
\<\Delta_hu^{n+1},u^{n+1}-u^n\>
=-\frac{1}{2}\|\nabla_hu^{n+1}\|^2+\frac{1}{2}\|\nabla_hu^n\|^2-\frac{1}{2}\|\nabla_hu^{n+1}-\nabla_hu^n\|^2,
\]
we obtain
\[
\hE_h(u^{n+1},s^{n+1}) - \hE_h(u^n,s^n)
= -\Big(\frac{1}{\dt}+\kp  g(u^n,s^n)\Big)\|u^{n+1}-u^n\|^2-\frac{\eps^2}{2}\|\nabla_hu^{n+1}-\nabla_hu^n\|^2,
\]
which completes the proof.
\end{proof}

Theorem \ref{thm_sav1_es} implies that the scheme \eqref{eq_sav1stab} is energy dissipative
with respect to the modified energy $\hE_h(u^n,s^n)$ rather than the original energy $E_h(u^n)$.
Note that $s^n\not=E_{1h}(u^n)$ for $n\ge 1$ in general,
and thus $\hE_h(u^n,s^n)\not=E_h(u^n)$.

\begin{corollary}
\label{cor_sav1_es}
For any $\kp\ge0$ and $\dt>0$, it holds $s^n\le E_h(\uinit)$ for all $n$.
\end{corollary}

\begin{proof}
By Theorem \ref{thm_sav1_es}, since $s^0=E_{1h}(\uinit)$, we have
\[
\frac{\eps^2}{2} \|\nabla_h u^n\|^2 + s^n = \hE_h(u^n,s^n) \le \hE_h(u^{n-1},s^{n-1})
\le \cdots \le \hE_h(u^0,s^0)= E_h(\uinit).
\]
Dropping off the nonnegative term leads to the expected result.
\end{proof}

\begin{theorem}[MBP of sESAV1]
\label{thm_sav1_mbp}
If $\kp\ge \|f'\|_{C[-\beta,\beta]}$,
the sESAV1 scheme \eqref{eq_sav1stab} preserves the MBP for $\{u^n\}$, i.e.,
the discrete version of \eqref{mbp} is valid as follows:
\begin{equation}
\label{dismbp}
\|\uinit\|_\infty \le \beta \quad \Longrightarrow \quad
\|u^n\|_\infty \le \beta, \quad \forall \, n.
\end{equation}
\end{theorem}

\begin{proof}
Suppose $(u^n,s^n)$ is given and $\|u^n\|_\infty\le\beta$ for some $n$.
From \eqref{lem_sav1_pf1}, we have
\[
u^{n+1} = \Big[\Big(\frac{1}{\dt}+\kp  g(u^n,s^n)\Big)I-\eps^2\Delta_h\Big]^{-1}
\Big[\frac{1}{\dt}u^n +  g(u^n,s^n)(f(u^n) + \kp u^n) \Big].
\]
Since $  g(u^n,s^n)>0$, by Lemma \ref{lem_lapdiff}, we have
\[
\Big\|\Big[\Big(\frac{1}{\dt}+\kp  g(u^n,s^n)\Big)I-\eps^2\Delta_h\Big]^{-1}\Big\|_\infty
\le \Big(\frac{1}{\dt}+\kp  g(u^n,s^n)\Big)^{-1}.
\]
Since $\kp\ge \|f'\|_{C[-\beta,\beta]}$ and $\|u^n\|_\infty\le\beta$,
according to Lemma \ref{lem_nonlinear}, it holds
\begin{equation}
\label{thm_sav1_mbp_pf}
\Big\|\frac{1}{\dt}u^n +  g(u^n,s^n)(f(u^n) + \kp u^n)\Big\|_\infty \le \Big(\frac{1}{\dt}+\kp  g(u^n,s^n)\Big) \beta.
\end{equation}
Therefore, we obtain
\[
\|u^{n+1}\|_\infty
\le \Big(\frac{1}{\dt}+\kp  g(u^n,s^n)\Big)^{-1}  \Big(\frac{1}{\dt}+\kp  g(u^n,s^n)\Big) \beta = \beta.
\]
By induction, we have $\|u^n\|_\infty\le\beta$ for all $n$.
\end{proof}

\begin{remark}
\label{rmk_sav1_mbp_pf}
The inequality \eqref{thm_sav1_mbp_pf} is valid if  $(\dt  g(u^n,s^n))^{-1}+\kp\ge \|f'\|_{C[-\beta,\beta]}$.
In other words, when $\kp=0$ (no stabilization), the MBP still holds for the sESAV1 scheme if the time step size  satisfies
$\dt \le (  g(u^n,s^n)\|f'\|_{C[-\beta,\beta]})^{-1}$ for all $n$.
\end{remark}

\begin{remark}
\label{rmk_sav1_mbp}
For the sESAV1 scheme \eqref{eq_sav1stab},
we know that the extra term $-\kp  g(u^n,s^n)(u^{n+1}-u^n)$ stabilizes the time stepping
and $\kp  g(u^n,s^n)$ is indeed the stabilizing constant, which is an $n$-dependent quantity.
In the proof of Theorem \ref{thm_sav1_mbp},
the key ingredients to preserve the MBP for $\{u^n\}$ involve two aspects:
the positivity of $\kp  g(u^n,s^n)$ and the relation of $u^{n+1}$ and $u^n$.
The former implies that the extra term is really a good stabilization term
and the latter guarantees the balance between the linear and nonlinear parts
so that the stabilized linear operator is sufficient to dominate the nonlinear term in order to preserve the MBP.
\end{remark}

\begin{remark}
\label{rmk_sav1_nonmbp}
For the classic SAV1 scheme \eqref{eq_sav1shen},
the stabilization term in \eqref{eq_sav1shena} actually takes the form
\[
- \kp u^{n+1} + \frac{r^{n+1}}{\sqrt{E_{2h}(u^n)+\delta}}\kp u^n.
\]
The sign of $r^{n+1}$, and thus the sign of $\frac{r^{n+1}}{\sqrt{E_{2h}(u^n)+\delta}}\kp$, is uncertain,
which violates the positivity of the stabilizing constant.
Even though $r^{n+1}$ may be positive in practical computations in some specific cases,
such a stabilization term leads to an imbalance between the linear and nonlinear parts
since $r^{n+1}\not=\sqrt{E_{2h}(u^n)+\delta}$ for $n\ge0$ in general,
so the scheme \eqref{eq_sav1shen} cannot preserve the MBP theoretically,
which will be also observed later in our numerical experiments.
Similarly, the stabilization term in the ESAV scheme \eqref{eq_esav1a} reads as
\[
- \kp u^{n+1} + \frac{r^n}{\expp{E_{2h}(u^n)}}\kp u^n,
\]
and the imbalance also exists between the linear and nonlinear parts
since $r^n\not=\expp{E_{2h}(u^n)}$ for $n\ge1$ in general,
and thus the ESAV scheme \eqref{eq_esav1} also does not preserve the MBP theoretically.
Nevertheless, $r^n>0$ always holds due to the definition of the auxiliary variable,
and this is the reason why we consider the ESAV approach rather than the classic one in this work.
\end{remark}

Note that the coefficient $  g(u^n,s^n)$ may vary step-by-step,
which is different from the continuous case that $g(u,s)\equiv 1$ exactly.
Fortunately, the change of $  g(u^n,s^n)$ is controllable in the sense that
it can be bounded by some constants, which is illustrated in the following.

\begin{corollary}
\label{cor_g_upbound}
If $h\le h_0$, $\kp\ge \|f'\|_{C[-\beta,\beta]}$, and $\|\uinit\|_\infty\le\beta$,
then there exists a constant $G^*=G^*(\uinit,C_*)$ such that $0<  g(u^n,s^n)\le G^*$ for all $n$.
\end{corollary}

\begin{proof}
The positivity of $g(u^n,s^n)$ comes from its definition.
We know from Corollary \ref{cor_sav1_es} and Theorem \ref{thm_sav1_mbp}
that $  g(u^n,s^n)\le\expp{E_h(\uinit)+C_*}$ for all $n$.
The $h$-dependence of upper bound can be removed by \eqref{uinit_approx},
which completes the proof.
\end{proof}

Actually, it also holds that $  g(u^n,s^n)$ has a positive lower bound uniformly in $n$ for any fixed terminal time $T>0$.
To show it, we first prove an estimate on the discrete $H^2$ semi-norm of the numerical solution.

\begin{lemma}
\label{lem_sav1_h2bound}
Given a fixed  time $T>0$.
If $h\le h_0$, $\kp\ge \|f'\|_{C[-\beta,\beta]}$, and $\|\uinit\|_\infty\le\beta$,
there exists a constant $M>0$ depending on $C_*$, $|\Omega|$, $T$, $\uinit$, $\kp$, $\eps$, and $\|f\|_{C^1[-\beta,\beta]}$, such that
\[
\|\delta_t u^{n+1}\| + \|\Delta_h u^{n+1}\|\le M, \quad 0 \le n \le \lfloor T/\dt\rfloor-1.
\]
\end{lemma}

\begin{proof}
Taking the discrete inner product of \eqref{eq_sav1staba} with $2\dt\Delta_h^2u^{n+1}$,
we obtain
\begin{align*}
& (1+\kp  g(u^n,s^n)\dt) \<\Delta_h u^{n+1} - \Delta_h u^n, 2\Delta_hu^{n+1}\>
+ 2\eps^2\dt \|\nabla_h\Delta_h u^{n+1}\|^2 \\
& \qquad\quad = - 2  g(u^n,s^n) \dt\<\nabla_h f(u^n),\nabla_h\Delta_h u^{n+1}\>.
\end{align*}
Using the facts that
\begin{align*}
\<\Delta_h u^{n+1} - \Delta_h u^n, 2\Delta_hu^{n+1}\>
& = \|\Delta_h u^{n+1}\|^2 - \|\Delta_h u^n\|^2 + \|\Delta_h u^{n+1} - \Delta_h u^n\|^2, \\
- 2  g(u^n,s^n)\dt \<\nabla_h f(u^n),\nabla_h\Delta_h u^{n+1}\>
& \le \frac{(  g(u^n,s^n))^2}{2\eps^2}\dt \|\nabla_h f(u^n)\|^2 + 2\eps^2\dt\|\nabla_h\Delta_h u^{n+1}\|^2,
\end{align*}
we obtain
\begin{equation}
\label{111}
(1+\kp  g(u^n,s^n)\dt) (\|\Delta_h u^{n+1}\|^2 - \|\Delta_h u^n\|^2)
\le \frac{(  g(u^n,s^n))^2}{2\eps^2}\dt \|\nabla_h f(u^n)\|^2.
\end{equation}
By Theorem \ref{thm_sav1_mbp},
we have $\|u^n\|_\infty\le\beta$, and thus,
\begin{equation}
\label{222}
\|\nabla_h f(u^n)\| \le \|f'\|_{C[-\beta,\beta]} \|\nabla_h u^n\|
\le \|f'\|_{C[-\beta,\beta]} C_\Omega \|\Delta_h u^n\|,
\end{equation}
where  the second step comes from the discrete Poincar\'e's inequality
with $C_\Omega$ being a constant depending only on $|\Omega|$ (since
$\nabla_h u^n$ has a zero mean due to the periodic boundary condition).
Then, by Corollary \ref{cor_g_upbound}, \eqref{111} and \eqref{222}, we obtain
\begin{align}
\label{lem_sav1_h2bound_pf}
\|\Delta_h u^{n+1}\|^2 &\le (1+\kp  g(u^n,s^n)\dt) \|\Delta_h u^{n+1}\|^2\nn\\
&\le \Big[1+\Big(\kp G^*+\frac{(G^* \|f'\|_{C[-\beta,\beta]} C_\Omega)^2}{2\eps^2}\Big)\dt\Big]  \|\Delta_h u^n\|^2.
\end{align}
By recursion, we obtain
\begin{align*}
\|\Delta_h u^{n+1}\|^2 &\le \Big[1+\Big(\kp G^*+\frac{(G^* \|f'\|_{C[-\beta,\beta]} C_\Omega)^2}{2\eps^2}\Big)\dt\Big]^{n+1} \|\Delta_h u^0\|^2\\
&\le \e^{\big(\kp G^*+\frac{(G^* \|f'\|_{C[-\beta,\beta]} C_\Omega)^2}{2\eps^2}\big)T} \|\Delta_h \uinit\|^2.
\end{align*}
Then, using Corollary \ref{cor_g_upbound} again,
we derive from \eqref{eq_sav1staba} directly to get
\begin{align*}
\|\delta_t u^{n+1}\| &\le (1+\kp  g(u^n,s^n)\dt)\|\delta_t u^{n+1}\|\\
&\le \eps^2 \|\Delta_h u^{n+1}\| +  g(u^n,s^n)\|f(u^n)\|
\le \eps^2 \|\Delta_h u^{n+1}\| + G^* F_0 |\Omega|^{\frac{1}{2}},
\end{align*}
where $F_0:=\|f\|_{C[-\beta,\beta]}$.
This completes the proof.
\end{proof}

\begin{corollary}
\label{cor_sav1_g_bound}
Given a fixed time $T>0$.
If $h\le h_0$, $\kp\ge \|f'\|_{C[-\beta,\beta]}$, and $\|\uinit\|_\infty\le\beta$,
there exists a constant $G_*>0$
such that $  g(u^n,s^n)\ge G_*$ for $0\le n\le \lfloor T/\dt \rfloor$,
where $G_*$ depends on $C_*$, $|\Omega|$, $T$, $\uinit$, $\kp$, $\eps$, and $\|f\|_{C^1[-\beta,\beta]}$.
\end{corollary}

\begin{proof}
According to the definition  of $g(u^n,s^n)$ in \eqref{gg} and the MBP for $\{u^n\}$,
it suffices to show the existence of the lower bound of $\{s^n\}$.
Using Lemma \ref{lem_sav1_h2bound}, we have
\[
\<f(u^n), u^{n+1}-u^n\> \le \dt \|f(u^n)\| \|\delta_t u^{n+1}\| \le F_0 |\Omega|^{\frac{1}{2}} M \dt,
\]
where $M$ is the constant defined  in Lemma \ref{lem_sav1_h2bound}.
Then, from \eqref{lem_sav1_pf2}, we have
\[
s^{n+1} \ge s^n - G^* F_0 |\Omega|^{\frac{1}{2}} M \dt.
\]
By recursion, noting that $s^0=E_{1h}(\uinit)\ge-C_*$, we obtain
\begin{equation*}
s^{n} \ge s^0 - G^* F_0 |\Omega|^{\frac{1}{2}} M n\dt \ge - C_* - G^* F_0 |\Omega|^{\frac{1}{2}} MT,
\end{equation*}
which completes the proof.
\end{proof}

The combination of Corollaries \ref{cor_g_upbound} and \ref{cor_sav1_g_bound} implies that
$0<G_*\le g(u^n,s^n) \le G^*$ for any fixed terminal time $T>0$, which will play an important role in
error estimates of the sESAV1 scheme \eqref{eq_sav1stab} in the next subsection.

\subsubsection{Error estimates}

In the following error analysis, as well as that for the second-order scheme presented later,
we will use many generic constants, and for simplicity of notations,
we may denote the constants with the same dependence but different values by the same notation.

If the exact solution $u_e$ to \eqref{AllenCahn} is smooth sufficiently,
letting $s_e(t)=E_1(u_e(t))$, we have
\begin{subequations}
\label{sav1trun}
\begin{align}
\frac{u_e(t_{n+1})-u_e(t_n)}{\dt}
& = \eps^2 \Delta_h u_e(t_{n+1}) + g(u_e(t_n),s_e(t_n)) f(u_e(t_n)) \nn \\
& \quad - \kp g(u_e(t_n),s_e(t_n)) (u_e(t_{n+1})-u_e(t_n)) + R_{1u}^n, \label{sav1truna} \\
\frac{s_e(t_{n+1})-s_e(t_n)}{\dt}
& = - g(u_e(t_n),s_e(t_n)) \Big\<f(u_e(t_n)),\frac{u_e(t_{n+1})-u_e(t_n)}{\dt}\Big\>
+ R_{1s}^n, \label{sav1trunb}
\end{align}
\end{subequations}
where the truncation errors $R_{1u}^n$ and $R_{1s}^n$ satisfy
\begin{equation}
\label{sav1trunerr}
\|R_{1u}^n\|\le C_e(\dt+h^2), \qquad |R_{1s}^n|\le C_e(\dt+h^2)
\end{equation}
with $C_e>0$ depending only on $u_e$, $\kp$, $\eps$, and $\|f\|_{C^1[-\beta,\beta]}$.
Define the error functions as
\begin{equation}
\label{errorfuns}
e_u^n = u^n - u_e(t_n), \qquad e_s^n = s^n - s_e(t_n).
\end{equation}
We first show a lemma on the error estimate for the nonlinear term.

\begin{lemma}
\label{lem_sav1_nonlinear}
If $h\le h_0$ and $\|u^n\|_\infty\le\beta$, we have
\begin{equation}
\label{lem_sav1_nonlinear1}
| g(u^n,s^n) - g(u_e(t_n),s_e(t_n)) | \le  C_{g} (\|e_u^n\| + |e_s^n|),
\end{equation}
and
\begin{equation}
\label{lem_sav1_nonlinear2}
\| g(u^n,s^n) f(u^n) - g(u_e(t_n),s_e(t_n)) f(u_e(t_n)) \| \le  C_{g} (\|e_u^n\| + |e_s^n|),
\end{equation}
where the constant $ C_{g} >0$ depends on $C_*$, $|\Omega|$, $\uinit$, and $\|f\|_{C^1[-\beta,\beta]}$.
\end{lemma}

\begin{proof}
For the exact solutions $u_e(t_n)$ and $s_e(t_n)$,
we have $\|u_e(t_n)\|_\infty\le\beta$ by the MBP and $s_e(t_n)\le E(\uinit)$ by the energy dissipation law.
Some careful calculations yield
\begin{align*}
| g(u^n,s^n) - g(u^n,s_e(t_n)) |
& = \frac{1}{\exp \{E_{1h}(u^n)\}} |\exp \{s^n\} - \exp \{s_e(t_n)\}| \\
& \le \frac{\exp \{\xi^n\}}{\exp \{E_{1h}(u^n)\}} |s^n - s_e(t_n)| \\
& \le G^* |s^n - s_e(t_n)|
\end{align*}
with $\xi^n$ being a number between $s^n$ and $s_e(t_n)$, and
\begin{align*}
& | g(u^n,s_e(t_n)) - g(u_e(t_n),s_e(t_n)) | \\
& \qquad\quad = \exp \{s_e(t_n)\} \Big|\frac{1}{\exp\{ E_{1h}(u^n)\}} - \frac{1}{\exp \{E_{1h}(u_e(t_n))\}} \Big| \nn \\
& \qquad\quad \le (\expp{E(\uinit)+C_*}) |E_{1h}(u^n) - E_{1h}(u_e(t_n))| \nn \\
& \qquad\quad \le |\Omega|^{\frac{1}{2}} (\expp{E(\uinit)+C_*}) \|F(u^n) - F(u_e(t_n))\| \nn \\
& \qquad\quad \le F_0 |\Omega|^{\frac{1}{2}} (\expp{E(\uinit)+C_*}) \|u^n-u_e(t_n)\|.
\end{align*}
By combining both of the above inqualities, we obtain  \eqref{lem_sav1_nonlinear1}.
In addition, we have
\begin{align*}
& \| g(u^n,s^n) f(u_e(t_n)) - g(u_e(t_n),s_e(t_n)) f(u_e(t_n))\| \\
& \qquad\quad \le \|f(u_e(t_n))\| |g(u^n,s^n) - g(u_e(t_n),s_e(t_n))| \nn \\
& \qquad\quad \le F_0 |\Omega|^{\frac{1}{2}} C (\|e_u^n\| + |e_s^n|).
\end{align*}
According to Corollary \ref{cor_g_upbound}, it holds
\begin{align*}
\| g(u^n,s^n) f(u^n) - g(u^n,s^n) f(u_e(t_n)) \|
& \le G^* \|f(u^n) - f(u_e(t_n))\| \\
& \le G^* \|f'\|_{C[-\beta,\beta]} \|u^n - u_e(t_n)\|.
\end{align*}
Then, we obtain \eqref{lem_sav1_nonlinear2} with the help of the triangular inequality to the above two inequalities.
\end{proof}

\begin{theorem}[Error estimate of sESAV1]
\label{thm_sav1_error}
Given a fixed time $T>0$ and suppose the exact solution $u_e$ is smooth enough on $[0,T]\times\overline{\Omega}$.
Assume that $\kp\ge \|f'\|_{C[-\beta,\beta]}$ and $\|\uinit\|_\infty\le\beta$.
If $\dt$ and $h$ are  small sufficiently, then we have the error estimate
for the sESAV1 scheme \eqref{eq_sav1stab} as follows:
\begin{equation*}
\|e_u^n\| + \|\nabla_h e_u^n\| + |e_s^n| \le C(\dt+h^2),\qquad 0 \le n \le \lfloor T/\dt\rfloor,
\end{equation*}
where the constant $C>0$ depends on $C_*$, $|\Omega|$, $T$, $u_e$, $\kp$, $\eps$, and $\|f\|_{C^1[-\beta,\beta]}$
but is independent of $\dt$ and $h$.
\end{theorem}

\begin{proof}
The difference between \eqref{eq_sav1stab} and \eqref{sav1trun} leads to
\begin{subequations}
\label{sav1err}
\begin{align}
\delta_te_u^{n+1}
& = \eps^2 \Delta_h e_u^{n+1} + g(u^n,s^n) f(u^n) 
- g(u_e(t_n),s_e(t_n)) f(u_e(t_n)) - \kp g(u^n,s^n) (e_u^{n+1}-e_u^n) \nn \\
& \quad + \kp (g(u_e(t_n),s_e(t_n)) - g(u^n,s^n)) (u_e(t_{n+1})-u_e(t_n)) - R_{1u}^n, \label{sav1erra} \\
\delta_te_s^{n+1}
& = \Big\< g(u_e(t_n),s_e(t_n)) f(u_e(t_n)) - g(u^n,s^n) f(u^n), \frac{u_e(t_{n+1})-u_e(t_n)}{\dt} \Big\> \nn \\
& \quad - g(u^n,s^n) \<f(u^n),\delta_te_u^{n+1}\> - R_{1s}^n. \label{sav1errb}
\end{align}
\end{subequations}

Taking the discrete inner product of \eqref{sav1erra} with $2\dt\delta_te_u^{n+1}$
and rearranging the terms give us
\begin{align*}
&
2\dt\|\delta_te_u^{n+1}\|^2 + 2\kp g(u^n,s^n) \|e_u^{n+1}-e_u^n\|^2 - 2\eps^2 \<\Delta_h e_u^{n+1},e_u^{n+1}-e_u^n\> \\
& \qquad = 2\dt \<g(u^n,s^n) f(u^n) - g(u_e(t_n),s_e(t_n)) f(u_e(t_n)), \delta_te_u^{n+1}\> \nn \\
& \qquad\quad + 2 \kp\dt (g(u_e(t_n),s_e(t_n)) - g(u^n,s^n)) \<u_e(t_{n+1})-u_e(t_n), \delta_te_u^{n+1}\>
- 2\dt\<R_{1u}^n,\delta_te_u^{n+1}\>.
\end{align*}
Since $g(u^n,s^n)\ge G_*>0$ by Corollary \ref{cor_sav1_g_bound},
using the identities
\begin{align}
\<\Delta_h e_u^{n+1},e_u^{n+1}-e_u^n\>
& = -\frac{1}{2}\|\nabla_h e_u^{n+1}\|^2 + \frac{1}{2}\|\nabla_h e_u^n\|^2
- \frac{1}{2}\dt^2\|\nabla_h \delta_te_u^{n+1}\|^2, \nn \\ 
\|e_u^{n+1}-e_u^n\|^2 & = \|e_u^{n+1}\|^2 - \|e_u^n\|^2 - 2\dt \<e_u^n, \delta_te_u^{n+1}\>, \label{sav1err_pf0}
\end{align}
we obtain
\begin{align}
&  2G_*\kp\|e_u^{n+1}\|^2 - 2G_*\kp\|e_u^n\|^2
+ \eps^2 \|\nabla_h e_u^{n+1}\|^2 - \eps^2 \|\nabla_h e_u^n\|^2
+ 2\dt \|\delta_te_u^{n+1}\|^2 \nn \\
& \qquad\quad \le 2\dt \<g(u^n,s^n) f(u^n) - g(u_e(t_n),s_e(t_n)) f(u_e(t_n)), \delta_te_u^{n+1}\> \nn \\
& \qquad\qquad + 2\kp\dt (g(u_e(t_n),s_e(t_n)) - g(u^n,s^n)) \<u_e(t_{n+1})-u_e(t_n), \delta_te_u^{n+1}\> \nn \\
& \qquad\qquad + 4G_*\kp\dt\<e_u^n, \delta_te_u^{n+1}\> - 2\dt \<R_{1u}^n, \delta_te_u^{n+1}\>. \label{sav1err_pf3}
\end{align}
For the first term in the right-hand side of \eqref{sav1err_pf3}, by Lemma \ref{lem_sav1_nonlinear} we have
\begin{align}
& 2\dt \<g(u^n,s^n) f(u^n) - g(u_e(t_n),s_e(t_n)) f(u_e(t_n)), \delta_te_u^{n+1}\> \nn \\
& \qquad\quad \le 2 C_{g} \dt (\|e_u^n\| + |e_s^n|) \|\delta_te_u^{n+1}\| 
\le 4 C_{g} ^2 \dt (\|e_u^n\|^2 + |e_s^n|^2)  + \frac{\dt}{2} \|\delta_te_u^{n+1}\|^2, \label{sav1err_pf4a}
\end{align}
where $ C_{g} >0$ is the constant in Lemma \ref{lem_sav1_nonlinear}.
For the second term in the right-hand side of \eqref{sav1err_pf3}, we have
\begin{align}
& 2\kp\dt (g(u_e(t_n),s_e(t_n)) - g(u^n,s^n)) \<u_e(t_{n+1})-u_e(t_n), \delta_te_u^{n+1}\> \nn \\
& \qquad\quad \le 2\kp\dt |g(u_e(t_n),s_e(t_n)) - g(u^n,s^n)| \|u_e(t_{n+1})-u_e(t_n)\| \|\delta_te_u^{n+1}\| \nn \\
& \qquad\quad \le 2 C_{g} \kp\dt (\|u_e(t_{n+1})\|+\|u_e(t_n)\|) (\|e_u^n\| + |e_s^n|) \|\delta_te_u^{n+1}\|\nn \\
& \qquad\quad \le C_1\kp^2\dt (\|e_u^n\|^2 + |e_s^n|^2) + \frac{\dt}{2} \|\delta_te_u^{n+1}\|^2, \label{sav1err_pf4b0}
\end{align}
where $C_1>0$ depends on $C_*$, $|\Omega|$, $u_e$, and $\|f\|_{C^1[-\beta,\beta]}$.
Using the Young's inequality,
the third and fourth terms in the right-hand side of \eqref{sav1err_pf3} can be bounded respectively as
\begin{eqnarray}
& 4G_*\kp\dt\<e_u^n, \delta_te_u^{n+1}\>
\le 4G_*\kp\dt \|e_u^n\| \|\delta_te_u^{n+1}\|
\le 8G_*^2\kp^2\dt \|e_u^n\|^2 + \dfrac{\dt}{2}\|\delta_te_u^{n+1}\|^2, \qquad \label{sav1err_pf4b} \\
& -2\dt \<R_{1u}^n, \delta_te_u^{n+1}\> \le 2\dt \|R_{1u}^n\| \|\delta_te_u^{n+1}\|
\le 4\dt \|R_{1u}^n\|^2 + \dfrac{\dt}{4}\|\delta_te_u^{n+1}\|^2. \quad \label{sav1err_pf4c}
\end{eqnarray}
Then, substituting \eqref{sav1err_pf4a}--\eqref{sav1err_pf4c} into \eqref{sav1err_pf3} leads to
\begin{align}
& 2G_*\kp\|e_u^{n+1}\|^2 - 2G_*\kp\|e_u^n\|^2
+ \eps^2 \|\nabla_h e_u^{n+1}\|^2 - \eps^2 \|\nabla_h e_u^n\|^2
+ \frac{\dt}{4} \|\delta_te_u^{n+1}\|^2 \nn \\
& \qquad \le (4 C_{g} ^2+C_1\kp^2+8G_*^2\kp^2) \dt \|e_u^n\|^2
+ (4 C_{g} ^2+C_1\kp^2) \dt|e_s^n|^2 + 4\dt \|R_{1u}^n\|^2. \label{sav1err_pf5}
\end{align}

Multiplying \eqref{sav1errb} by $2\dt e_s^{n+1}$ yields
\begin{align}
& |e_s^{n+1}|^2 - |e_s^n|^2 + |e_s^{n+1}-e_s^n|^2 \nn \\
& \qquad\quad = 2e_s^{n+1} \< g(u_e(t_n),s_e(t_n)) f(u_e(t_n)) - g(u^n,s^n) f(u^n), u_e(t_{n+1})-u_e(t_n) \> \nn \\
& \qquad\qquad - 2\dt e_s^{n+1} g(u^n,s^n) \<f(u^n),\delta_te_u^{n+1}\> - 2\dt R_{1s}^n e_s^{n+1}. \label{sav1err_pf6}
\end{align}
For the first term in the right-hand side of \eqref{sav1err_pf6}, by Lemma \ref{lem_sav1_nonlinear} we have
\begin{align}
& 2e_s^{n+1} \< g(u_e(t_n),s_e(t_n)) f(u_e(t_n)) - g(u^n,s^n) f(u^n), u_e(t_{n+1})-u_e(t_n) \> \nn \\
& \qquad\quad \le 2|e_s^{n+1}| \|g(u_e(t_n),s_e(t_n)) f(u_e(t_n)) - g(u^n,s^n) f(u^n)\|  \|u_e(t_{n+1})-u_e(t_n)\| \nn \\
& \qquad\quad \le 2 C_{g} \dt |e_s^{n+1}|  (\|e_u^n\| + |e_s^n|) \|(u_e)_t(\theta_n)\| \qquad (\mbox{for some } t_n<\theta_n<t_{n+1}) \nn \\
& \qquad\quad \le C_2 \dt (\|e_u^n\|^2 + |e_s^n|^2 + |e_s^{n+1}|^2), \label{sav1err_pf7a}
\end{align}
where $C_2>0$ depends on $C_*$, $|\Omega|$, $u_e$, and $\|f\|_{C^1[-\beta,\beta]}$.
For the second term in the right-hand side of \eqref{sav1err_pf6},
using Corollary \ref{cor_g_upbound}, we obtain
\begin{align}
\label{sav1err_pf7b}
- 2\dt e_s^{n+1} g(u^n,s^n) \<f(u^n),\delta_te_u^{n+1}\>
& \le 2G^*\dt \|f(u^n)\| |e_s^{n+1}| \|\delta_te_u^{n+1}\| \nn \\
& \le C_3\dt |e_s^{n+1}|^2 + \frac{\dt}{4}\|\delta_te_u^{n+1}\|^2,
\end{align}
where $C_3>0$ depends on $C_*$, $|\Omega|$, $\uinit$, and $\|f\|_{C[-\beta,\beta]}$.
For the third term in the right-hand side of \eqref{sav1err_pf6}, we have
\begin{equation}
\label{sav1err_pf7c}
- 2\dt R_{1s}^n e_s^{n+1} \le \dt |R_{1s}^n|^2 + \dt |e_s^{n+1}|^2.
\end{equation}
Then, substituting \eqref{sav1err_pf7a}--\eqref{sav1err_pf7c} into \eqref{sav1err_pf6} leads to
\begin{equation}
\label{sav1err_pf8}
|e_s^{n+1}|^2 - |e_s^n|^2
\le C_2 \dt \|e_u^n\|^2 + C_2 \dt |e_s^n|^2
+ (1 + C_2 + C_3) \dt |e_s^{n+1}|^2 + \frac{\dt}{4}\|\delta_te_u^{n+1}\|^2 + \dt |R_{1s}^n|^2.
\end{equation}

Adding \eqref{sav1err_pf5} and \eqref{sav1err_pf8}, we obtain
\begin{align*}
& 2G_*\kp (\|e_u^{n+1}\|^2 - \|e_u^n\|^2)
+ \eps^2 (\|\nabla_h e_u^{n+1}\|^2 - \|\nabla_h e_u^n\|^2) + (|e_s^{n+1}|^2 - |e_s^n|^2) \\
& \qquad\quad \le (4 C_{g} ^2+C_1\kp^2+8G_*^2\kp^2+C_2) \dt \|e_u^n\|^2 + (4 C_{g} ^2+C_1\kp^2+C_2) \dt|e_s^n|^2 \\
& \qquad\qquad + (1 + C_2 + C_3) \dt |e_s^{n+1}|^2 + 4\dt \|R_{1u}^n\|^2 + \dt |R_{1s}^n|^2.
\end{align*}
Then, using \eqref{sav1trunerr},
we reach
\begin{align*}
& \quad~ 2G_*\kp (\|e_u^{n+1}\|^2 - \|e_u^n\|^2)
+ \eps^2 (\|\nabla_h e_u^{n+1}\|^2 - \|\nabla_h e_u^n\|^2)
+ (|e_s^{n+1}|^2 - |e_s^n|^2) \\
&\qquad\qquad \le C_4 \dt (\|e_u^n\|^2 + |e_s^n|^2 + |e_s^{n+1}|^2) + 5C_e^2\dt(\dt+h^2)^2,
\end{align*}
where the constant $C_4$ depends on $C_*$, $|\Omega|$, $T$, $u_e$, $\kp$, $\eps$, and $\|f\|_{C^1[-\beta,\beta]}$.

Letting $W^n := 2G_*\kp \|e_u^n\|^2 + \eps^2 \|\nabla_h e_u^n\|^2 + |e_s^n|^2$,
we have
\[
W^{n+1} - W^n \le  \widetilde{C}_4\dt (W^n + W^{n+1}) + 5C_e^2\dt(\dt+h^2)^2,
\]
where $ \widetilde{C}_4$ depends on $C_4$ and $\kp$.
When $\dt\le\frac{1}{2 \widetilde{C}_4}$,
noting that $\frac{1+ \widetilde{C}_4\dt}{1- \widetilde{C}_4\dt}\le 1 + 4 \widetilde{C}_4\dt$, we obtain
\[
W^{n+1} \le (1+4 \widetilde{C}_4\dt) W^n + 10C_e^2\dt (\dt+h^2)^2.
\]
Using the discrete Gronwall's inequality, we obtain
\[
2G_*\kp\|e_u^n\|^2 + \eps^2 \|\nabla_h e_u^n\|^2 + |e_s^n|^2
= W^n \le 10C_e^2 \e^{4 \widetilde{C}_4T} (\dt+h^2)^2,
\]
which completes the proof.
\end{proof}

\begin{remark}
For any fixed $h>0$, let us recall the space-discrete problem \eqref{eq_semidis}
and denote by $u_{h,e}(t)$ the exact solution.
By similar analysis as Theorem \ref{thm_sav1_error},
one can obtain the error estimates for sufficiently small $\dt$ as follows:
\begin{equation*}
\|u^n-u_{h,e}(t_n)\| + \|\nabla_h u^n - \nabla_h u_{h,e}(t_n)\| + |s^n-E_{1h}(u_{h,e}(t))| \le C_h\dt,
\end{equation*}
where the constant  $C_h>0$ depends on $C_*$, $|\Omega|$, $T$, $u_{h,e}$, $\kp$, $\eps$, and $\|f\|_{C^1[-\beta,\beta]}$
but is independent of $\dt$.
\end{remark}

\subsection{Second-order sESAV scheme}
\label{sect_sav2}

For the space-discrete system \eqref{eq_semidis},
the second-order stabilized ESAV scheme (sESAV2) is given by
\begin{subequations}
\label{eq_sav2stab}
\begin{align}
\delta_t u^{n+1}
& = \eps^2 \Delta_h u^{n+\frac{1}{2}} +  g(\widehat{u}^{n+\frac{1}{2}},\widehat{s}^{n+\frac{1}{2}}) f(\widehat{u}^{n+\frac{1}{2}})
- \kp  g(\widehat{u}^{n+\frac{1}{2}},\widehat{s}^{n+\frac{1}{2}}) (u^{n+\frac{1}{2}}-\widehat{u}^{n+\frac{1}{2}}), \label{eq_sav2staba} \\
\delta_t s^{n+1}
& = -  g(\widehat{u}^{n+\frac{1}{2}},\widehat{s}^{n+\frac{1}{2}}) \<f(\widehat{u}^{n+\frac{1}{2}}),\delta_t u^{n+1}\>
+ \kp  g(\widehat{u}^{n+\frac{1}{2}},\widehat{s}^{n+\frac{1}{2}}) \<u^{n+\frac{1}{2}}-\widehat{u}^{n+\frac{1}{2}}, \delta_t u^{n+1}\>, \label{eq_sav2stabb}
\end{align}
\end{subequations}
where $  g(\widehat{u}^{n+\frac{1}{2}},\widehat{s}^{n+\frac{1}{2}}) > 0$
with $(\widehat{u}^{n+\frac{1}{2}},\widehat{s}^{n+\frac{1}{2}})$ being generated by
the first-order scheme \eqref{eq_sav1stab} with the time step size $\dt/2$, i.e.,
\begin{subequations}
\label{eq_sav2stab0}
\begin{align}
\frac{\widehat{u}^{n+\frac{1}{2}}-u^n}{\dt/2}
& = \eps^2 \Delta_h \widehat{u}^{n+\frac{1}{2}} +  g(u^n,s^n)f(u^n) - \kp  g(u^n,s^n)(\widehat{u}^{n+\frac{1}{2}}-u^n),
\label{eq_sav2stab0a}\\
\widehat{s}^{n+\frac{1}{2}}-s^n & = -  g(u^n,s^n)\<f(u^n),\widehat{u}^{n+\frac{1}{2}}-u^n\>. \label{eq_sav2stab0b}
\end{align}
\end{subequations}
The scheme \eqref{eq_sav2stab} is started by $u^0=\uinit$ and $s^0=E_{1h}(u^0)$.
By the definition of $u^{n+\frac{1}{2}}$,
the last term in \eqref{eq_sav2staba} is actually
$- \frac{1}{2} \kp  g(\widehat{u}^{n+\frac{1}{2}},\widehat{s}^{n+\frac{1}{2}}) (u^{n+1}-2\widehat{u}^{n+\frac{1}{2}}+u^n)$,
which provides a second-order truncation error in time.
We can rewrite \eqref{eq_sav2stab} in the following form:
\begin{subequations}
\label{eq_sav2var}
\begin{align}
& \Big[\Big(\frac{2}{\dt}+\kp  g(\widehat{u}^{n+\frac{1}{2}},\widehat{s}^{n+\frac{1}{2}})\Big)I-\eps^2\Delta_h\Big] u^{n+1}
 \nn\\
&\qquad = \Big[\Big(\frac{2}{\dt}-\kp  g(\widehat{u}^{n+\frac{1}{2}},\widehat{s}^{n+\frac{1}{2}})\Big)I+\eps^2\Delta_h\Big] u^n + 2g(\widehat{u}^{n+\frac{1}{2}},\widehat{s}^{n+\frac{1}{2}})[f(\widehat{u}^{n+\frac{1}{2}}) + \kp\widehat{u}^{n+\frac{1}{2}}], \label{lem_sav2_pf} \\
&s^{n+1}  = s^n -  g(\widehat{u}^{n+\frac{1}{2}},\widehat{s}^{n+\frac{1}{2}})
\<f(\widehat{u}^{n+\frac{1}{2}}) - \kp (u^{n+\frac{1}{2}}-\widehat{u}^{n+\frac{1}{2}}),u^{n+1}-u^n\>.
\end{align}
\end{subequations}
It is then easy to see that the system \eqref{eq_sav2var} is linear and uniquely solvable for any $\dt>0$
since $(\frac{2}{\dt}+\kp  g(\widehat{u}^{n+\frac{1}{2}},\widehat{s}^{n+\frac{1}{2}}))I-\eps^2\Delta_h$ is self-adjoint and positive definite.

\subsubsection{Energy dissipation and MBP}

The energy dissipation law and the MBP preservation of the  sESAV2  scheme \eqref{eq_sav2stab}  are stated below.

\begin{theorem}[Energy dissipation of sESAV2]
\label{thm_sav2_es}
For any $\kp\ge0$ and $\dt>0$, the sESAV2 scheme \eqref{eq_sav2stab} is energy dissipative
in the sense that $\hE_h(u^{n+1},s^{n+1})\le \hE_h(u^n,s^n)$,
where $\hE_h(u^n,s^n)$ is given by \eqref{modified_energy}.
Moreover, it holds that  $s^n\le E_h(\uinit)$ and $\widehat{s}^{n+\frac{1}{2}}\le E_h(\uinit)$ for all $n$.
\end{theorem}

\begin{proof}
Taking the inner product of \eqref{eq_sav2stab} with $u^{n+1}-u^n$ yields
\begin{align}
\label{sav2_es_pf1}
\frac{1}{\dt}\|u^{n+1}-u^n\|^2
& = -\frac{\eps^2}{2}\|\nabla_hu^{n+1}\|^2 + \frac{\eps^2}{2}\|\nabla_hu^n\|^2 
+  g(\widehat{u}^{n+\frac{1}{2}},\widehat{s}^{n+\frac{1}{2}}) \<f(\widehat{u}^{n+\frac{1}{2}}),u^{n+1}-u^n\>\nn\\
&\quad - \kp  g(\widehat{u}^{n+\frac{1}{2}},\widehat{s}^{n+\frac{1}{2}}) \<u^{n+\frac{1}{2}}-\widehat{u}^{n+\frac{1}{2}},u^{n+1}-u^n\>.
\end{align}
Combining \eqref{sav2_es_pf1} and \eqref{eq_sav2stabb},
we obtain
\[
\hE_h(u^{n+1},s^{n+1}) - \hE_h(u^n,s^n) = -\frac{1}{\dt} \|u^{n+1}-u^n\|^2 \le 0.
\]

Similar to the proof of Corollary \ref{cor_sav1_es},
the uniform boundedness of $\{s^n\}$ is a direct result of the energy stability.
Since $\widehat{s}^{n+\frac{1}{2}}$ is generated by the sESAV1 scheme \eqref{eq_sav2stab0},
according to Theorem \ref{thm_sav1_es}, we have
$\widehat{s}^{n+\frac{1}{2}}
\le \hE_h(\widehat{u}^{n+\frac{1}{2}},\widehat{s}^{n+\frac{1}{2}})
\le \hE_h(u^n,s^n)$,
and thus, we have $\widehat{s}^{n+\frac{1}{2}}\le E_h(\uinit)$.
\end{proof}


\begin{theorem}[MBP of sESAV2]
\label{thm_sav2_mbp}
If $h\le h_0$, $\kp\ge \|f'\|_{C[-\beta,\beta]}$, and
\begin{equation}
\label{sav2_mbp_dt}
\dt \le \Big( \frac{\kp G^*}{2} + \frac{\eps^2}{h^2} \Big)^{-1},
\end{equation}
where $G^*$ is the positive constant defined in Corollary \ref{cor_g_upbound}, then the sESAV2 scheme \eqref{eq_sav2stab} preserves the MBP for $\{u^n\}$,
i.e., \eqref{dismbp} is valid.
\end{theorem}

\begin{proof}
Suppose $(u^n,s^n)$ is given and $\|u^n\|_\infty\le\beta$ for some $n$.
By Theorems \ref{thm_sav1_mbp} and \ref{thm_sav2_es},
we have $\|\widehat{u}^{n+\frac{1}{2}}\|_\infty\le\beta$
and $\widehat{s}^{n+\frac{1}{2}}\le E_h(\uinit)$.
Then, we know that $0< g(\widehat{u}^{n+\frac{1}{2}},\widehat{s}^{n+\frac{1}{2}})\le G^*$
by the similar analysis as Corollary \ref{cor_g_upbound}.
The condition \eqref{sav2_mbp_dt} implies
\begin{equation*}
\frac{2}{\dt}-\kp g(\widehat{u}^{n+\frac{1}{2}},\widehat{s}^{n+\frac{1}{2}})\ge\frac{2\eps^2}{h^2}.
\end{equation*}
According to the definition of the matrix $\infty$-norm,
we have
\[
\Big\|\Big(\frac{2}{\dt}-\kp g(\widehat{u}^{n+\frac{1}{2}},\widehat{s}^{n+\frac{1}{2}})\Big)I+\eps^2\Delta_h\Big\|_\infty
= \frac{2}{\dt}-\kp g(\widehat{u}^{n+\frac{1}{2}},\widehat{s}^{n+\frac{1}{2}}).
\]
Since $\kp\ge\|f'\|_{C[-\beta,\beta]}$ and $\|\widehat{u}^{n+\frac{1}{2}}\|_\infty\le\beta$,
according to Lemma \ref{lem_nonlinear}, we have
\[
\|f(\widehat{u}^{n+\frac{1}{2}}) + \kp\widehat{u}^{n+\frac{1}{2}}\|_\infty \le \kp\beta.
\]
Therefore, using Lemma \ref{lem_lapdiff}, we obtain from \eqref{lem_sav2_pf} that
\[
\|u^{n+1}\|_\infty
\le \Big(\frac{2}{\dt}+\kp g(\widehat{u}^{n+\frac{1}{2}},\widehat{s}^{n+\frac{1}{2}})\Big)^{-1}
\Big[\Big(\frac{2}{\dt}-\kp g(\widehat{u}^{n+\frac{1}{2}},\widehat{s}^{n+\frac{1}{2}})\Big) \beta + 2\kp g(\widehat{u}^{n+\frac{1}{2}},\widehat{s}^{n+\frac{1}{2}})\beta \Big] = \beta.
\]
By induction, we have $\|u^n\|_\infty\le\beta$ for all $n$.
\end{proof}

\begin{remark}
\label{rmk_sav2_gbound}
Theorem \ref{thm_sav2_mbp} implies that $0<  g(u^n,s^n)\le G^*$ and $0< g(\widehat{u}^{n+\frac{1}{2}},\widehat{s}^{n+\frac{1}{2}})\le G^*$ hold for all $n$.
\end{remark}

\begin{remark}
\label{rmk_sav2_mbp}
The  condition \eqref{sav2_mbp_dt} on the time step size implies $\dt=\mathcal{O}(h^2/\eps^2)$,
which is the same as those enforced in \cite{HoLe20,HoTaYa17}.
This restriction comes essentially from the explicit term $\Delta_hu^n$
due to the use of the Crank--Nicolson approximation,
which also means that
the second-order scheme \eqref{eq_sav2stab} cannot preserve the MBP unconditionally
even though we introduce the stabilization term.
In practical computations, $ g(\widehat{u}^{n+\frac{1}{2}},\widehat{s}^{n+\frac{1}{2}})\approx1$
so that the requirement for the time step size can be set to be $\dt \le ( \frac{\kp}{2} + \frac{\eps^2}{h^2} )^{-1}$ in order to preserve the MBP,
which is later used in our numerical experiments.
\end{remark}

Similar to the analysis for the sESAV1 scheme,
we can show that both $\{  g(u^n,s^n)\}$ and $\{ g(\widehat{u}^{n+\frac{1}{2}},\widehat{s}^{n+\frac{1}{2}})\}$ have uniform positive lower bounds.

\begin{lemma}
\label{lem_sav2_h2bound}
Given a fixed time $T>0$.
If $h\le h_0$, $\kp\ge \|f'\|_{C[-\beta,\beta]}$, $\|\uinit\|_\infty\le\beta$,
and $\dt\le 1$ satisfying \eqref{sav2_mbp_dt},
there exists a constant $M>0$ depending on  $C_*$, $|\Omega|$, $T$, $\uinit$, $\kp$, $\eps$, and $\|f\|_{C^1[-\beta,\beta]}$ such that
\begin{align*}
\dt^{-1}\|\widehat{u}^{n+\frac{1}{2}}-u^n\| + \|\Delta_h\widehat{u}^{n+\frac{1}{2}}\| \le M, \\
\dt^{-1}\|u^{n+1}-u^n\| + \|\Delta_h u^{n+1}\|\le M,
\end{align*}
for $0 \le n \le \lfloor T/\dt\rfloor-1$.
\end{lemma}

\begin{proof}
Since $\widehat{u}^{n+\frac{1}{2}}$ is the solution to the sESAV1 substep \eqref{eq_sav2stab0},
according to \eqref{lem_sav1_h2bound_pf},
we have
\begin{equation}
\label{lem_sav2_h2bound_pf}
\|\Delta_h\widehat{u}^{n+\frac{1}{2}}\|^2
\le \Big(1+\frac{G^*\kp}{2}+\frac{(G^* \|f'\|_{C[-\beta,\beta]} C_\Omega)^2}{4\eps^2}\Big) \|\Delta_h u^n\|^2,
\end{equation}
where we used $\dt\le1$.

Taking the discrete inner product of \eqref{eq_sav2staba} with $2\dt\Delta_h^2u^{n+\frac{1}{2}}$,
using the fact $0< g(\widehat{u}^{n+\frac{1}{2}},\widehat{s}^{n+\frac{1}{2}})\le G^*$,
and conducting the similar analysis as the proof of Lemma \ref{lem_sav1_h2bound},
we can obtain
\begin{align*}
\|\Delta_h u^{n+1}\|^2
\le \|\Delta_h u^n\|^2
+ \Big(G^*\kp + \frac{(G^* \|f'\|_{C[-\beta,\beta]} C_\Omega)^2}{2\eps^2}\Big) \dt \|\Delta_h\widehat{u}^{n+\frac{1}{2}}\|^2.
\end{align*}
Substituting \eqref{lem_sav2_h2bound_pf} into the above inequality, we have
\begin{align*}
\|\Delta_h u^{n+1}\|^2
& \le \Big[1 + \Big(G^*\kp + \frac{(G^* \|f'\|_{C[-\beta,\beta]} C_\Omega)^2}{2\eps^2}\Big) \\
& \qquad \cdot \Big(1+\frac{G^*\kp}{2}+\frac{(C_\Omega \|f'\|_{C[-\beta,\beta]} G^*)^2}{4\eps^2}\Big) \dt \Big] \|\Delta_h u^n\|^2.
\end{align*}
By recursion, we can obtain a uniform upper bound for $\|\Delta_h u^{n+1}\|$.
Then, by \eqref{lem_sav2_h2bound_pf} we also can get the upper bound for $\|\Delta_h\widehat{u}^{n+\frac{1}{2}}\|$.

Finally, as a consequence of the above analysis,
using \eqref{eq_sav2stab0a} and \eqref{eq_sav2staba}, we also get
the boundedness of $\dt^{-1}\|\widehat{u}^{n+\frac{1}{2}}-u^n\|$ and $\dt^{-1}\|u^{n+1}-u^n\|$, and the proof is completed.
\end{proof}

\begin{corollary}
\label{cor_sav2_g_bound}
Given a fixed time $T>0$.
If $h\le h_0$, $\kp\ge \|f'\|_{C[-\beta,\beta]}$, $\|\uinit\|_\infty\le\beta$,
and $\dt\le 1$ satisfying \eqref{sav2_mbp_dt},
there exists a constant ${\widetilde G}_{*}>0$ such that
$  g(u^n,s^n)\ge {\widetilde G}_{*}$ and $ g(\widehat{u}^{n+\frac{1}{2}},\widehat{s}^{n+\frac{1}{2}})\ge {\widetilde G}_{*}$,
where ${\widetilde G}_{*}$ depends on $C_*$, $|\Omega|$, $T$, $\uinit$, $\kp$, $\eps$, and $\|f\|_{C^1[-\beta,\beta]}$.
\end{corollary}

\begin{proof}
It suffices to show the existence of the lower bounds of $\{s^n\}$ and $\{\widehat{s}^{n+\frac{1}{2}}\}$.
By Lemma \ref{lem_sav2_h2bound}, we have $\|u^{n+1}-u^n\|\le M\dt$.
From \eqref{eq_sav2stabb},
the similar analysis as Corollary \ref{cor_sav1_g_bound} leads to the lower boundedness of $\{s^n\}$.
Then, since $\|\widehat{u}^{n+\frac{1}{2}}-u^n\|\le M\dt\le M$,
we can derive from \eqref{eq_sav2stab0b} to give
\[
\widehat{s}^{n+\frac{1}{2}} \ge s^n - G^* F_0 |\Omega|^{\frac{1}{2}} M,
\]
which completes the proof.
\end{proof}

\subsubsection{Error estimates}

It is easy to check that the exact solution $u_e$ to \eqref{AllenCahn} with $s_e(t)=E_1(u_e(t))$ satisfies
\begin{subequations}
\label{sav2trun}
\begin{align}
&\frac{u_e(t_{n+1})-u_e(t_n)}{\dt}
 = \frac{\eps^2}{2} \Delta_h (u_e(t_{n+1})+u_e(t_n))
+ g(u_e(t_{n+\frac{1}{2}}),s_e(t_{n+\frac{1}{2}})) f(u_e(t_{n+\frac{1}{2}})) \nn \\
& \qquad - \kp g(u_e(t_{n+\frac{1}{2}}),s_e(t_{n+\frac{1}{2}})) \Big(\frac{u_e(t_{n+1})+u_e(t_n)}{2}-u_e(t_{n+\frac{1}{2}})\Big)
+ R_{2u}^n, \label{sav2truna} \\
&\frac{s_e(t_{n+1})-s_e(t_n)}{\dt}
 = - g(u_e(t_{n+\frac{1}{2}}),s_e(t_{n+\frac{1}{2}})) \Big\<f(u_e(t_{n+\frac{1}{2}})),\frac{u_e(t_{n+1})-u_e(t_n)}{\dt}\Big\> \nn \\
& \qquad + \kp g(u_e(t_{n+\frac{1}{2}}),s_e(t_{n+\frac{1}{2}}))
\Big\<\frac{u_e(t_{n+1})+u_e(t_n)}{2}-u_e(t_{n+\frac{1}{2}}),\frac{u_e(t_{n+1})-u_e(t_n)}{\dt}\Big\>
+ R_{2s}^n, \label{sav2trunb}
\end{align}
\end{subequations}
where the truncation errors $R_{2u}^n$ and $R_{2s}^n$ satisfy
\begin{equation}
\label{sav2trunerr}
\|R_{2u}^n\|\le C_e(\dt^2+h^2), \qquad |R_{2s}^n|\le C_e(\dt^2+h^2).
\end{equation}
Apart from the numerical error functions $e_u^n$ and $e_s^n$ defined by \eqref{errorfuns},
let us also define
\[
\widehat{e}_u^{n+\frac{1}{2}} = \widehat{u}^{n+\frac{1}{2}} - u_e(t_{n+\frac{1}{2}}), \qquad
\widehat{e}_s^{n+\frac{1}{2}} = \widehat{s}^{n+\frac{1}{2}} - s_e(t_{n+\frac{1}{2}}).
\]
We first present an estimate for $\widehat{e}_u^{n+\frac{1}{2}}$ and $\widehat{e}_s^{n+\frac{1}{2}}$,
which will be used in the proof of the error estimate for the sESAV2 scheme \eqref{eq_sav2stab}.
Recalling the proof of Theorem \ref{thm_sav1_error} for the sESAV1 scheme,
the error equations with respect to $\widehat{e}_u^{n+\frac{1}{2}}$ and $\widehat{e}_s^{n+\frac{1}{2}}$ read as
\begin{subequations}
\label{sav2err1}
\begin{align}
\frac{\widehat{e}_u^{n+\frac{1}{2}} - e_u^n}{\dt/2}
& = \eps^2 \Delta_h \widehat{e}_u^{n+\frac{1}{2}} + g(u^n,s^n)f(u^n) 
- g(u_e(t_n),s_e(t_n))f(u_e(t_n)) - \kp g(u^n,s^n) (\widehat{e}_u^{n+\frac{1}{2}}-e_u^n) \nn \\
& \quad  + \kp (g(u_e(t_n),s_e(t_n))-g(u^n,s^n)) (u_e(t_{n+\frac{1}{2}})-u_e(t_n))
- \widehat{R}_{1u}^n, \label{sav2err1a} \\
\widehat{e}_s^{n+\frac{1}{2}} - e_s^n
& = \< g(u_e(t_n),s_e(t_n))f(u_e(t_n)) - g(u^n,s^n)f(u^n), u_e(t_{n+\frac{1}{2}})-u_e(t_n) \> \nn \\
& \quad - g(u^n,s^n) \<f(u^n),\widehat{e}_u^{n+\frac{1}{2}} - e_u^n\>
- \frac{\dt}{2}\widehat{R}_{1s}^n, \label{sav2err1b}
\end{align}
\end{subequations}
where
\begin{equation}
\label{sav2trunerr1}
\|\widehat{R}_{1u}^n\|\le C_e(\dt+h^2), \qquad |\widehat{R}_{1s}^n|\le C_e(\dt+h^2).
\end{equation}

\begin{lemma}
Suppose that $h\le h_0$ and $\|u^n\|_\infty\le\beta$.
If $\dt$ is small sufficiently, we have
\begin{equation}
\label{sav2err_pf11}
\|\widehat{e}_u^{n+\frac{1}{2}}\|^2 + |\widehat{e}_s^{n+\frac{1}{2}}|^2
\le \widehat{C} (\|e_u^n\|^2 + |e_s^n|^2) + \widehat{C} C_e^2(\dt^2+\dt h^2)^2,
\end{equation}
where the constant $\widehat{C}>0$ depends on $C_*$, $|\Omega|$, $u_e$, $\kp$, and $\|f\|_{C^1[-\beta,\beta]}$.
\end{lemma}

\begin{proof}
Taking the discrete inner product of \eqref{sav2err1a} with $\dt\widehat{e}_u^{n+\frac{1}{2}}$
and rearranging the terms, we have
\begin{align*}
& \Big(1+\frac{\kp  g(u^n,s^n)}{2}\dt\Big)
\big(\|\widehat{e}_u^{n+\frac{1}{2}}\|^2 - \|e_u^n\|^2 + \|\widehat{e}_u^{n+\frac{1}{2}} - e_u^n\|^2\big)
+ \eps^2 \dt \|\nabla_h \widehat{e}_u^{n+\frac{1}{2}}\|^2 \\
& \qquad = \dt \<g(u^n,s^n)f(u^n) - g(u_e(t_n),s_e(t_n))f(u_e(t_n)), \widehat{e}_u^{n+\frac{1}{2}}\> \\
& \qquad\quad
+ \kp\dt \<(g(u_e(t_n),s_e(t_n))-g(u^n,s^n)) (u_e(t_{n+\frac{1}{2}})-u_e(t_n)), \widehat{e}_u^{n+\frac{1}{2}}\>
- \dt \<\widehat{R}_{1u}^n, \widehat{e}_u^{n+\frac{1}{2}}\>.
\end{align*}
Using Young's inequality, we then get
\begin{equation*}
- \dt \<\widehat{R}_{1u}^n, \widehat{e}_u^{n+\frac{1}{2}}\>
\le \dt \|\widehat{R}_{1u}^n\| \|\widehat{e}_u^{n+\frac{1}{2}}\|
\le \dt^2 \|\widehat{R}_{1u}^n\|^2 + \frac{1}{4}\|\widehat{e}_u^{n+\frac{1}{2}}\|^2.
\end{equation*}
Similar to the deductions of \eqref{sav1err_pf4a} and \eqref{sav1err_pf4b0},
applying Lemma \ref{lem_sav1_nonlinear} leads to
\begin{align*}
& \dt \<g(u^n,s^n)f(u^n) - g(u_e(t_n),s_e(t_n))f(u_e(t_n)), \widehat{e}_u^{n+\frac{1}{2}}\>
\le 2 C_{g} ^2\dt^2 (\|e_u^n\|^2+|e_s^n|^2) + \frac{1}{4}\|\widehat{e}_u^{n+\frac{1}{2}}\|^2, \\
& \kp\dt \<(g(u_e(t_n),s_e(t_n))-g(u^n,s^n)) (u_e(t_{n+\frac{1}{2}})-u_e(t_n)), \widehat{e}_u^{n+\frac{1}{2}}\>
\le C_1\kp^2\dt^2 (\|e_u^n\|^2+|e_s^n|^2) + \frac{1}{4} \|\widehat{e}_u^{n+\frac{1}{2}}\|^2.
\end{align*}
Then, we have
\begin{align*}
& \Big(\frac{1}{4}+\frac{\kp  g(u^n,s^n)}{2}\dt\Big) \|\widehat{e}_u^{n+\frac{1}{2}}\|^2
+ \Big(1+\frac{\kp  g(u^n,s^n)}{2}\dt\Big) \|\widehat{e}_u^{n+\frac{1}{2}} - e_u^n\|^2 \nn\\
& \qquad\quad \le \Big(1+\frac{\kp  g(u^n,s^n)}{2}\dt\Big) \|e_u^n\|^2
+ (2 C_{g} ^2 + C_1\kp^2) \dt^2 (\|e_u^n\|^2+|e_s^n|^2)
+ \dt^2 \|\widehat{R}_{1u}^n\|^2.
\end{align*}
By Remark \ref{rmk_sav2_gbound},
we then can simplify the above equation to get
\begin{align*}
\|\widehat{e}_u^{n+\frac{1}{2}}\|^2 + 4\|\widehat{e}_u^{n+\frac{1}{2}} - e_u^n\|^2
& \le (4+2G^*\kp\dt) \|e_u^n\|^2
+ (8 C_{g} ^2 + 4C_1\kp^2)\dt^2 (\|e_u^n\|^2+|e_s^n|^2)
+ 4\dt^2 \|\widehat{R}_{1u}^n\|^2.
\end{align*}
When $\dt\le1$, using \eqref{sav2trunerr1}, we obtain
\begin{align}
\label{sav2err_pf11a}
\|\widehat{e}_u^{n+\frac{1}{2}}\|^2 + 4\|\widehat{e}_u^{n+\frac{1}{2}} - e_u^n\|^2
& \le (4+8 C_{g} ^2+2G^*\kp+4C_1\kp^2) \|e_u^n\|^2 \nn \\
& \quad + (8 C_{g} ^2 + 4C_1\kp^2) |e_s^n|^2
+ 4C_e^2 \dt^2 (\dt+h^2)^2.
\end{align}

Multiplying \eqref{sav2err1b} by $2\widehat{e}_s^{n+\frac{1}{2}}$ yields
\begin{align*}
& |\widehat{e}_s^{n+\frac{1}{2}}|^2 - |e_s^n|^2 + |\widehat{e}_s^{n+\frac{1}{2}}-e_s^n|^2 \\
& \qquad\quad = 2\widehat{e}_s^{n+\frac{1}{2}}
\< g(u_e(t_n),s_e(t_n))f(u_e(t_n)) - g(u^n,s^n)f(u^n), u_e(t_{n+\frac{1}{2}})-u_e(t_n) \> \nn \\
& \qquad\qquad - 2\widehat{e}_s^{n+\frac{1}{2}} g(u^n,s^n) \<f(u^n),\widehat{e}_u^{n+\frac{1}{2}}-e_u^n\>
- \dt \widehat{R}_{1s}^n \widehat{e}_s^{n+\frac{1}{2}}.
\end{align*}
We can bound the third and second terms in the right-hand side of the above equation repectively as
\begin{eqnarray*}
& \displaystyle
- \dt \widehat{R}_{1s}^n \widehat{e}_s^{n+\frac{1}{2}}
\le \dt^2 |\widehat{R}_{1s}^n|^2 + \frac{1}{4} |\widehat{e}_s^{n+\frac{1}{2}}|^2, \\
& \displaystyle
- 2\widehat{e}_s^{n+\frac{1}{2}} g(u^n,s^n) \<f(u^n),\widehat{e}_u^{n+\frac{1}{2}}-e_u^n\>
\le \frac{1}{4} |\widehat{e}_s^{n+\frac{1}{2}}|^2
+ C_5 \|\widehat{e}_u^{n+\frac{1}{2}}-e_u^n\|^2,
\end{eqnarray*}
where $C_5>0$ depends on $C_*$, $|\Omega|$, $\uinit$, and $\|f\|_{C[-\beta,\beta]}$.
By estimating the first term in the similar way as \eqref{sav1err_pf7a},
we obtain
\begin{align*}
& |\widehat{e}_s^{n+\frac{1}{2}}|^2 - |e_s^n|^2 + |\widehat{e}_s^{n+\frac{1}{2}}-e_s^n|^2 \\
& \qquad\quad \le C_2 \dt (\|e_u^n\|^2 + |e_s^n|^2 + |\widehat{e}_s^{n+\frac{1}{2}}|^2)
+ C_5 \|\widehat{e}_u^{n+\frac{1}{2}}-e_u^n\|^2
+ \frac{1}{2} |\widehat{e}_s^{n+\frac{1}{2}}|^2 + \dt^2 |\widehat{R}_{1s}^n|^2,
\end{align*}
and thus,
\[
(1-2C_2\dt) |\widehat{e}_s^{n+\frac{1}{2}}|^2
\le 2|e_s^n|^2 + 2C_2 \dt (\|e_u^n\|^2 + |e_s^n|^2)
+ 2C_5 \|\widehat{e}_u^{n+\frac{1}{2}}-e_u^n\|^2 + 2\dt^2 |\widehat{R}_{1s}^n|^2.
\]
When $\dt\le \frac{1}{4C_2}$, using \eqref{sav2trunerr1}, we can get
\begin{equation}
\label{sav2err_pf11b}
|\widehat{e}_s^{n+\frac{1}{2}}|^2
\le \|e_u^n\|^2 + 5|e_s^n|^2 + 4C_5 \|\widehat{e}_u^{n+\frac{1}{2}}-e_u^n\|^2 + 4C_e^2 \dt^2 (\dt+h^2)^2.
\end{equation}
The sum of \eqref{sav2err_pf11a} multiplied by $C_5$ and \eqref{sav2err_pf11b} leads to \eqref{sav2err_pf11}.
\end{proof}

\begin{theorem}[Error estimate of sESAV2]
\label{thm_sav2_error}
Given a fixed time $T > 0$ and
suppose the exact solution $u_e$ is smooth enough on $[0,T]\times\overline{\Omega}$.
Assume that $\kp\ge \|f'\|_{C[-\beta,\beta]}$ and $\|\uinit\|_\infty\le\beta$.
If $\dt$ and $h$ are small sufficiently and satisfy \eqref{sav2_mbp_dt}, then we have the error estimate
for the sESAV2 scheme \eqref{eq_sav2stab} as follows:
\begin{equation*}
\|e_u^n\| + \|\nabla_h e_u^n\| + |e_s^n| \le C(\dt^2+h^2),\qquad 0 \le n \le \lfloor T/\dt\rfloor,
\end{equation*}
where the constant $C>0$ depends on $C_*$, $|\Omega|$, $T$, $u_e$, $\kp$, $\eps$, and $\|f\|_{C^1[-\beta,\beta]}$
but is independent of $\dt$ and $h$.
\end{theorem}

\begin{proof}
The difference between \eqref{eq_sav2stab} and \eqref{sav2trun} leads to
\begin{subequations}
\label{sav2err}
\begin{align}
&\delta_te_u^{n+1}
 = \eps^2 \Delta_h e_u^{n+\frac{1}{2}}
+ g(\widehat{u}^{n+\frac{1}{2}},\widehat{s}^{n+\frac{1}{2}}) f(\widehat{u}^{n+\frac{1}{2}})
- g(u_e(t_{n+\frac{1}{2}}),s_e(t_{n+\frac{1}{2}})) f(u_e(t_{n+\frac{1}{2}})) \nn \\
& \qquad\qquad
+ \kp \big(g(u_e(t_{n+\frac{1}{2}}),s_e(t_{n+\frac{1}{2}}))-g(\widehat{u}^{n+\frac{1}{2}},\widehat{s}^{n+\frac{1}{2}})\big)
\Big(\frac{u_e(t_{n+1})+u_e(t_n)}{2}-u_e(t_{n+\frac{1}{2}})\Big) \nn \\
& \qquad\qquad
- \kp g(\widehat{u}^{n+\frac{1}{2}},\widehat{s}^{n+\frac{1}{2}}) (e_u^{n+\frac{1}{2}}-\widehat{e}_u^{n+\frac{1}{2}})
- R_{2u}^n, \label{sav2erra} \\
&\delta_te_s^{n+1}
 = \Big\< g(u_e(t_{n+\frac{1}{2}}),s_e(t_{n+\frac{1}{2}})) f(u_e(t_{n+\frac{1}{2}}))
- g(\widehat{u}^{n+\frac{1}{2}}, \widehat{s}^{n+\frac{1}{2}}) f(\widehat{u}^{n+\frac{1}{2}}),
\frac{u_e(t_{n+1})-u_e(t_n)}{\dt} \Big\> \nn \\
& \qquad\qquad
+ \kp \big(g(\widehat{u}^{n+\frac{1}{2}},\widehat{s}^{n+\frac{1}{2}})-g(u_e(t_{n+\frac{1}{2}}),s_e(t_{n+\frac{1}{2}}))\big)
\Big\<\frac{u_e(t_{n+1})+u_e(t_n)}{2}-u_e(t_{n+\frac{1}{2}}),\frac{u_e(t_{n+1})-u_e(t_n)}{\dt}\Big\> \nn \\
& \qquad\qquad - g(\widehat{u}^{n+\frac{1}{2}},\widehat{s}^{n+\frac{1}{2}})
\<f(\widehat{u}^{n+\frac{1}{2}}),\delta_te_u^{n+1}\>
+ \kp g(\widehat{u}^{n+\frac{1}{2}},\widehat{s}^{n+\frac{1}{2}})
\<u^{n+\frac{1}{2}}-\widehat{u}^{n+\frac{1}{2}}, \delta_te_u^{n+1}\> \nn \\
& \qquad\qquad + \kp g(\widehat{u}^{n+\frac{1}{2}},\widehat{s}^{n+\frac{1}{2}})
\Big\<e_u^{n+\frac{1}{2}}-\widehat{e}_u^{n+\frac{1}{2}}, \frac{u_e(t_{n+1})-u_e(t_n)}{\dt} \Big\>
- R_{2s}^n. \label{sav2errb}
\end{align}
\end{subequations}

Taking the discrete inner product of \eqref{sav2erra} with $2\dt\delta_te_u^{n+1}$
and rearranging the term yield
\begin{align}
& \eps^2\|\nabla_h e_u^{n+1}\|^2 - \eps^2\|\nabla_h e_u^n\|^2 + 2\dt\|\delta_te_u^{n+1}\|^2 \nn \\
& \qquad
= 2\dt \<g(\widehat{u}^{n+\frac{1}{2}},\widehat{s}^{n+\frac{1}{2}}) f(\widehat{u}^{n+\frac{1}{2}})
- g(u_e(t_{n+\frac{1}{2}}),s_e(t_{n+\frac{1}{2}})) f(u_e(t_{n+\frac{1}{2}})),\delta_te_u^{n+1}\> \nn \\
& \qquad\quad
+ 2\kp\dt \big(g(u_e(t_{n+\frac{1}{2}}),s_e(t_{n+\frac{1}{2}}))-g(\widehat{u}^{n+\frac{1}{2}},\widehat{s}^{n+\frac{1}{2}})\big)
\Big\<\frac{u_e(t_{n+1})+u_e(t_n)}{2}-u_e(t_{n+\frac{1}{2}}),\delta_te_u^{n+1}\Big\> \nn \\
& \qquad\quad - 2\kp \dt  g(\widehat{u}^{n+\frac{1}{2}},\widehat{s}^{n+\frac{1}{2}}) \<e_u^{n+\frac{1}{2}}-\widehat{e}_u^{n+\frac{1}{2}},\delta_te_u^{n+1}\> - 2\dt\<R_{2u}^n,\delta_te_u^{n+1}\>. \label{sav2err_pf0}
\end{align}
Since $ g(\widehat{u}^{n+\frac{1}{2}},\widehat{s}^{n+\frac{1}{2}})\ge {\widetilde G}_{*}>0$, we get
\begin{align*}
& 2\kp \dt  g(\widehat{u}^{n+\frac{1}{2}},\widehat{s}^{n+\frac{1}{2}}) \<e_u^{n+\frac{1}{2}}-\widehat{e}_u^{n+\frac{1}{2}},\delta_te_u^{n+1}\>\nn\\
&\qquad = 2\kp \dt  g(\widehat{u}^{n+\frac{1}{2}},\widehat{s}^{n+\frac{1}{2}})
\Big\<\frac{e_u^{n+1}-e_u^n}{2}+e_u^n-\widehat{e}_u^{n+\frac{1}{2}},\delta_te_u^{n+1}\Big\> \\
& \qquad= \kp  g(\widehat{u}^{n+\frac{1}{2}},\widehat{s}^{n+\frac{1}{2}}) \|e_u^{n+1}-e_u^n\|^2
+ 2\kp \dt  g(\widehat{u}^{n+\frac{1}{2}},\widehat{s}^{n+\frac{1}{2}}) \<e_u^n-\widehat{e}_u^{n+\frac{1}{2}},\delta_te_u^{n+1}\> \\
&\qquad \ge \kp {\widetilde G}_{*} \|e_u^{n+1}-e_u^n\|^2
+ 2\kp \dt  g(\widehat{u}^{n+\frac{1}{2}},\widehat{s}^{n+\frac{1}{2}}) \<e_u^n-\widehat{e}_u^{n+\frac{1}{2}},\delta_te_u^{n+1}\> \\
&\qquad = \kp {\widetilde G}_{*} \|e_u^{n+1}\|^2 - \kp {\widetilde G}_{*} \|e_u^n\|^2 - 2\kp {\widetilde G}_{*}\dt \<e_u^n, \delta_te_u^{n+1}\> 
+ 2\kp \dt  g(\widehat{u}^{n+\frac{1}{2}},\widehat{s}^{n+\frac{1}{2}}) \<e_u^n-\widehat{e}_u^{n+\frac{1}{2}},\delta_te_u^{n+1}\>,
\end{align*}
where we have used \eqref{sav1err_pf0} in the last step.
Then, we obtain from \eqref{sav2err_pf0} that
\begin{align}
& {\widetilde G}_{*}\kp \|e_u^{n+1}\|^2 - {\widetilde G}_{*}\kp \|e_u^n\|^2
+ \eps^2\|\nabla_h e_u^{n+1}\|^2 - \eps^2\|\nabla_h e_u^n\|^2 + 2\dt \|\delta_te_u^{n+1}\|^2 \nn \\
& \qquad = 2\dt \<g(\widehat{u}^{n+\frac{1}{2}},\widehat{s}^{n+\frac{1}{2}}) f(\widehat{u}^{n+\frac{1}{2}})
- g(u_e(t_{n+\frac{1}{2}}),s_e(t_{n+\frac{1}{2}})) f(u_e(t_{n+\frac{1}{2}})), \delta_te_u^{n+1}\> \nn \\
& \qquad\quad + 2\kp\dt \big(g(u_e(t_{n+\frac{1}{2}}),s_e(t_{n+\frac{1}{2}}))-g(\widehat{u}^{n+\frac{1}{2}},\widehat{s}^{n+\frac{1}{2}})\big)
\Big\<\frac{u_e(t_{n+1})+u_e(t_n)}{2}-u_e(t_{n+\frac{1}{2}}),\delta_te_u^{n+1}\Big\> \nn \\
& \qquad\quad + 2{\widetilde G}_{*}\kp\dt \<e_u^n, \delta_te_u^{n+1}\>
+ 2\kp\dt  g(\widehat{u}^{n+\frac{1}{2}},\widehat{s}^{n+\frac{1}{2}}) \<\widehat{e}_u^{n+\frac{1}{2}}-e_u^n,\delta_te_u^{n+1}\>
- 2 \dt \<R_{2u}^n, \delta_te_u^{n+1}\>. \label{sav2err_pf1}
\end{align}
For the last three terms in the right-hand side of \eqref{sav2err_pf1}, we have respectively
\begin{align}
2\kp\dt  g(\widehat{u}^{n+\frac{1}{2}},\widehat{s}^{n+\frac{1}{2}}) \<\widehat{e}_u^{n+\frac{1}{2}}-e_u^n,\delta_te_u^{n+1}\>
& \le 2G^*\kp\dt (\|\widehat{e}_u^{n+\frac{1}{2}}\|+\|e_u^n\|) \|\delta_te_u^{n+1}\| \nn \\
& \le 6{G^*}^2\kp^2\dt (\|\widehat{e}_u^{n+\frac{1}{2}}\|^2+\|e_u^n\|^2) + \frac{\dt}{3}\|\delta_te_u^{n+1}\|^2,
\label{sav2err_pf2a} \\
2{\widetilde G}_{*}\kp\dt \<e_u^n, \delta_te_u^{n+1}\>
& \le 3{\widetilde G}_{*}^2\kp^2 \dt\|e_u^n\|^2 + \frac{\dt}{3} \|\delta_te_u^{n+1}\|^2, \label{sav2err_pf2b} \\
- 2\dt \<R_{2u}^n, \delta_te_u^{n+1}\>
& \le 3\dt \|R_{2u}^n\|^2 + \frac{\dt}{3}\|\delta_te_u^{n+1}\|^2. \label{sav2err_pf2c}
\end{align}
By the  energy dissipation and MBP of the sESAV1 substep \eqref{eq_sav2stab0},
we know that $\widehat{u}^{n+\frac{1}{2}}$ and $\widehat{s}^{n+\frac{1}{2}}$ are bounded uniformly.
By conducting the similar deductions to the proof of Lemma \ref{lem_sav1_nonlinear},
we can obtain
\begin{subequations}
\label{sav1_nonlinearhat}
\begin{align}
& |g(\widehat{u}^{n+\frac{1}{2}},\widehat{s}^{n+\frac{1}{2}})
- g(u_e(t_{n+\frac{1}{2}}),s_e(t_{n+\frac{1}{2}}))|
\le  C_{g} (\|\widehat{e}_u^{n+\frac{1}{2}}\|+|\widehat{e}_s^{n+\frac{1}{2}}|), \\
& \|g(\widehat{u}^{n+\frac{1}{2}},\widehat{s}^{n+\frac{1}{2}}) f(\widehat{u}^{n+\frac{1}{2}})
- g(u_e(t_{n+\frac{1}{2}}),s_e(t_{n+\frac{1}{2}})) f(u_e(t_{n+\frac{1}{2}}))\|
\le  C_{g} (\|\widehat{e}_u^{n+\frac{1}{2}}\|+|\widehat{e}_s^{n+\frac{1}{2}}|),
\end{align}
\end{subequations}
where $ C_{g} >0$ is the same constant defined in Lemma \ref{lem_sav1_nonlinear}.
Then, the first and second terms in the right-hand side of \eqref{sav2err_pf1} can be bounded respectively as
\begin{align}
& 2\dt \<g(\widehat{u}^{n+\frac{1}{2}},\widehat{s}^{n+\frac{1}{2}}) f(\widehat{u}^{n+\frac{1}{2}})
- g(u_e(t_{n+\frac{1}{2}}),s_e(t_{n+\frac{1}{2}})) f(u_e(t_{n+\frac{1}{2}})), \delta_te_u^{n+1}\> \nn \\
& \qquad \le 2 C_{g} \dt (\|\widehat{e}_u^{n+\frac{1}{2}}\|+|\widehat{e}_s^{n+\frac{1}{2}}|) \|\delta_te_u^{n+1}\|\nn\\
&\qquad\le 6 C_{g} ^2 \dt (\|\widehat{e}_u^{n+\frac{1}{2}}\|^2 + |\widehat{e}_s^{n+\frac{1}{2}}|^2)
+ \frac{\dt}{3} \|\delta_te_u^{n+1}\|^2, \label{sav2err_pf2d}
\end{align}
and
\begin{align}
& 2\kp\dt \big(g(u_e(t_{n+\frac{1}{2}}),s_e(t_{n+\frac{1}{2}}))-g(\widehat{u}^{n+\frac{1}{2}},\widehat{s}^{n+\frac{1}{2}})\big)
\Big\<\frac{u_e(t_{n+1})+u_e(t_n)}{2}-u_e(t_{n+\frac{1}{2}}),\delta_te_u^{n+1}\Big\> \nn \\
& \qquad \le  C_{g} \kp\dt (\|\widehat{e}_u^{n+\frac{1}{2}}\|+|\widehat{e}_s^{n+\frac{1}{2}}|)
(\|u_e(t_{n+1})\|+\|u_e(t_n)\|+2\|u_e(t_{n+\frac{1}{2}})\|) \|\delta_te_u^{n+1}\| \nn \\
& \qquad \le C_1 \kp^2 \dt(\|\widehat{e}_u^{n+\frac{1}{2}}\|^2 + |\widehat{e}_s^{n+\frac{1}{2}}|^2)
+ \frac{\dt}{3} \|\delta_te_u^{n+1}\|^2, \label{sav2err_pf2e}
\end{align}
where $C_1>0$ has the same dependence as the constant $C_1$ used in \eqref{sav1err_pf4b0} but may have a different value.
Substituting \eqref{sav2err_pf2a}--\eqref{sav2err_pf2e} into \eqref{sav2err_pf1} leads to
\begin{align}
& {\widetilde G}_{*}\kp \|e_u^{n+1}\|^2 - {\widetilde G}_{*}\kp \|e_u^n\|^2
+ \eps^2\|\nabla_h e_u^{n+1}\|^2 - \eps^2\|\nabla_h e_u^n\|^2 + \frac{\dt}{3} \|\delta_te_u^{n+1}\|^2 \nn \\
& \qquad \le
(6{G^*}^2\kp^2 + 6 C_{g} ^2 + C_1\kp^2) \dt \|\widehat{e}_u^{n+\frac{1}{2}}\|^2
+ (6 C_{g} ^2 + C_1\kp^2) \dt  |\widehat{e}_s^{n+\frac{1}{2}}|^2 \nn \\
& \qquad\quad + (3{\widetilde G}_{*}^2+6{G^*}^2)\kp^2\dt \|e_u^n\|^2 + 3\dt \|R_{2u}^n\|^2. \label{sav2err_pf3}
\end{align}

Multiplying \eqref{sav2errb} by $2\dt e_s^{n+1}$ yields
\begin{align}
& |e_s^{n+1}|^2 - |e_s^n|^2 + |e_s^{n+1}-e_s^n|^2 \nn \\
& \quad = 2e_s^{n+1} \< g(u_e(t_{n+\frac{1}{2}}),s_e(t_{n+\frac{1}{2}})) f(u_e(t_{n+\frac{1}{2}}))
- g(\widehat{u}^{n+\frac{1}{2}},\widehat{s}^{n+\frac{1}{2}}) f(\widehat{u}^{n+\frac{1}{2}}),
u_e(t_{n+1})-u_e(t_n) \> \nn \\
& \quad\quad + \kp e_s^{n+1}
\big(g(\widehat{u}^{n+\frac{1}{2}},\widehat{s}^{n+\frac{1}{2}})-g(u_e(t_{n+\frac{1}{2}}),s_e(t_{n+\frac{1}{2}}))\big)
\<u_e(t_{n+1})+u_e(t_n)-2u_e(t_{n+\frac{1}{2}}),u_e(t_{n+1})-u_e(t_n)\> \nn \\
& \quad\quad - 2\dt e_s^{n+1} g(\widehat{u}^{n+\frac{1}{2}},\widehat{s}^{n+\frac{1}{2}})
\<f(\widehat{u}^{n+\frac{1}{2}}),\delta_te_u^{n+1}\>
+ 2\kp\dt e_s^{n+1} g(\widehat{u}^{n+\frac{1}{2}},\widehat{s}^{n+\frac{1}{2}})
\<u^{n+\frac{1}{2}}-\widehat{u}^{n+\frac{1}{2}}, \delta_te_u^{n+1}\> \nn \\
& \quad\quad + 2\kp e_s^{n+1} g(\widehat{u}^{n+\frac{1}{2}},\widehat{s}^{n+\frac{1}{2}})
\<e_u^{n+\frac{1}{2}}-\widehat{e}_u^{n+\frac{1}{2}}, u_e(t_{n+1})-u_e(t_n)\>
- 2\dt R_{2s}^n e_s^{n+1}. \label{sav2err_pf4}
\end{align}
The last term in the right-hand side of \eqref{sav2err_pf4} can be estimated by
\begin{equation}
\label{sav2err_pf5a}
- 2\dt R_{2s}^n e_s^{n+1} \le \dt |e_s^{n+1}|^2 + \dt |R_{2s}^n|^2.
\end{equation}
By the boundedness of $u^{n+\frac{1}{2}}$, $\widehat{u}^{n+\frac{1}{2}}$, and $\widehat{s}^{n+\frac{1}{2}}$,
the sum of the third and fourth terms in the right-hand side of \eqref{sav2err_pf4}
can be estimated similarly to \eqref{sav1err_pf7b} as follows:
\begin{align}
& - 2\dt e_s^{n+1} g(\widehat{u}^{n+\frac{1}{2}},\widehat{s}^{n+\frac{1}{2}})
\<f(\widehat{u}^{n+\frac{1}{2}}),\delta_te_u^{n+1}\>
+ 2\kp\dt e_s^{n+1} g(\widehat{u}^{n+\frac{1}{2}},\widehat{s}^{n+\frac{1}{2}})
\<u^{n+\frac{1}{2}}-\widehat{u}^{n+\frac{1}{2}}, \delta_te_u^{n+1}\> \nn \\
& \qquad\le 2G^*\kp \dt (\|u^{n+\frac{1}{2}}\|+\|\widehat{u}^{n+\frac{1}{2}}\|) |e_s^{n+1}| \|\delta_te_u^{n+1}\|
+ 2G^*\dt \|f(\widehat{u}^{n+\frac{1}{2}})\| |e_s^{n+1}| \|\delta_te_u^{n+1}\| \nn \\
& \qquad\le C_6 \dt |e_s^{n+1}|^2 + \frac{\dt}{3} \|\delta_te_u^{n+1}\|^2, 
\end{align}
where $C_6>0$ depends on $C_*$, $|\Omega|$, $\uinit$, $\kp$, and $\|f\|_{C[-\beta,\beta]}$.
Then, using the facts that $\|u_e(t_{n+1})-u_e(t_n)\|\le C\dt$,
$\|u_e(t_{n+1})+u_e(t_n)-2u_e(t_{n+\frac{1}{2}})\|\le C\tau^2$ (where $C>0$ is a constant due to smoothness of $u_e$), and the inequalities \eqref{sav1_nonlinearhat},
in the similar spirit of deriving \eqref{sav1err_pf7a},
the sum of the first, second and fifth terms in the right-hand side of \eqref{sav2err_pf4} can be bounded above by
\begin{equation}
\label{sav2err_pf5c}
\dt \big(\|\widehat{e}_u^{n+\frac{1}{2}}\|^2 + |\widehat{e}_s^{n+\frac{1}{2}}|^2
+ \|e_u^n\|^2 + \|e_u^{n+1}\|^2 + |e_s^{n+1}|^2\big)
\end{equation}
multiplied with a positive constant depending on $C_*$, $|\Omega|$, $u_e$, $\kp$, and $\|f\|_{C^1[-\beta,\beta]}$.
Combining \eqref{sav2err_pf4} with \eqref{sav2err_pf5a}--\eqref{sav2err_pf5c}, we obtain
\begin{align}
|e_s^{n+1}|^2 - |e_s^n|^2
& \le C_7 \dt \big(\|\widehat{e}_u^{n+\frac{1}{2}}\|^2 + |\widehat{e}_s^{n+\frac{1}{2}}|^2
+ \|e_u^n\|^2 + \|e_u^{n+1}\|^2 + |e_s^{n+1}|^2\big) \nn\\
& \quad
+ \frac{\dt}{3} \|\delta_te_u^{n+1}\|^2 + \dt |R_{2s}^n|^2 \label{sav2err_pf6}
\end{align}
with $C_7$ depending on $C_*$, $|\Omega|$, $u_e$, $\kp$, and $\|f\|_{C^1[-\beta,\beta]}$.

Adding \eqref{sav2err_pf3} and \eqref{sav2err_pf6}, we obtain
\begin{align}
& {\widetilde G}_{*}\kp (\|e_u^{n+1}\|^2 - \|e_u^n\|^2)
+ \eps^2 (\|\nabla_h e_u^{n+1}\|^2 - \|\nabla_h e_u^n\|^2)
+ (|e_s^{n+1}|^2 - |e_s^n|^2) \nn \\
& \qquad \le C_8 \dt \big(\|\widehat{e}_u^{n+\frac{1}{2}}\|^2 + |\widehat{e}_s^{n+\frac{1}{2}}|^2
+ \|e_u^n\|^2 + \|e_u^{n+1}\|^2 + |e_s^{n+1}|^2 \big)
+ 3\dt \|R_{2u}^n\|^2 + \dt |R_{2s}^n|^2, \label{sav2err_pf7}
\end{align}
where $C_8>0$ depends on $C_*$, $|\Omega|$, $u_e$, $\kp$, and $\|f\|_{C^1[-\beta,\beta]}$.
Substituting \eqref{sav2err_pf11} into \eqref{sav2err_pf7} and using the estimate \eqref{sav2trunerr},
when $\dt\le1$, we have
\begin{align*}
& {\widetilde G}_{*}\kp (\|e_u^{n+1}\|^2 - \|e_u^n\|^2)
+ \eps^2 (\|\nabla_h e_u^{n+1}\|^2 - \|\nabla_h e_u^n\|^2)
+ (|e_s^{n+1}|^2 - |e_s^n|^2) \nn \\
& \qquad \le C_8(\widehat{C}+1) \dt (\|e_u^n\|^2 + \|e_u^{n+1}\|^2 + |e_s^{n+1}|^2)
+ (C_8 \widehat{C}+4) C_e^2 \dt(\dt^2+h^2)^2.
\end{align*}
When $\dt$ is small sufficiently,
similar to the last paragraph in the proof of Theorem \ref{thm_sav1_error},
applying the discrete Gronwall's inequality yields
\[
{\widetilde G}_{*}\kp \|e_u^n\|^2 + \eps^2 \|\nabla_h e_u^n\|^2 + |e_s^n|^2 \le C(\dt^2+h^2)^2,
\]
which completes the proof.
\end{proof}

%

\section{Numerical experiments}
\label{sect_experiment}

This section is devoted to numerical tests and comparisons between
the proposed sESAV schemes and existing SAV schemes listed in Sections \ref{sect_classicSAV} and \ref{sect_ESAV}.
We consider the Allen--Cahn equation \eqref{AllenCahn} in two-dimensional spatial domain $\Omega=(0,1)\times(0,1)$
equipped with periodic boundary conditions,
so that the schemes can be solved efficiently by the fast Fourier transform.
We take two types of commonly-used nonlinear functions $f(u)$.
One is given by
\begin{equation}
\label{f_dw}
f(u)=-F'(u)=u-u^3
\end{equation}
with $F$ being the {\em double-well potential} $$F(u)=\frac{1}{4}(u^2-1)^2.$$
In this case, one has $\beta=1$ and $\|f'\|_{C[-1,1]}=2$.
The constant $C_0$ in \eqref{sav_e2} is $C_0=\frac{1}{4}(\kp^2+2\kp)$.
The other one is determined by the {\em Flory--Huggins potential}
\[
F(u) = \frac{\theta}{2}[(1+u)\ln(1+u) + (1-u)\ln(1-u)] - \frac{\theta_c}{2}u^2,
\]
which gives
\begin{equation}
\label{f_fh}
f(u) = -F'(u) = \frac{\theta}{2}\ln\frac{1-u}{1+u} + \theta_c u,
\end{equation}
where $\theta_c>\theta>0$.
In the following experiments,
we set $\theta=0.8$ and $\theta_c=1.6$, then the positive root of $f(\rho)=0$ gives us  $\beta\approx 0.9575$,
and $\|f'\|_{C[-\beta,\beta]}\approx8.02$.
The constant $C_0$ in \eqref{sav_e2} is then determined by $C_0=-F(\alpha)+\frac{\kp}{2}\alpha^2$,
where $\alpha>0$ solves $f(\alpha)+\kp\alpha=0$.

\subsection{Convergence in time}

We first verify the convergence order in time for the proposed sESAV schemes. Let us set $\eps=0.01$ in \eqref{AllenCahn} and take a smooth initial value
\[
\uinit(x,y) = 0.1 \sin(2\pi x) \sin(2\pi y).
\]
The temporal convergence tests  are conducted
by fixing the spatial mesh size $h=1/512$. As requested by the stabilizing condition $\kp\ge \|f'\|_{C[-\beta,\beta]}$,
we set $\kp=2$ for the double-well potential case (i.e., $f(u)$ given by \eqref{f_dw})
and $\kp=8.02$ for the Flory--Huggins potential case (i.e., $f(u)$ given by \eqref{f_fh}).
We compute the numerical solutions at $t=2$ using the sESAV1 and sESAV2 schemes
with various time step sizes $\dt=2^{-k}$, $k=4,5,\dots,12$.
To compute the numerical errors,
we treat the sESAV2 solution obtained by $\dt=0.1\times2^{-12}$ as the benchmark solution.
Figure \ref{fig_conv} shows the relation between the $L^2$-norm error and the time step size,
where the left picture corresponds to the double-well potential case
and the right one for the Flory--Huggins potential case.
The first-order temporal accuracy for sESAV1 and the second-order for sESAV2  are  observed for both cases as expected.

\begin{figure}[!ht]
\centering
\includegraphics[width=0.45\textwidth]{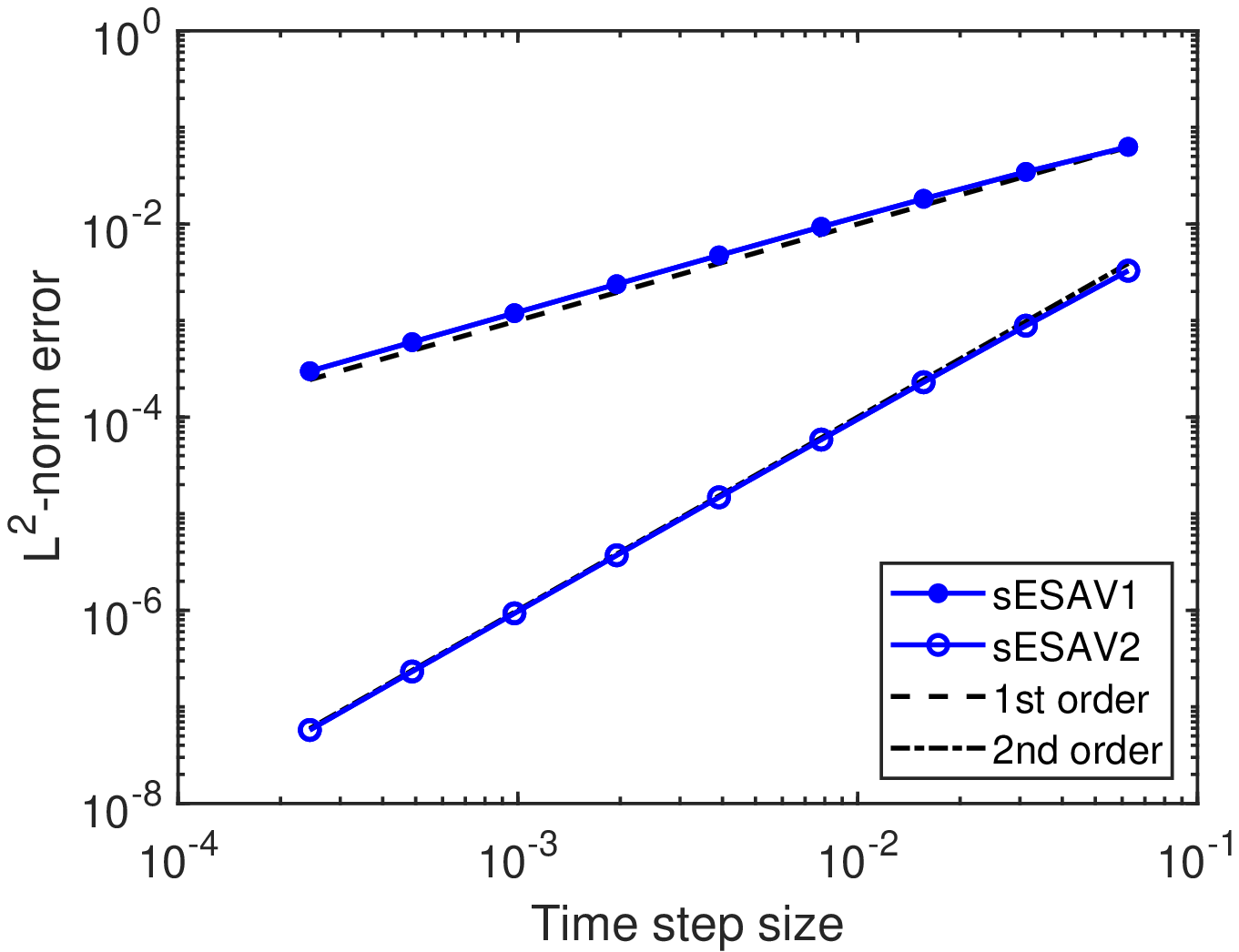}
\includegraphics[width=0.45\textwidth]{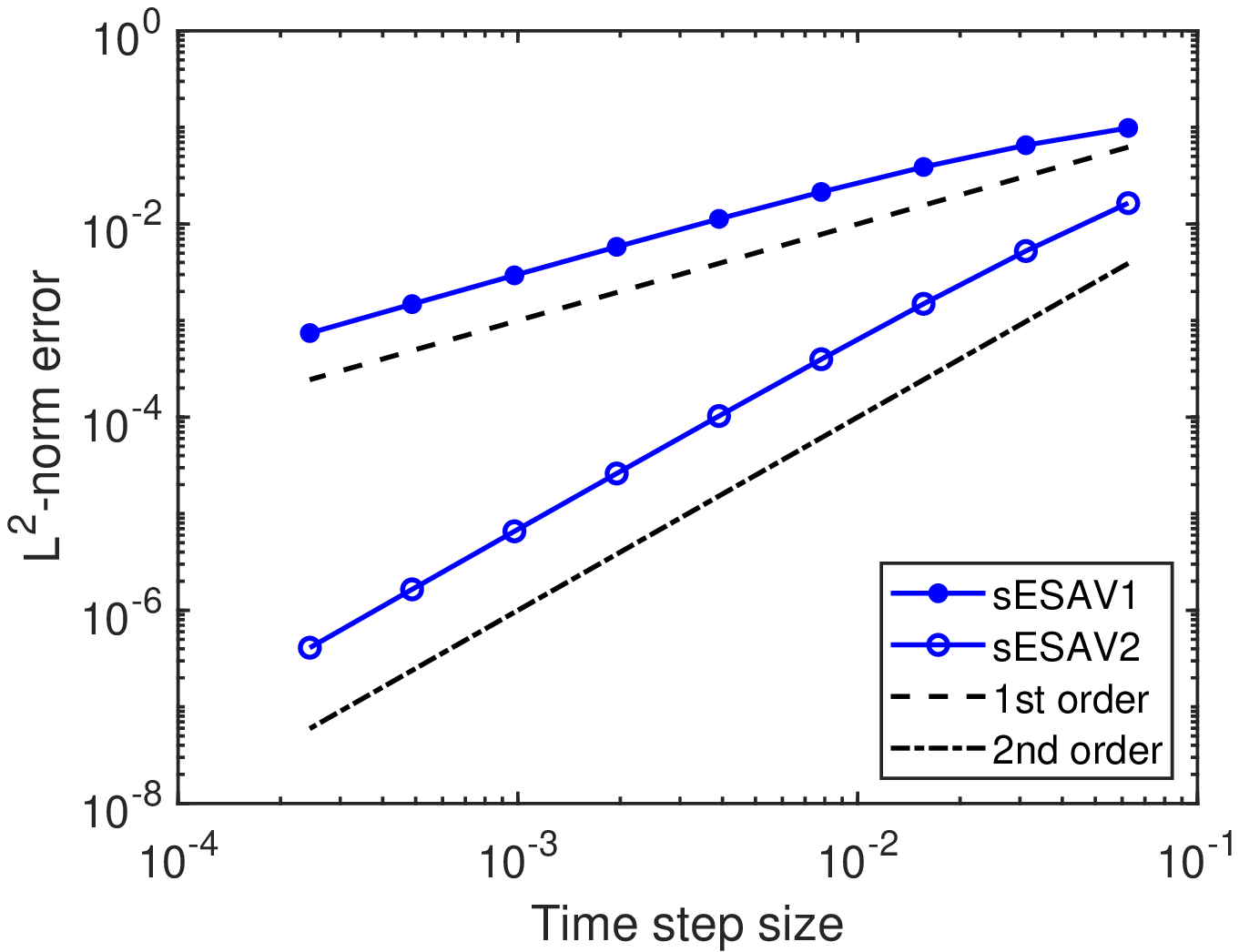}
\caption{The $L^2$-norm errors vs. the time step size produced by the proposed sESAV1 and sESAV2 schemes for the double-well potential case \eqref{f_dw} (left) and the Flory--Huggins potential case \eqref{f_fh} (right).}
\label{fig_conv}
\end{figure}

\subsection{Comparisons with existing SAV schemes}

In the following numerical experiments,
we compare the proposed sESAV schemes with classic SAV and ESAV schemes
by focusing on the MBP and energy dissipation law.
While  various modified energies are introduced as approximations of the original energy in discrete settings in order to facilitate the proof of energy dissipation law, the original one possesses the most accurate physical meaning for the model problem.
Therefore, we are concerned about the behavior of the original (discrete) energy $E_h(u)$ defined in \eqref{egydis} for
reflecting the phase transition process.
The dynamic process considered usually needs a long-time evolution to reach the steady state; here we conduct simulations in a short time interval for the comparison among these schemes.

Let us still consider the problem \eqref{AllenCahn} with $\eps=0.01$.
We adopt the uniform spatial mesh with $h=1/512$ and give the initial value by random numbers between $-0.8$ and $0.8$ on each mesh point. We then set the time step size $\dt=0.01$, and compute the numerical solutions by using the sESAV schemes, the classic SAV schemes (SAV1 and SAV2), and the ESAV schemes (ESAV1 and ESAV2).
Note that we set $\delta=C_0+0.01$
for the classic SAV schemes \eqref{eq_sav1shen} and \eqref{eq_sav2shen}.
For all comparison experiments, we will consider two settings for the stabilizing parameter:  $\kp=\|f'\|_{C[-\beta,\beta]}$ and $\kp=\frac{1}{2}\|f'\|_{C[-\beta,\beta]}$,
where the former one satisfies the requirement for the MBP preservation for the sESAV schemes
and the latter one was adopted in \cite{ShXuYa19} for the classic SAV schemes.
In addition, we take the numerical results obtained by the IFRK4 scheme \cite{JuLiQiYa21}
with the small time step size $10^{-4}$ as the benchmark solution.

First, we test the double-well  potential case \eqref{f_dw},
and correspondingly, set the stabilizing parameter $\kp=1$ and $\kp=2$ respectively to carry out the experiments.
Figure \ref{fig_comp_poly1} shows the evolutions of the supremum norms and the energies of simulated solutions
computed by the sESAV1, SAV1, and ESAV1 schemes.
For either $\kp=1$ or $\kp=2$, the sESAV1 scheme preserves the MBP,
while the supremum norms of the SAV1 and ESAV1 solutions obviously evolve beyond $1$,
which means that the MBP is violated.
The energy dissipation are observed for these three schemes,
where the sESAV1 scheme provides the most accurate result.
In addition, the larger $\kp$ leads to larger errors in the results,
especially for the ESAV1 scheme.
Figure \ref{fig_comp_poly2} plots corresponding results computed by the second-order schemes.
Again, only the sESAV2 scheme preserves the MBP and the energy dissipation perfectly.
The SAV2 and ESAV2 solutions evolve beyond $1$
but closer to $1$ than their first-order results due to the higher-order temporal accuracy.

\begin{figure}[!ht]
\centerline{
\includegraphics[width=0.43\textwidth]{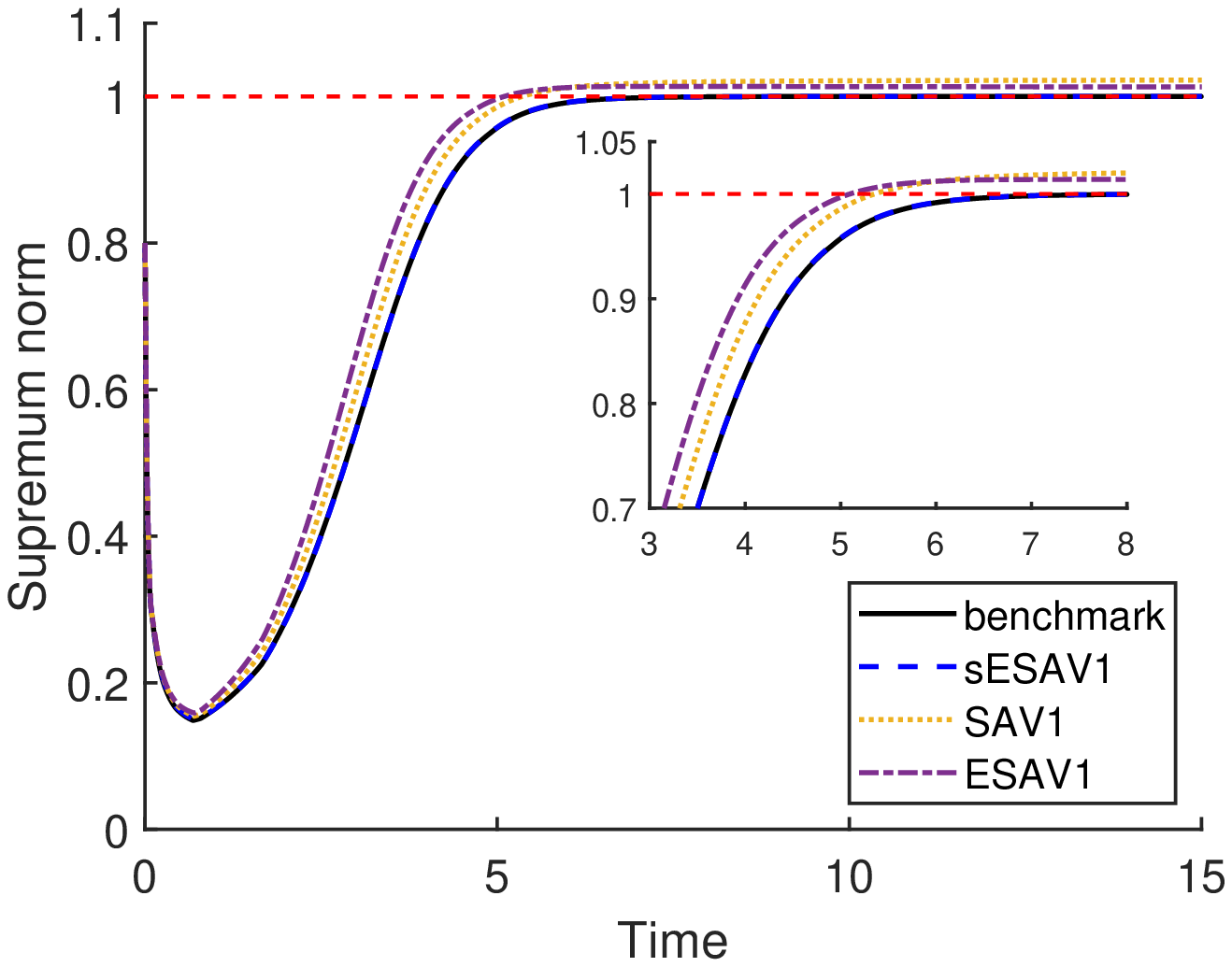}
\includegraphics[width=0.43\textwidth]{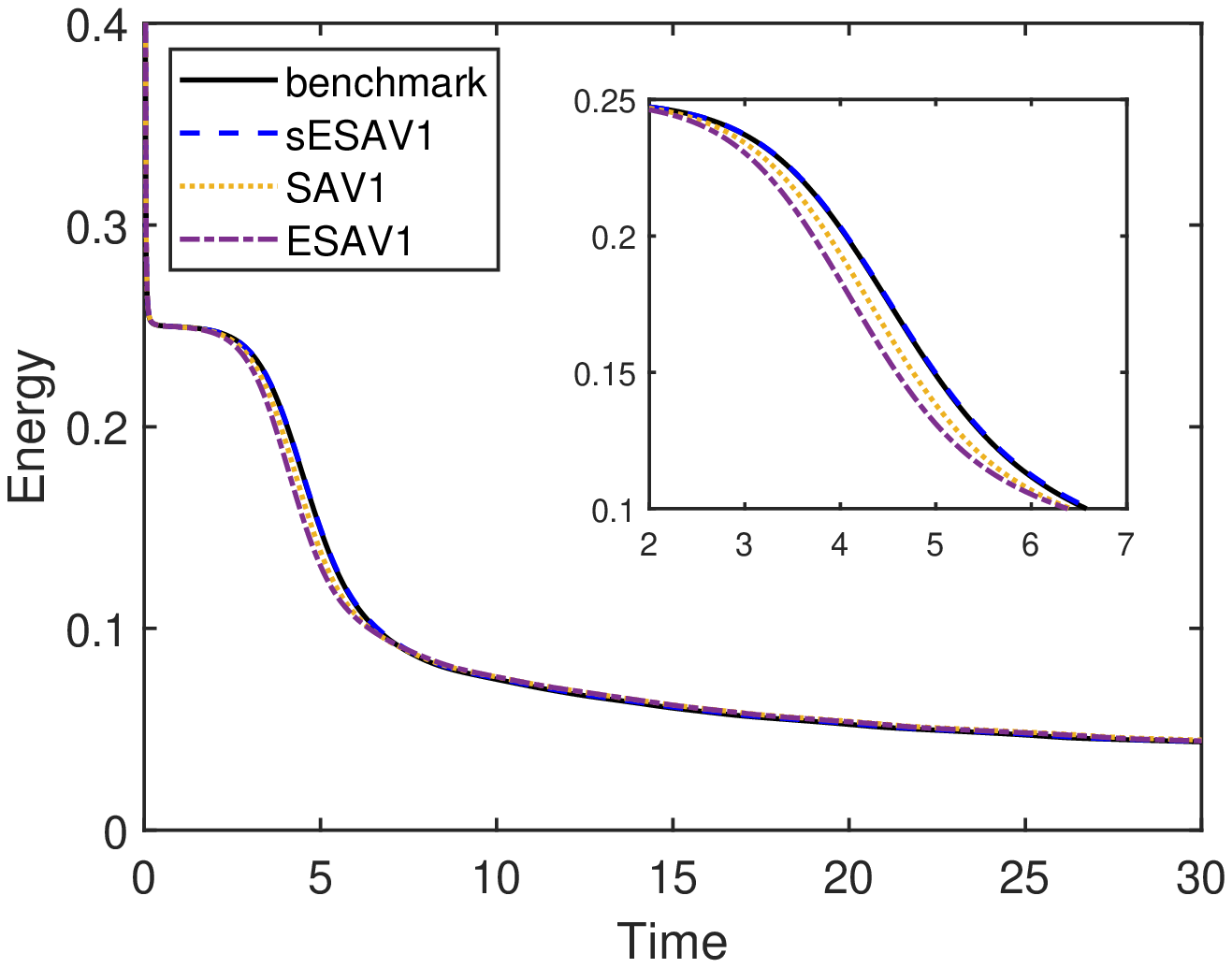}}
\centerline{
\includegraphics[width=0.43\textwidth]{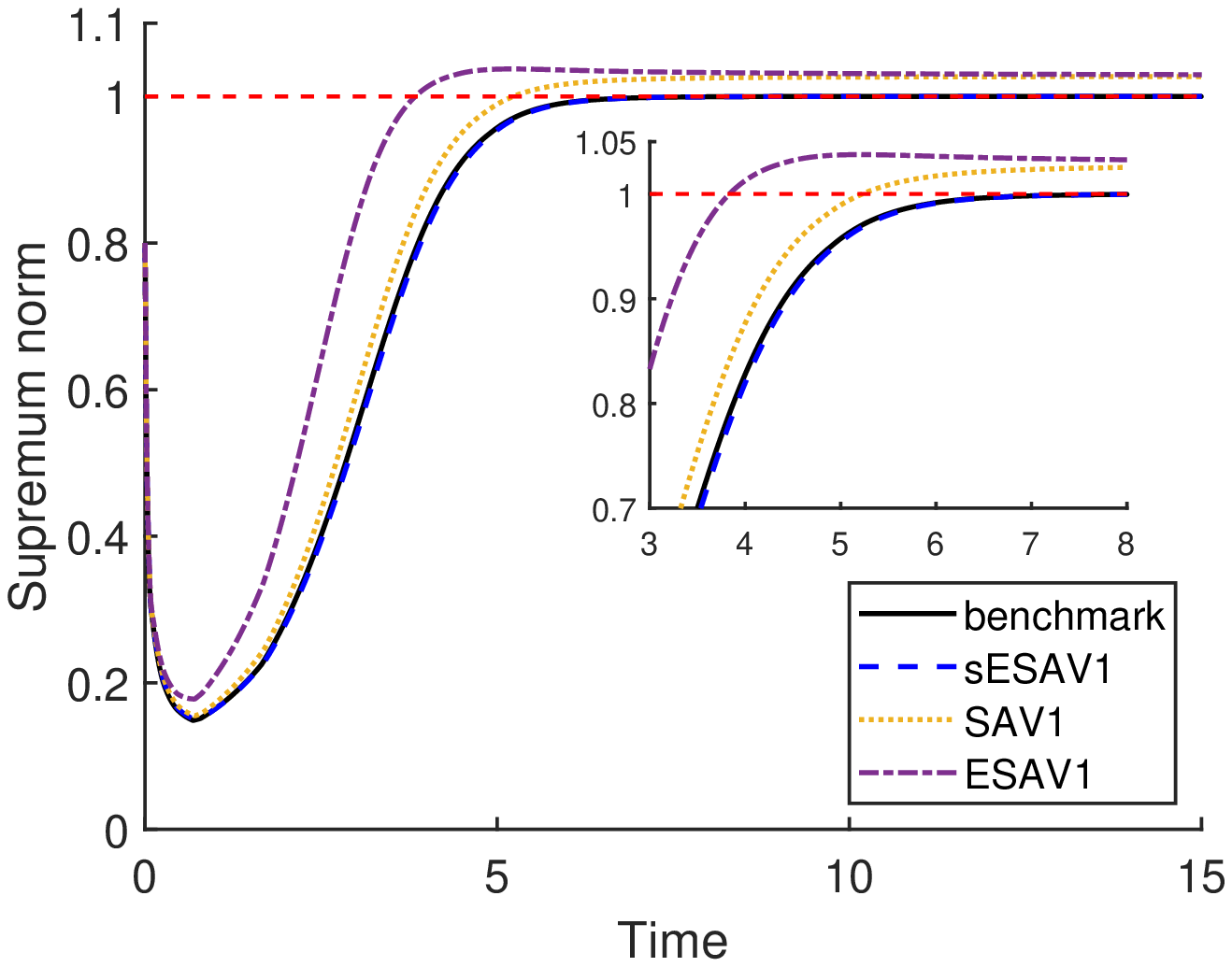}
\includegraphics[width=0.43\textwidth]{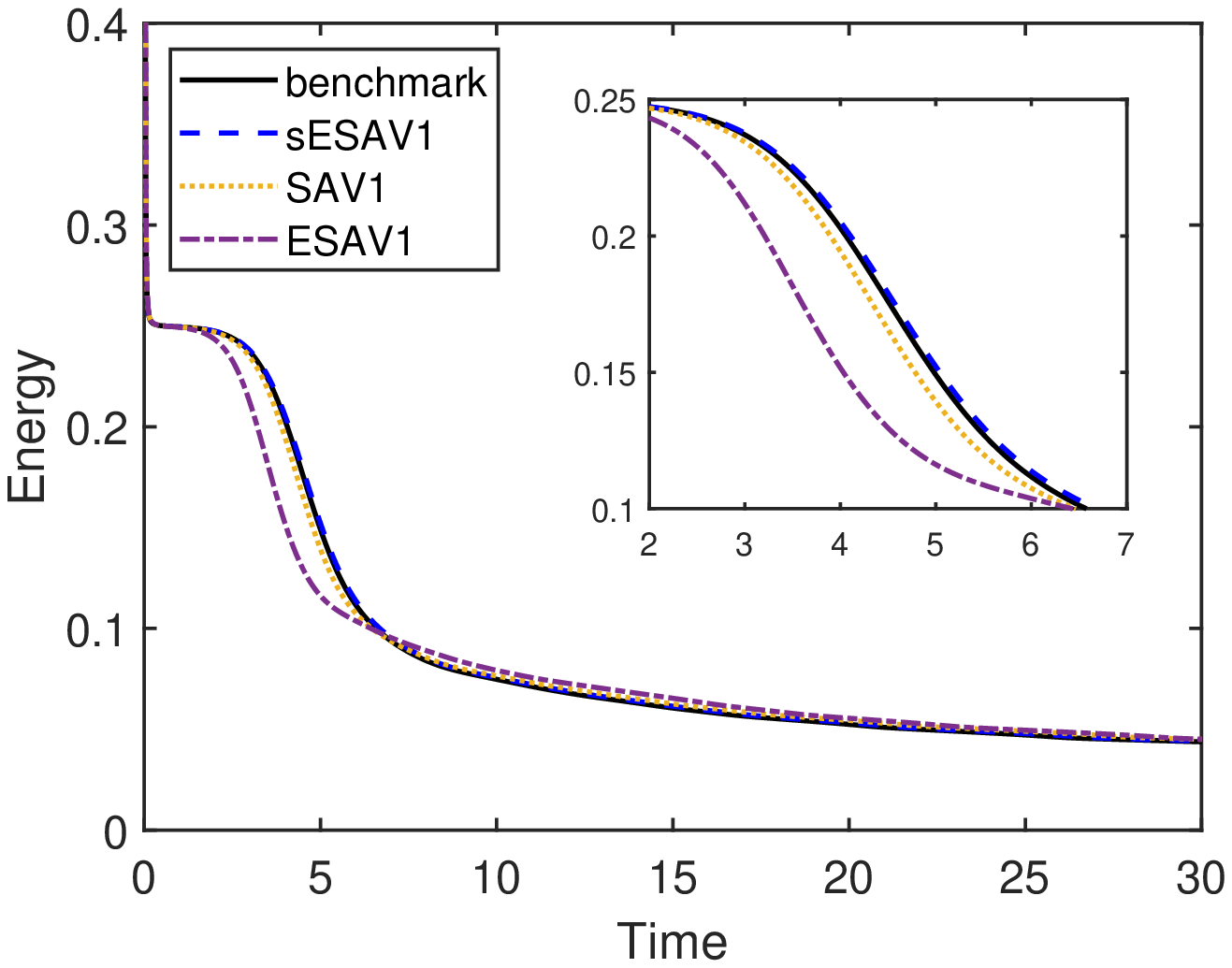}}
\caption{Evolutions of the supremum norms and the energies of  simulated solutions computed by the sESAV1, SAV1, and ESAV1 schemes
with  $\dt=0.01$ and  $\kp=1$ (top row) or $\kp=2$ (bottom row) for the double-well potential case.}
\label{fig_comp_poly1}
\end{figure}

\begin{figure}[!ht]
\centerline{
\includegraphics[width=0.43\textwidth]{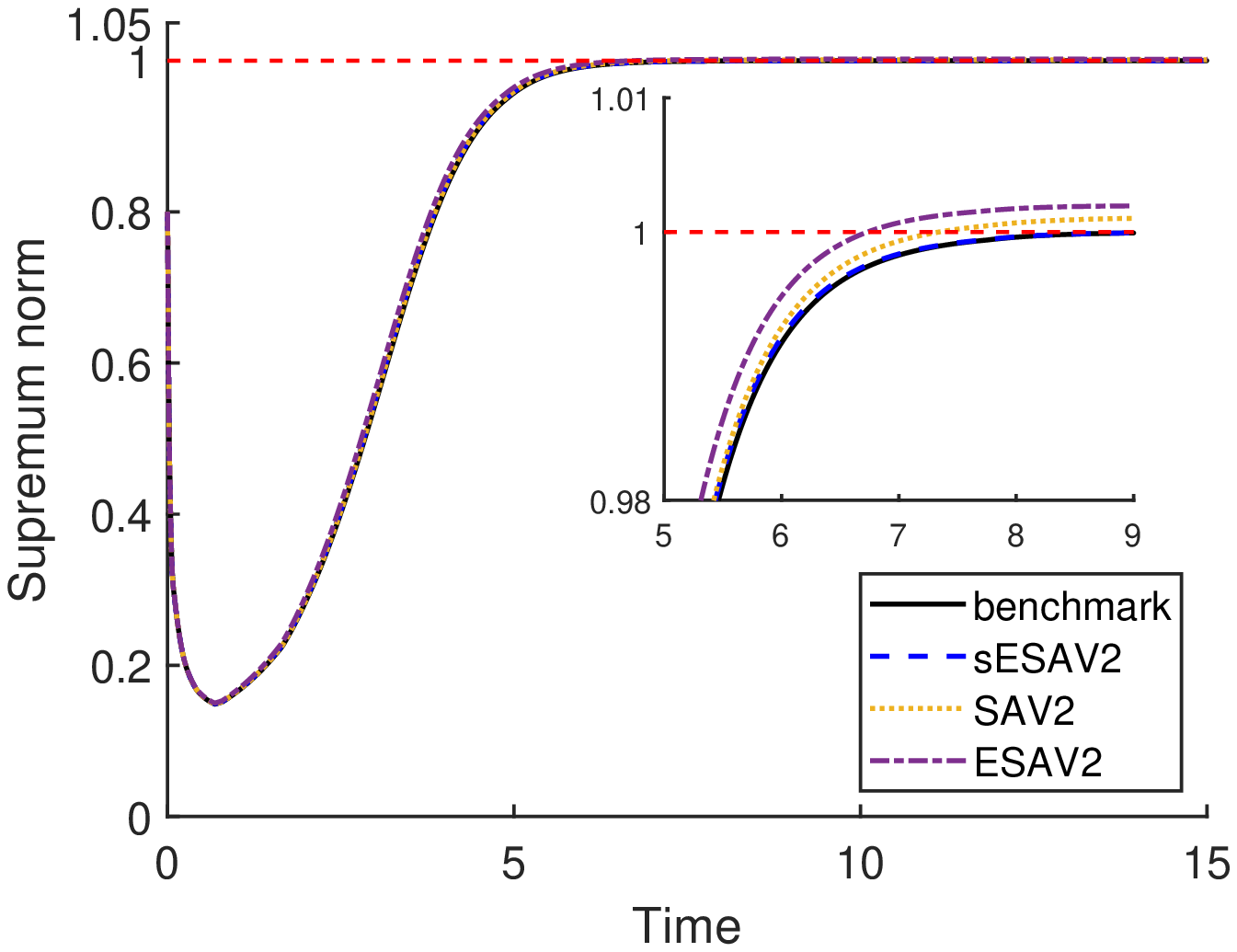}
\includegraphics[width=0.43\textwidth]{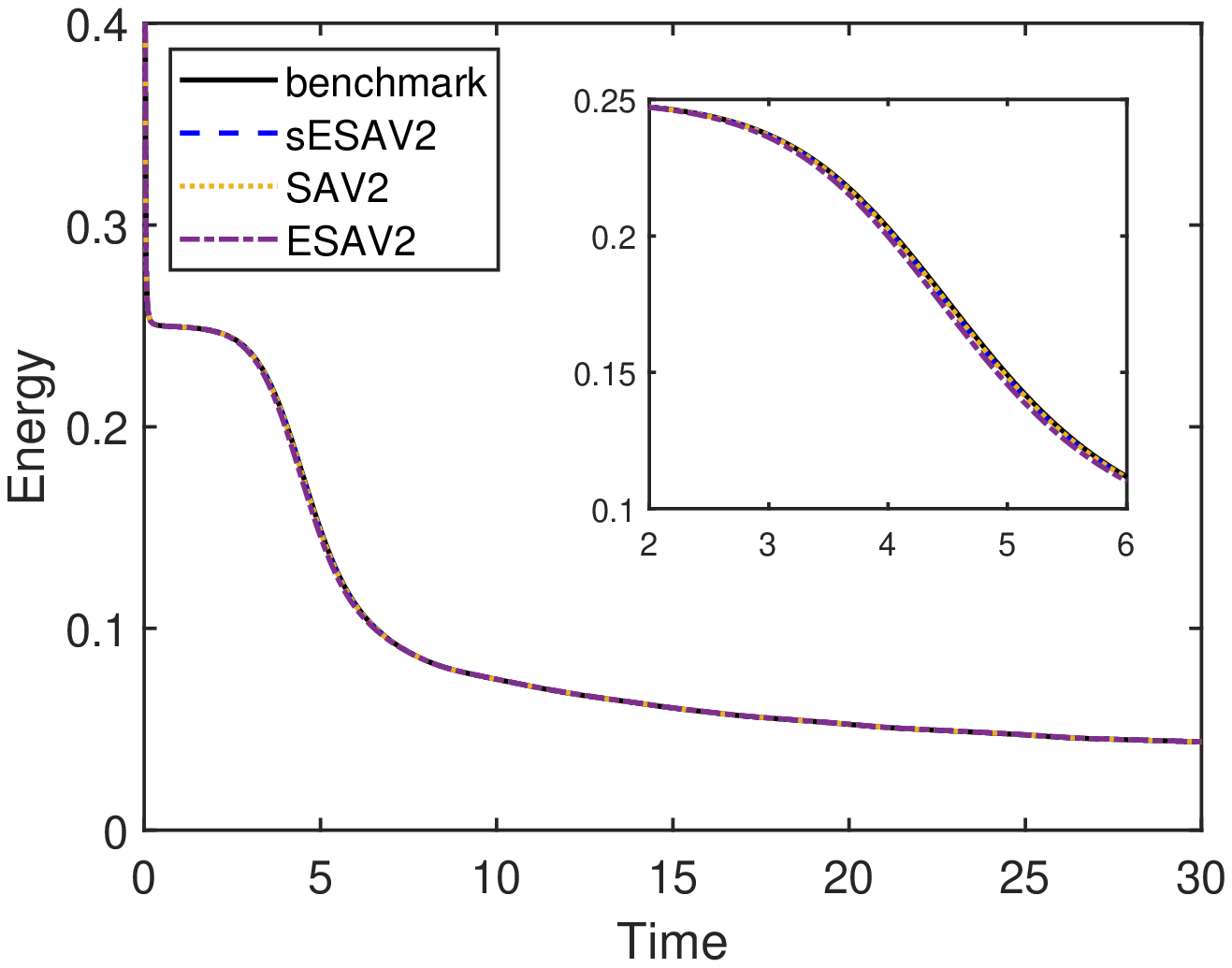}}
\centerline{
\includegraphics[width=0.43\textwidth]{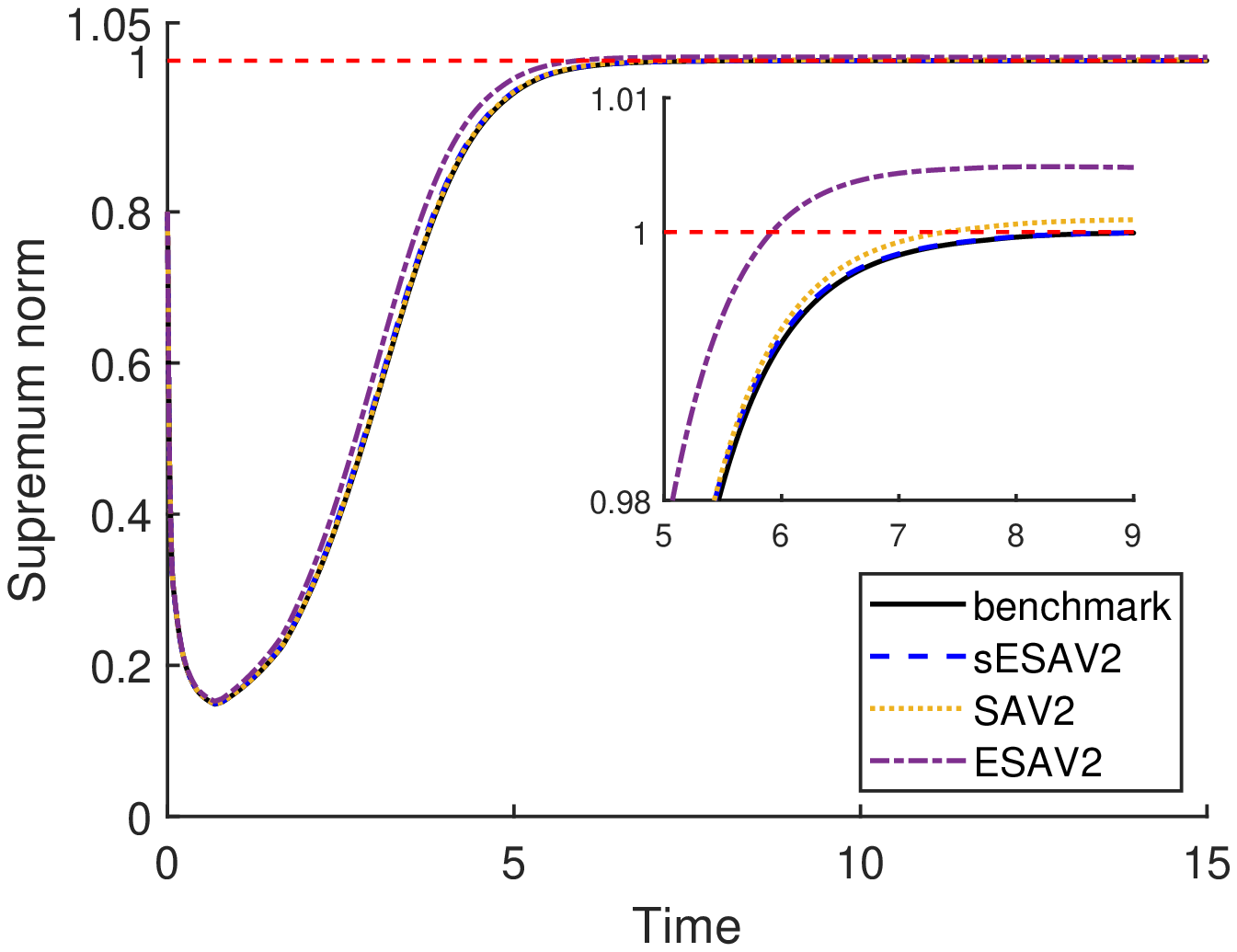}
\includegraphics[width=0.43\textwidth]{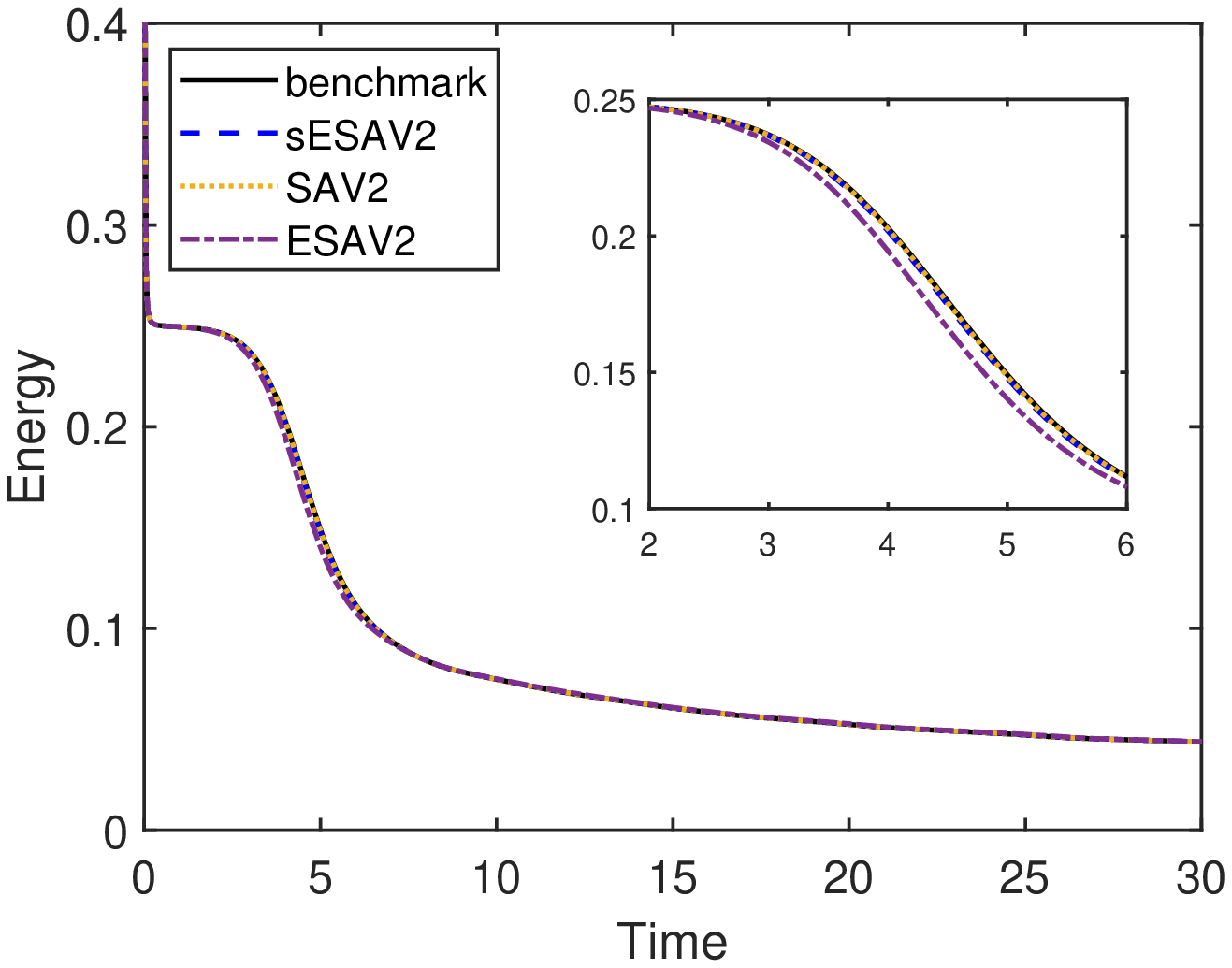}}
\caption{Evolutions of the supremum norms and the energies of  simulated solutions computed by the sESAV2, SAV2, and ESAV2 schemes with  $\dt=0.01$ and  $\kp=1$ (top row) or $\kp=2$ (bottom row) for the double-well potential case.}
\label{fig_comp_poly2}
\end{figure}

Next, we test the Flory--Huggins potential case  \eqref{f_fh} and correspondingly set $\kp=4.01$ and $\kp=8.02$ respectively.
Figures \ref{fig_comp_log1} and \ref{fig_comp_log2}
present the evolutions of the supremum norms and the energies of simulated solutions
obtained by the first- and second-order schemes, respectively.
Similar to the double-well potential case,
only the sESAV schemes preserve the MBP and the energy dissipation law as expected.
The SAV1, ESAV1, and ESAV2 schemes,
having the supremum norms beyond the theoretical bound $0.9575$, lead to inaccurate dynamic processes.
Especially, the ESAV1 solution with $\kp=8.02$ evolves beyond $1$,
which yields complex numbers due to the existence of the logarithmic term and gives the completely wrong dynamics.
For the SAV2 solutions, the dynamic processes look moderately correct according to the energy evolutions. Moreover,
it is interesting that the supremum norm goes larger than the desired bound for $\kp=8.02$
while it does not exceed for $\kp=4.01$, but both results are still a bit away from the expected value $0.9575$.

\begin{figure}[!ht]
\centerline{
\includegraphics[width=0.43\textwidth]{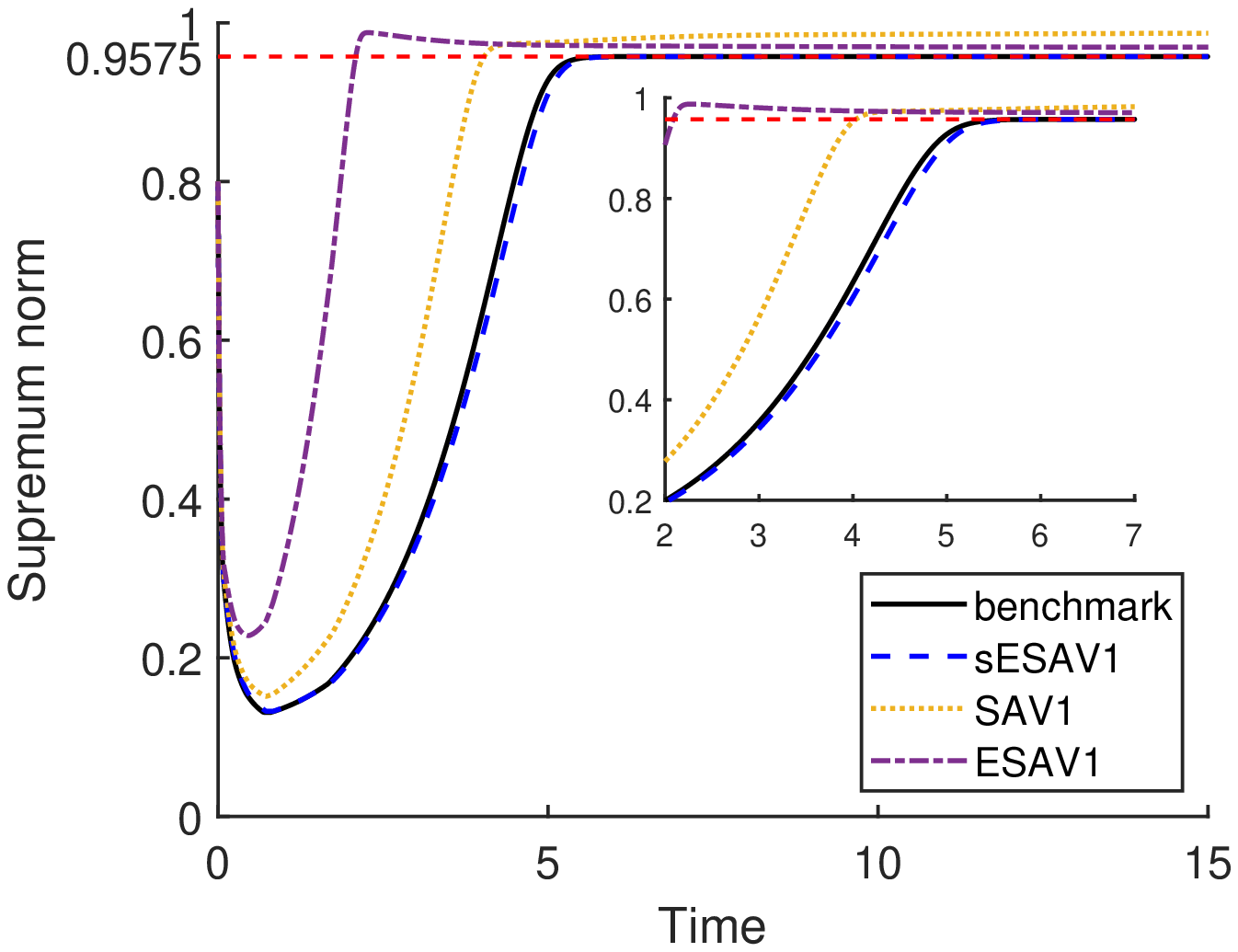}
\includegraphics[width=0.43\textwidth]{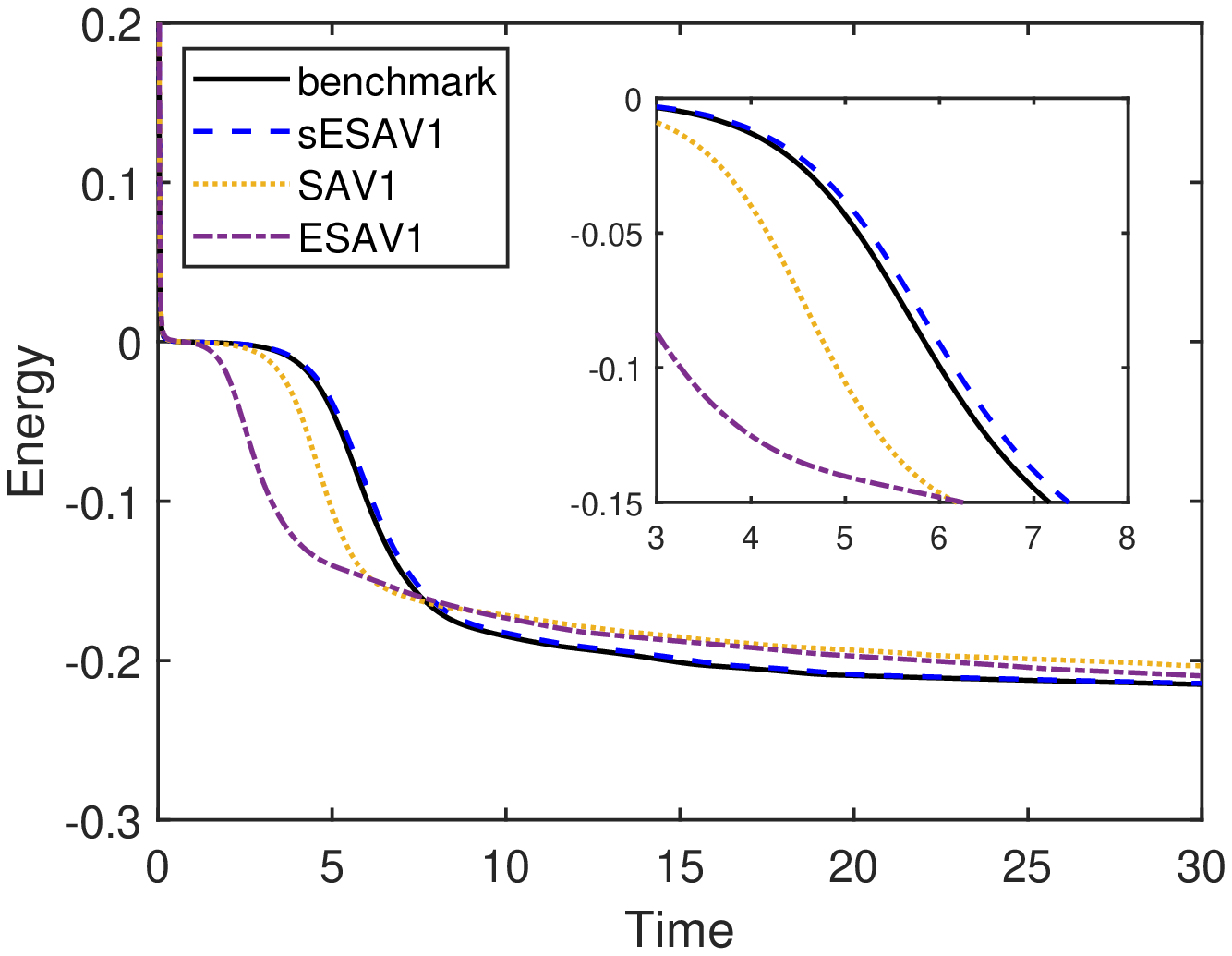}}
\centerline{
\includegraphics[width=0.43\textwidth]{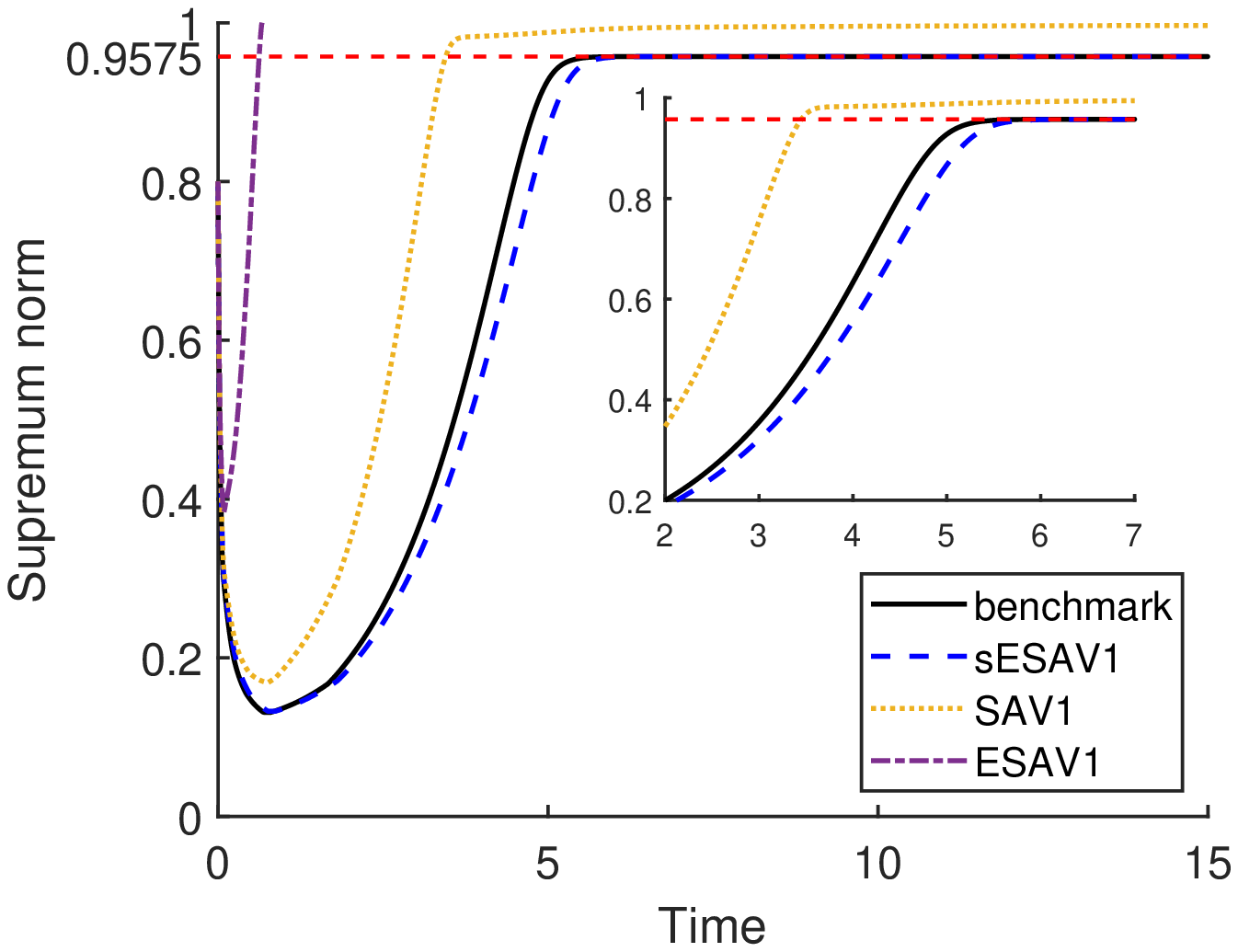}
\includegraphics[width=0.43\textwidth]{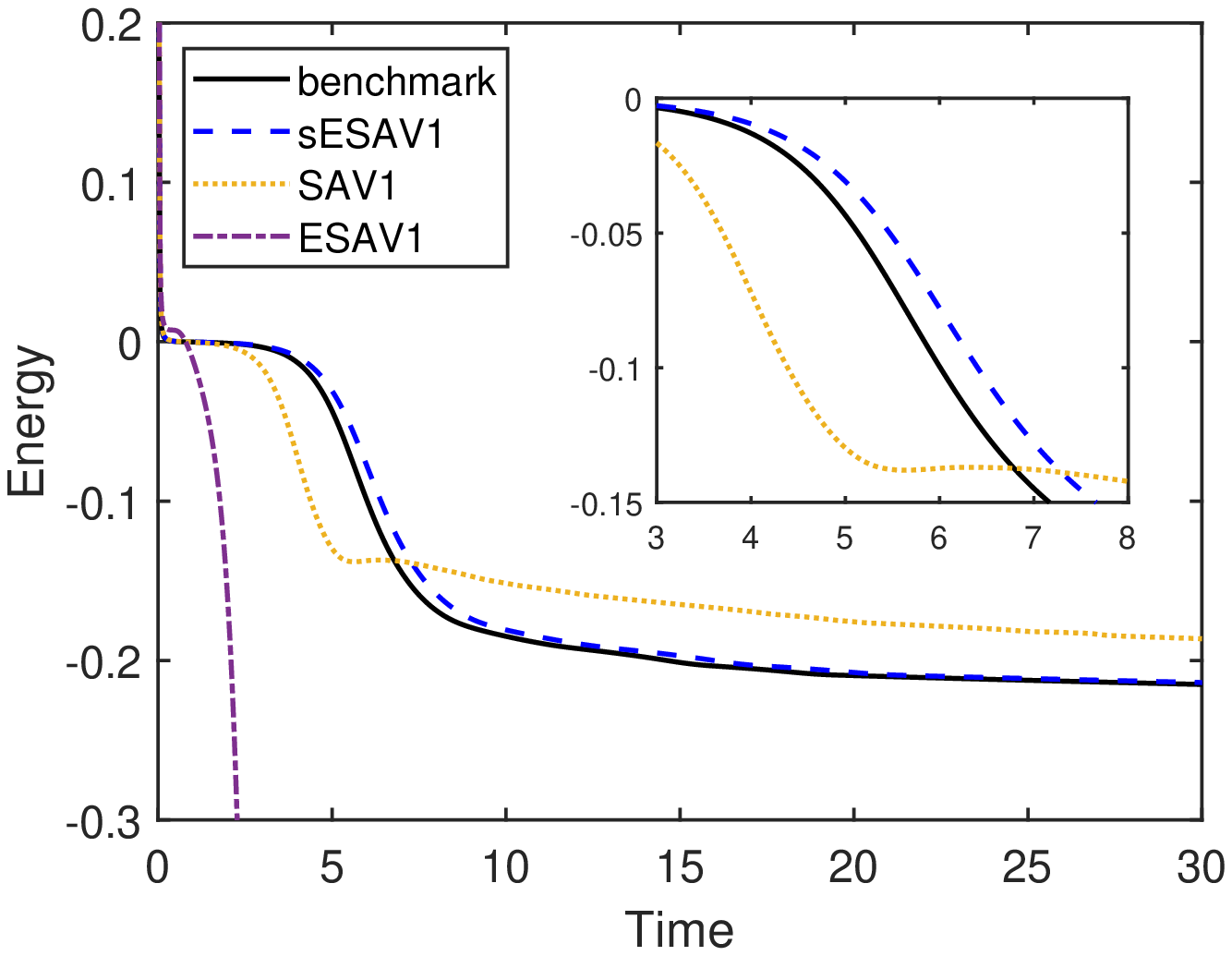}}
\caption{Evolutions of the supremum norms and the energies of  simulated solutions computed by the sESAV1, SAV1, and ESAV1 schemes
with  $\dt=0.01$ and  $\kp=4.01$ (top row) or $\kp=8.02$ (bottom row) for  the Flory--Huggins potential case.}
\label{fig_comp_log1}
\end{figure}

\begin{figure}[!ht]
\centerline{
\includegraphics[width=0.43\textwidth]{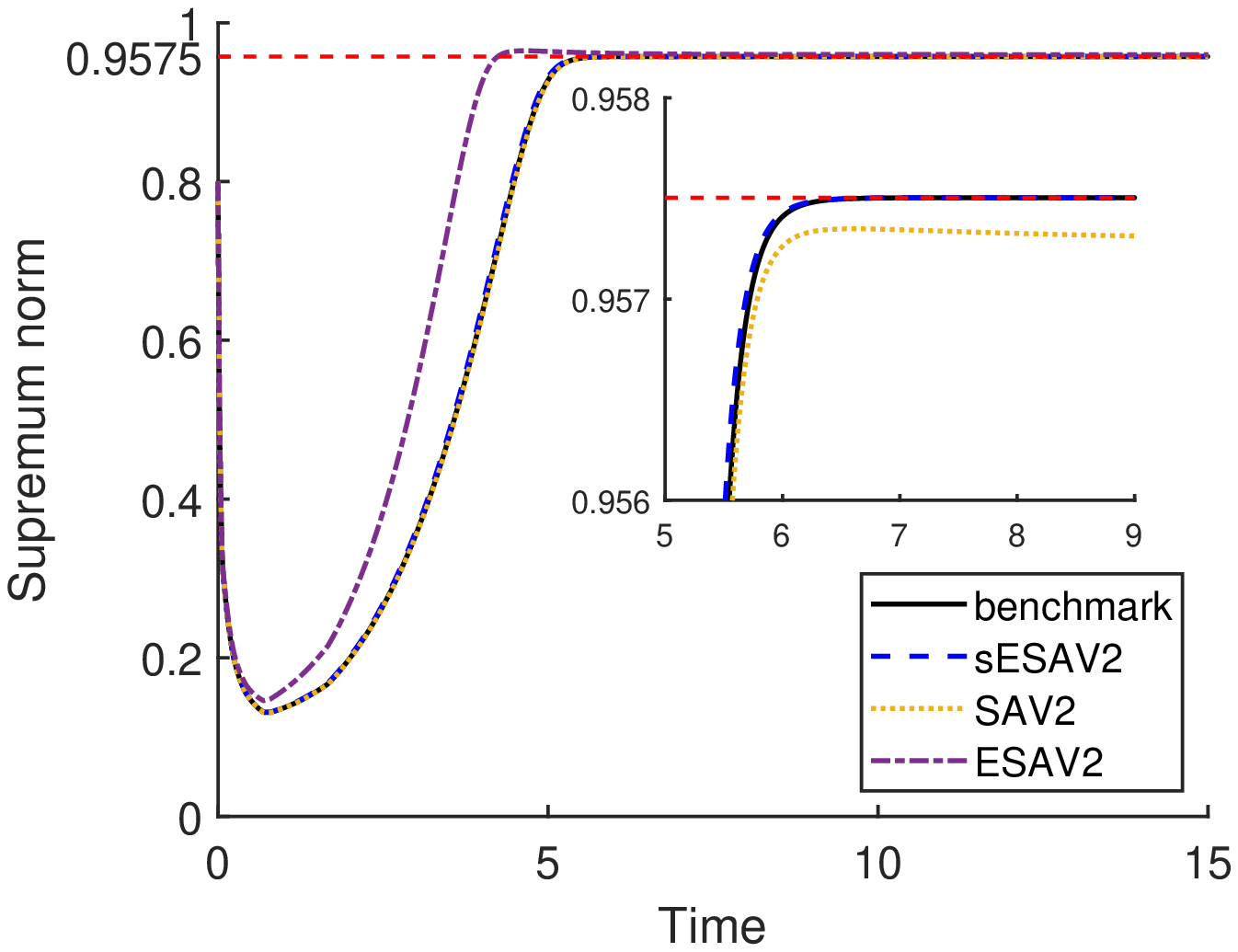}
\includegraphics[width=0.43\textwidth]{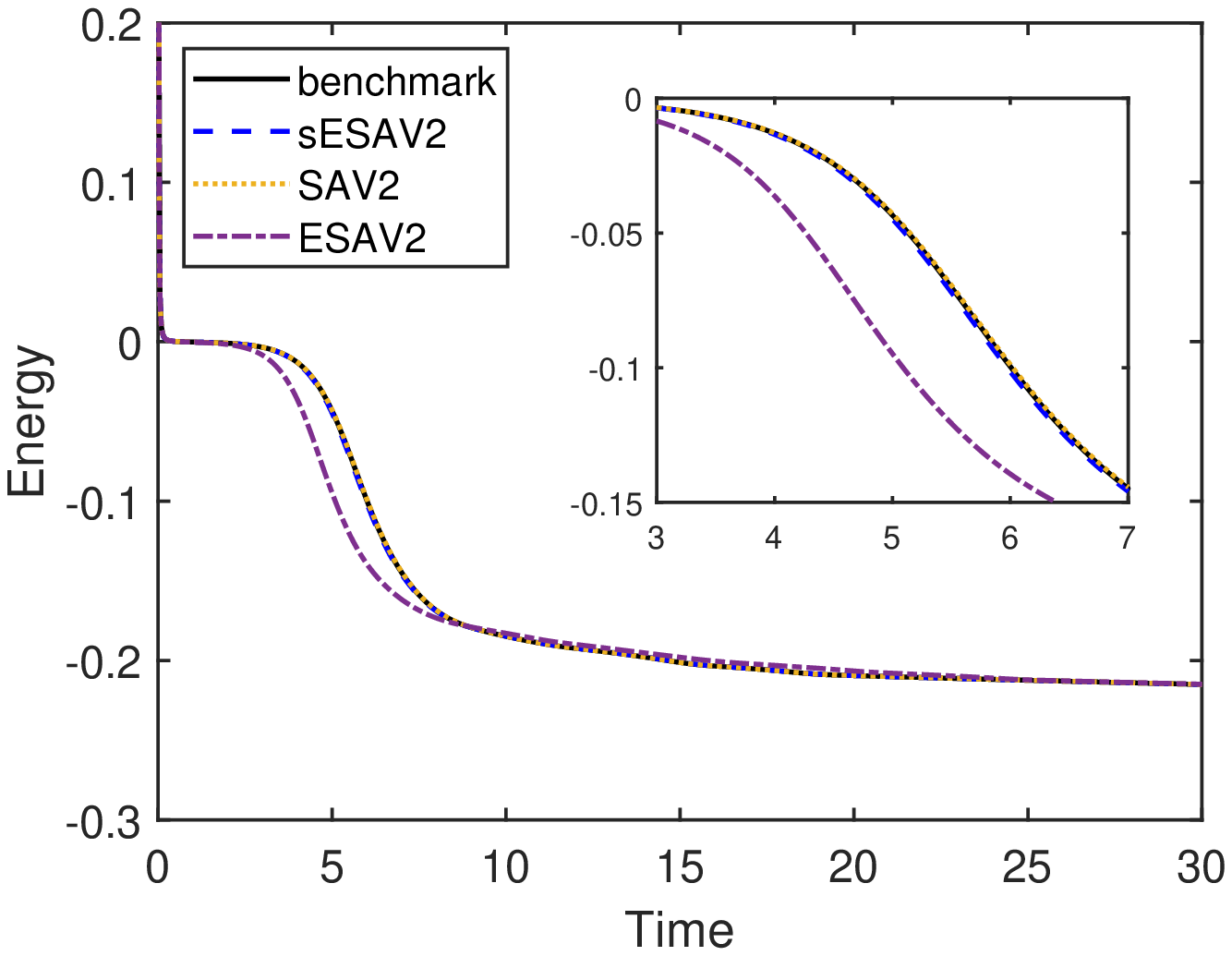}}
\centerline{
\includegraphics[width=0.43\textwidth]{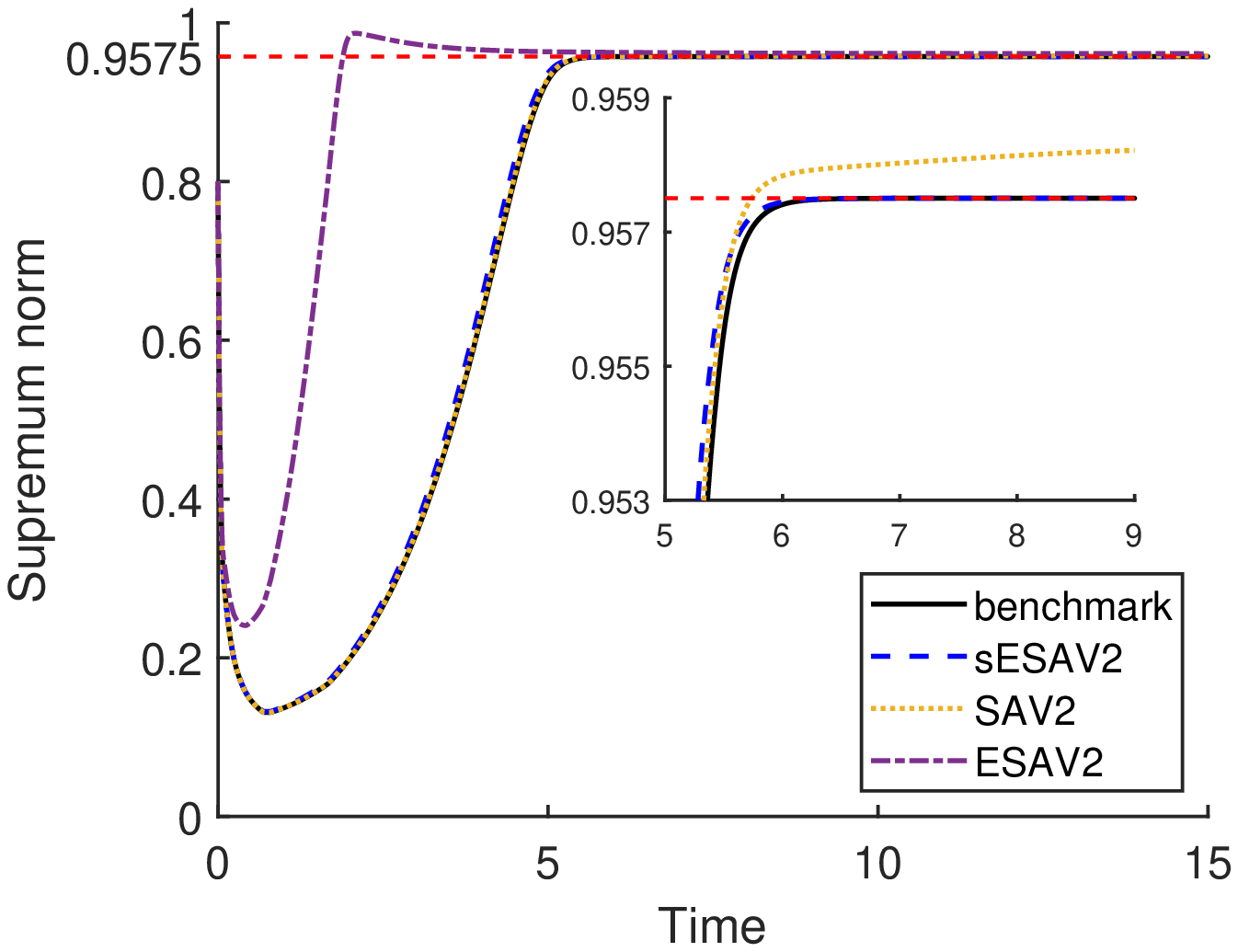}
\includegraphics[width=0.43\textwidth]{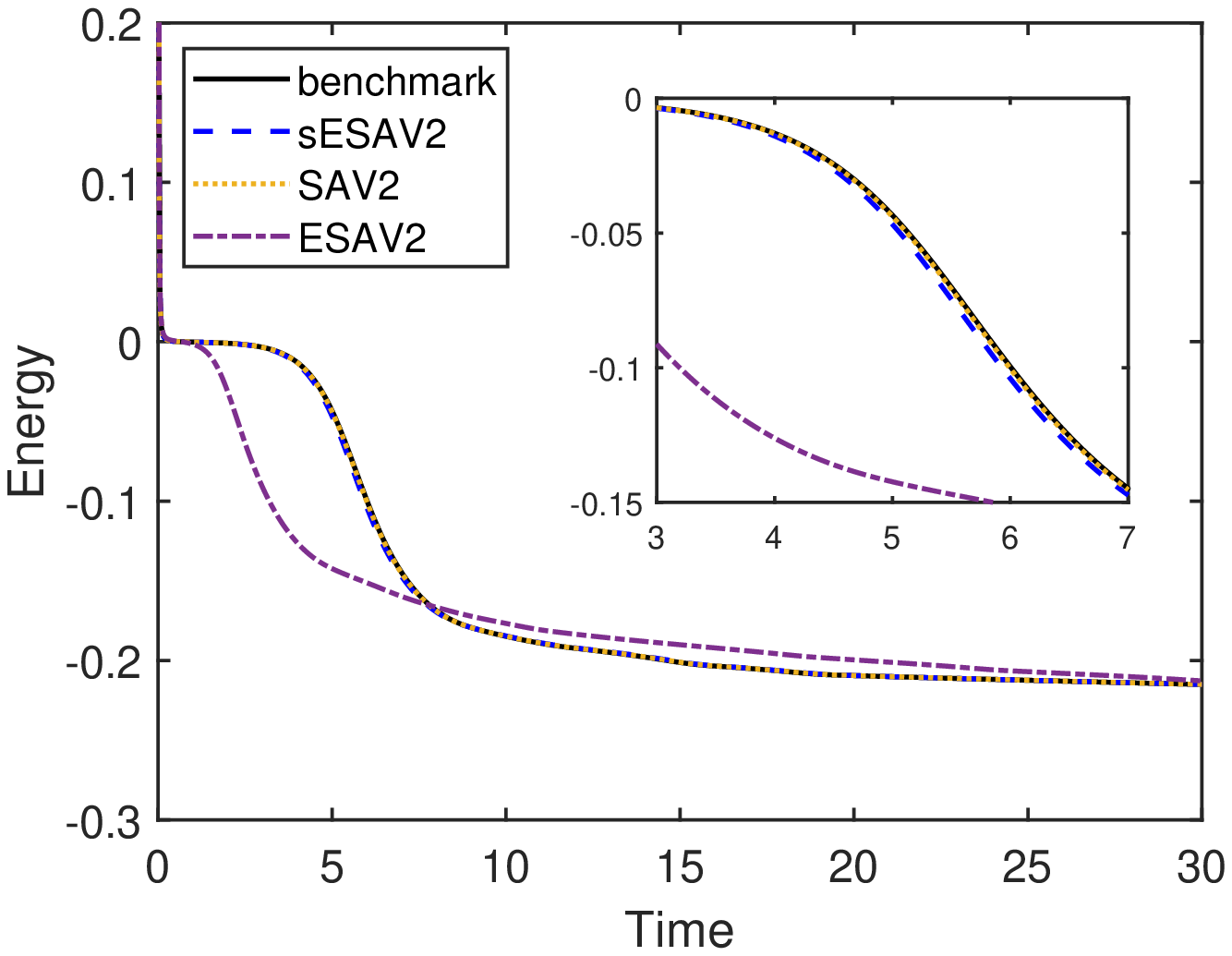}}
\caption{Evolutions of the supremum norms and the energies of  simulated solutions computed by the sESAV2, SAV2, and ESAV2 schemes
with  $\dt=0.01$ and   $\kp=4.01$ (top row) or $\kp=8.02$ (bottom row) for  the Flory--Huggins potential case.}
\label{fig_comp_log2}
\end{figure}

\subsection{Long-time coarsening dynamics simulations}

Now we study the coarsening dynamics driven by the Allen--Cahn equation \eqref{AllenCahn} with $\eps=0.01$.
The spatial mesh size is $h=1/512$ and the initial state is given by random numbers between $-0.8$ and $0.8$.
We adopt the sESAV2 scheme with $\dt=0.01$ to simulate the  long-time coarsening process.
By the comparisons shown above, we know that $\dt=0.01$ is sufficient to provide accurate numerical results.
The steady state of the coarsening dynamics is a constant state $u\equiv\beta$ or $u\equiv-\beta$.
When the absolute difference between the energies at the two consecutive moments is smaller than the tolerance value $10^{-8}$,
we regard the dynamics as reaching its steady state.

For the double-well potential case $f(u)$, we set $\kp=2$ and
the phase structures captured at some moments
are presented in Figure \ref{fig_coarsen_poly1}, and the constant steady state $u\equiv-1$ is reached at around $t=604$.
The left picture given in Figure \ref{fig_coarsen_poly2} implies
the preservation of the MBP during the whole phase transition process.
The  energy evolution  is plotted in the right graph of Figure \ref{fig_coarsen_poly2},
which states the energy dissipation of the process. For the Flory--Huggins potential case, we set $\kp=8.02$ and
the simulated results are shown in Figures \ref{fig_coarsen_log1} and \ref{fig_coarsen_log2}. We observe that
the steady state is reached at around $t=602$ and
the whole process of phase separation is similar to that of the double-well potential case. Those results  are almost identical to
those produced using the IFRK4 scheme in \cite{JuLiQiYa21}.

\begin{figure}[!ht]
\centerline{
\includegraphics[width=0.34\textwidth]{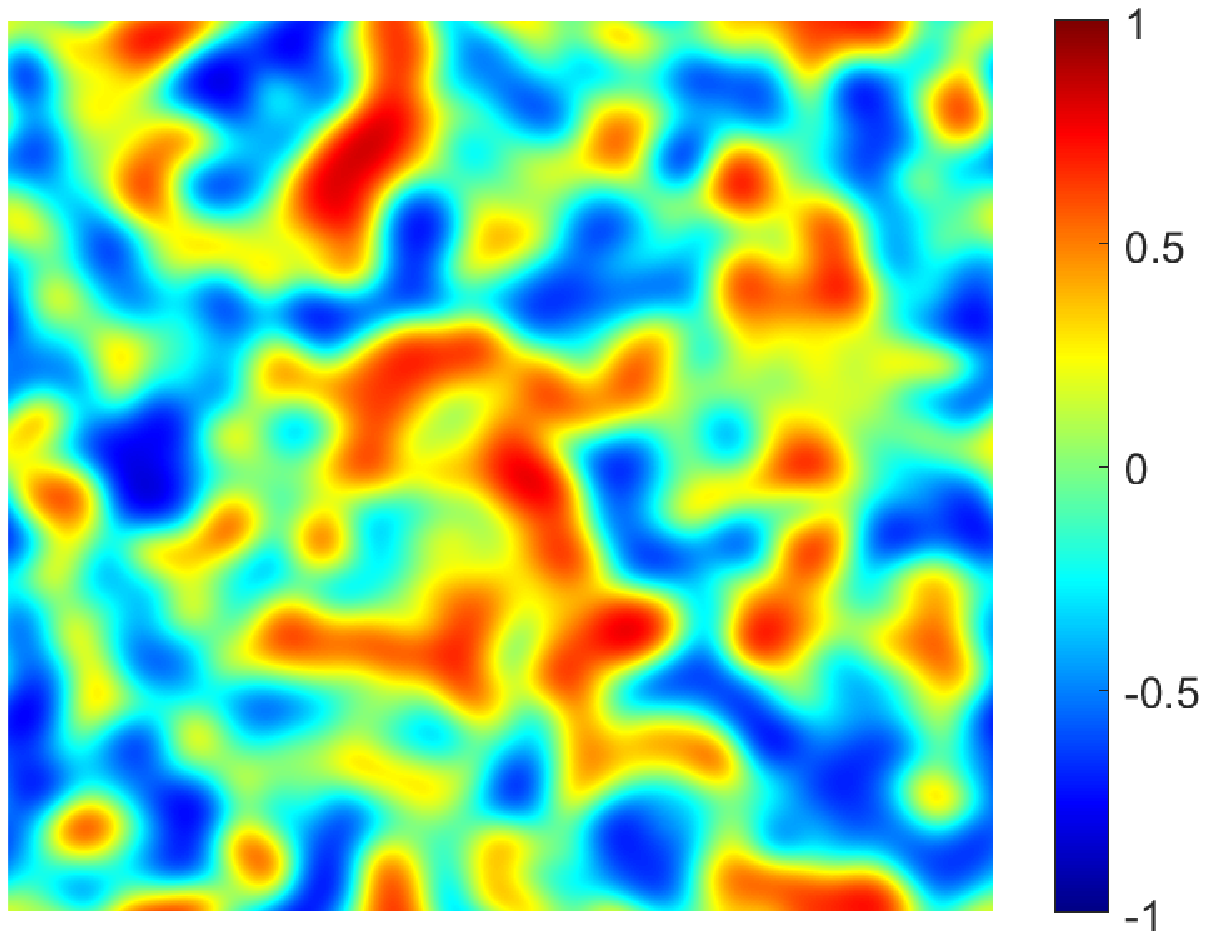}\hspace{-0.45cm}
\includegraphics[width=0.34\textwidth]{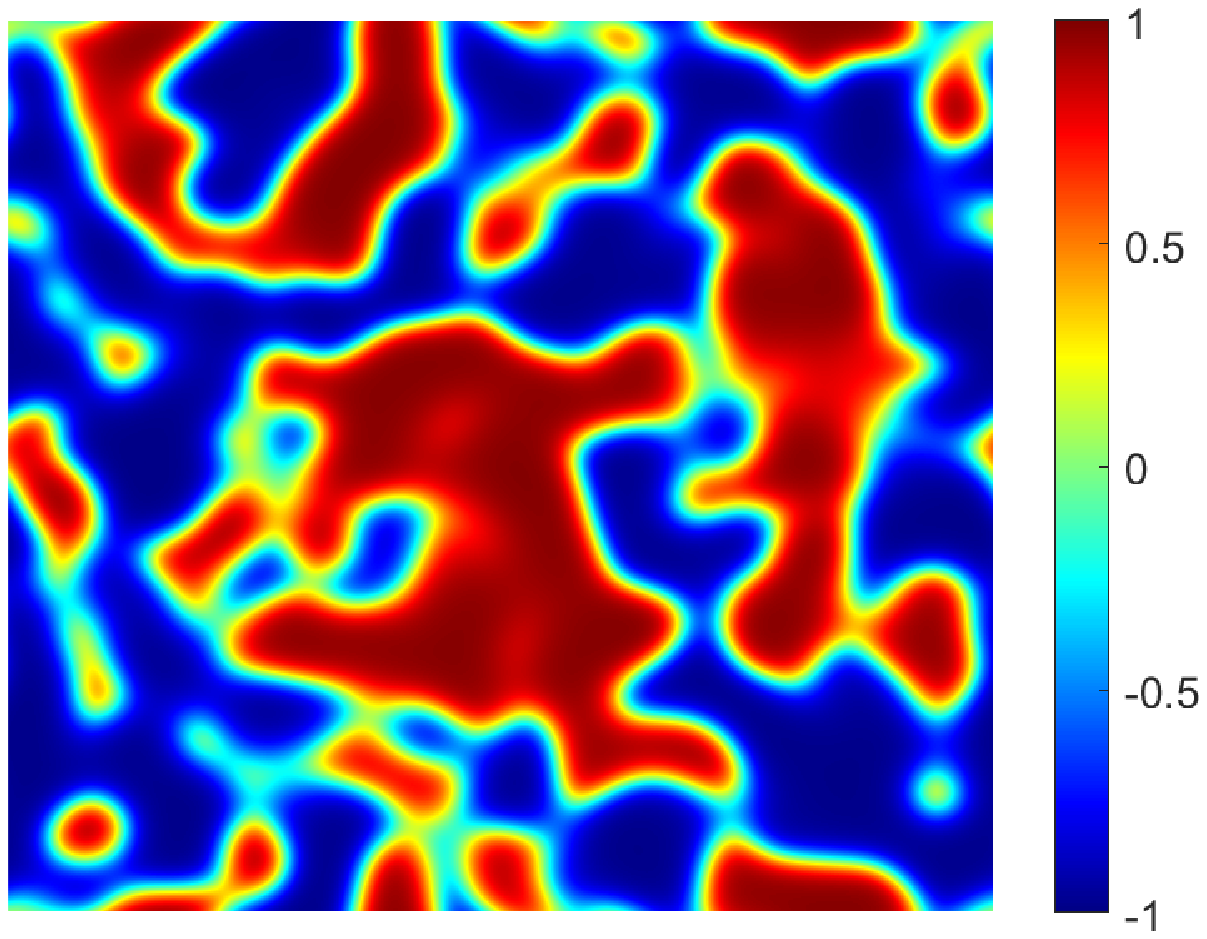}\hspace{-0.45cm}
\includegraphics[width=0.34\textwidth]{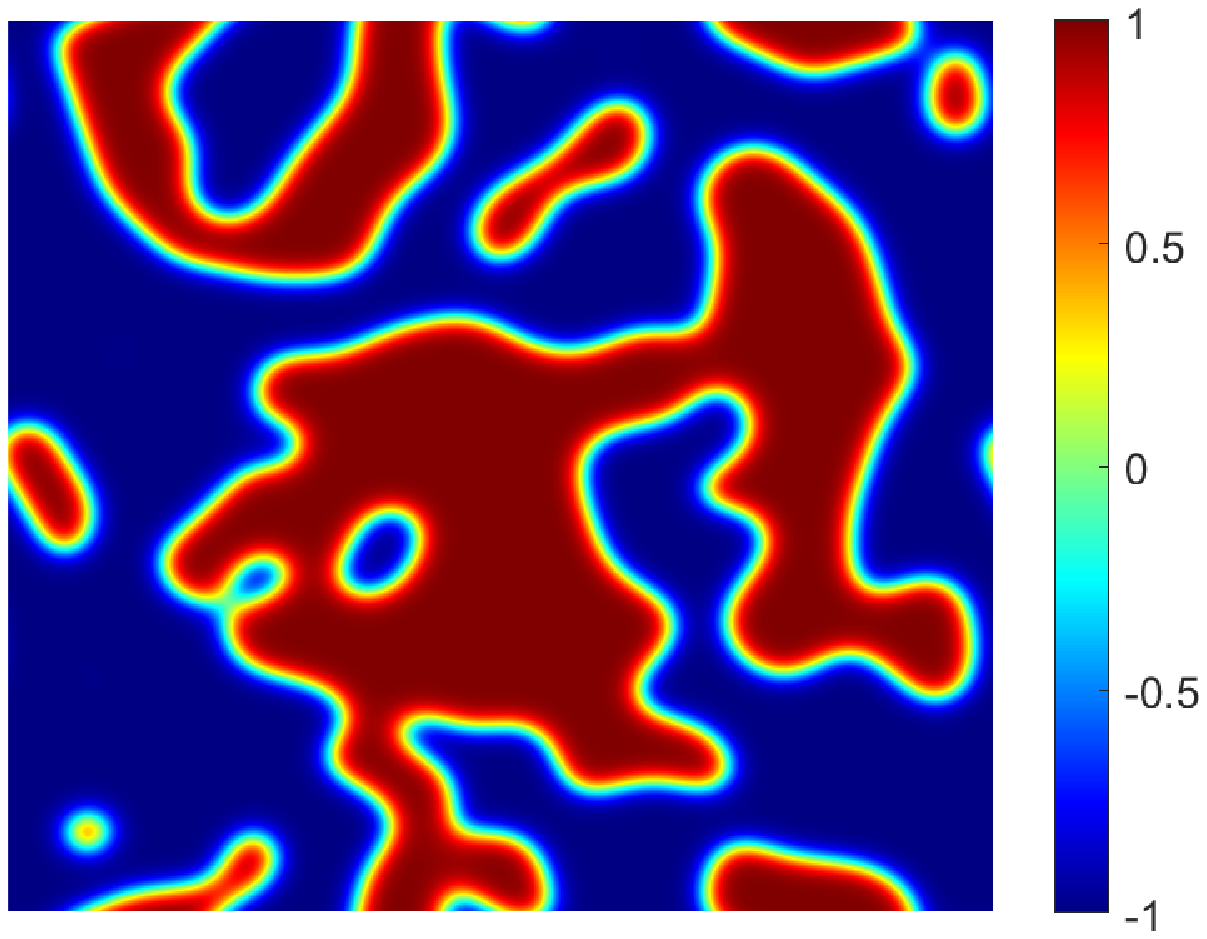}}\vspace{-0.2cm}
\centerline{
\includegraphics[width=0.34\textwidth]{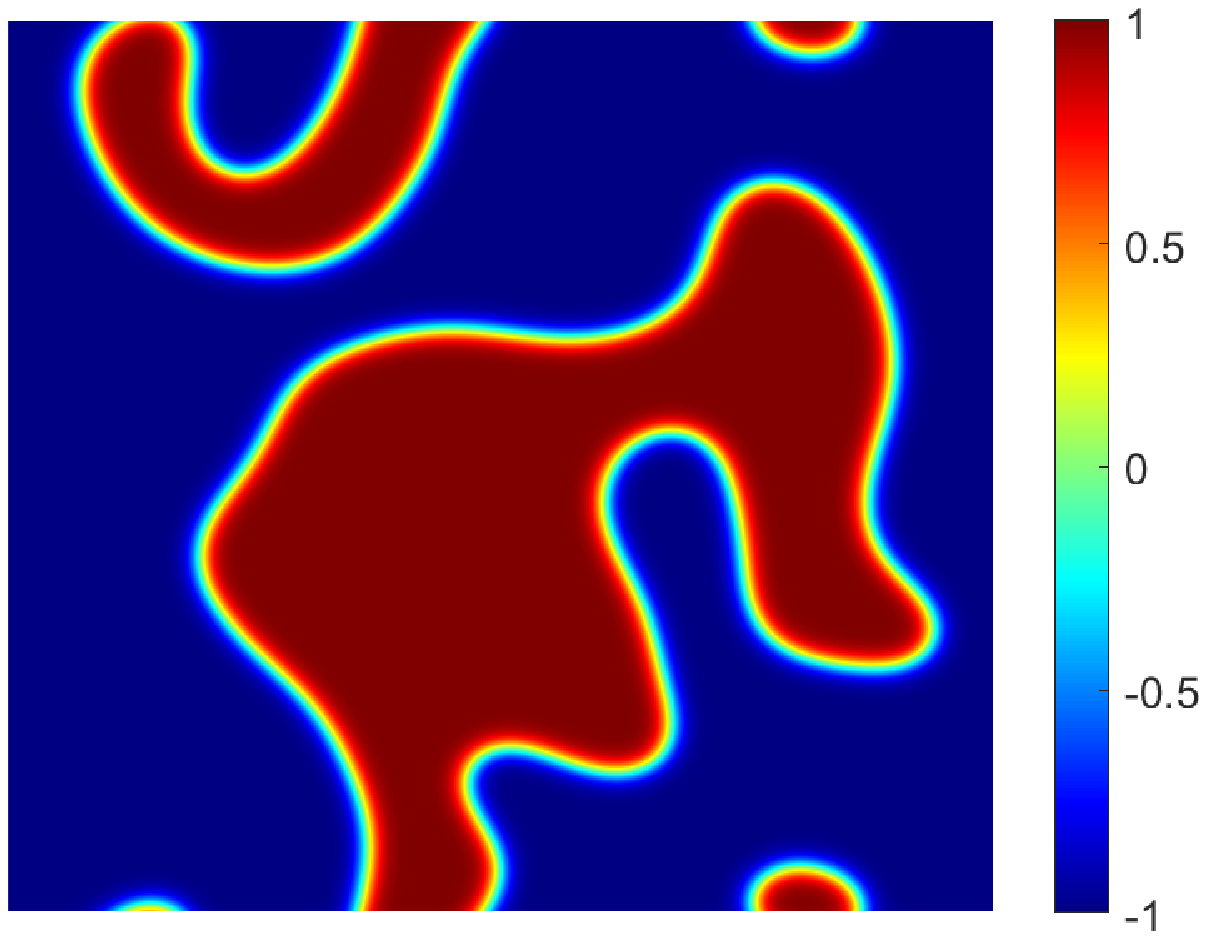}\hspace{-0.45cm}
\includegraphics[width=0.34\textwidth]{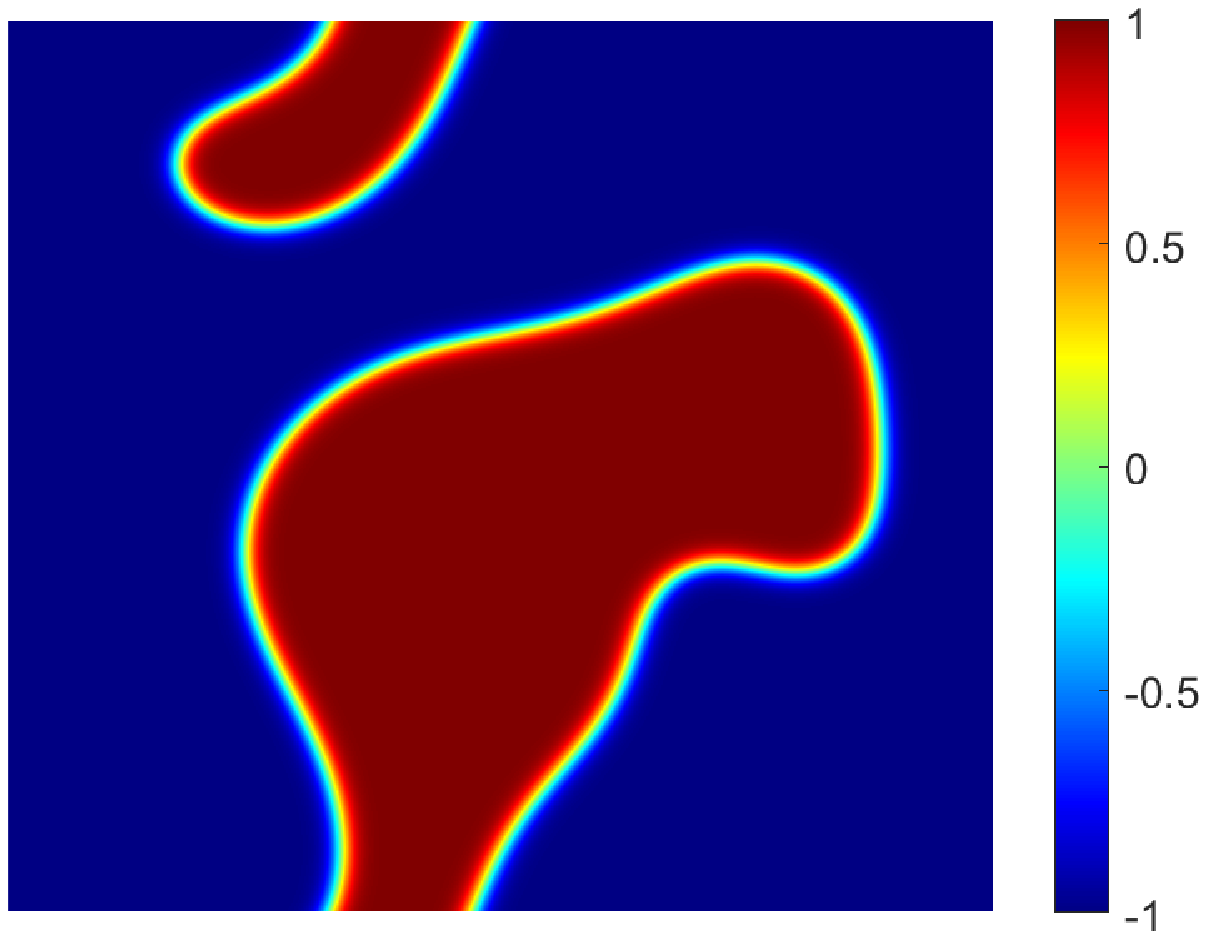}\hspace{-0.45cm}
\includegraphics[width=0.34\textwidth]{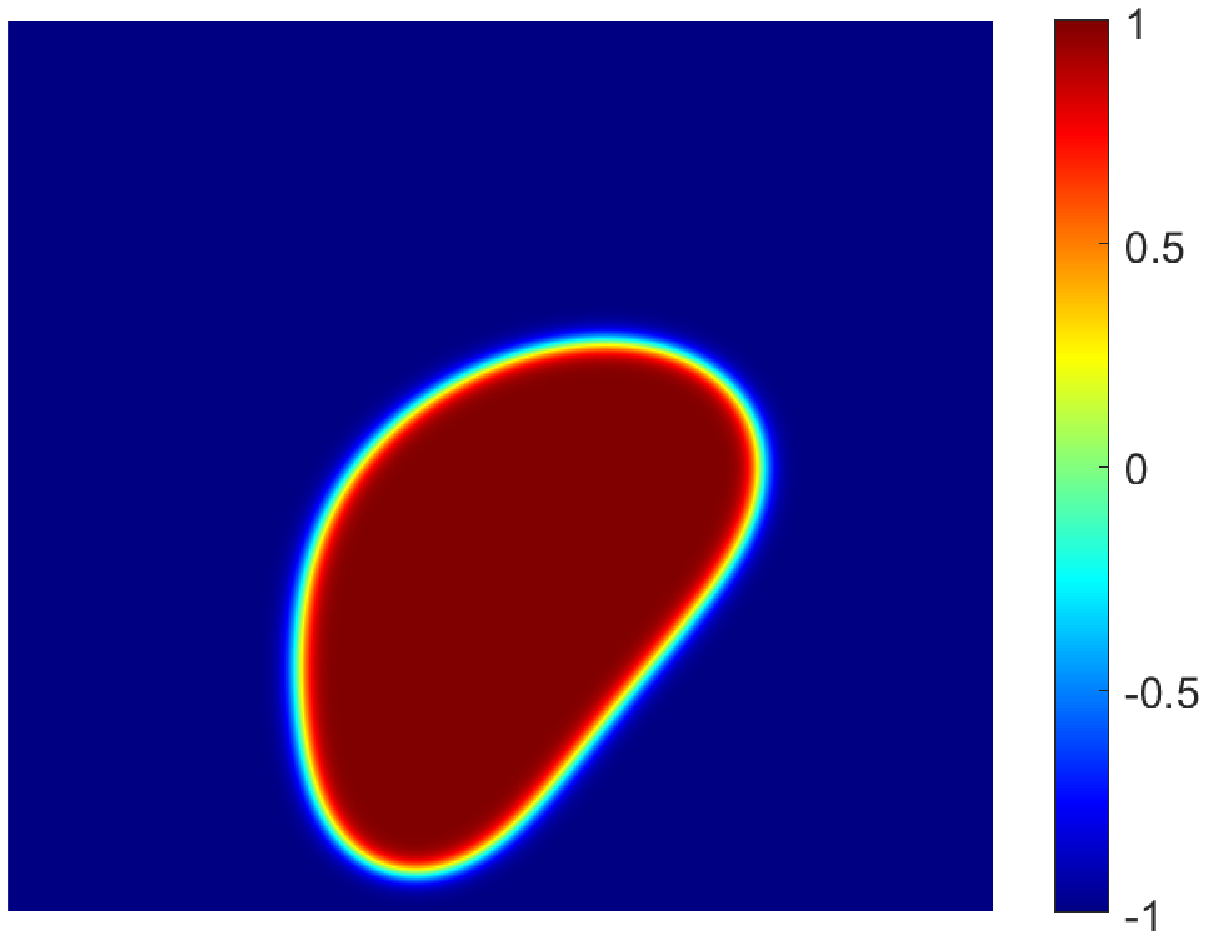}}\vspace{-0.2cm}
\caption{Simulated phase structures at $t=4$, $6$, $10$, $30$, $100$, and $300$, respectively
(left to right and top to bottom) by the sESAV2 scheme with  $\dt=0.01$ and $\kappa=2$  for the coarsening dynamics of the double-well potential case.}
\label{fig_coarsen_poly1}
\end{figure}

\begin{figure}[!ht]
\centerline{
\includegraphics[width=0.43\textwidth]{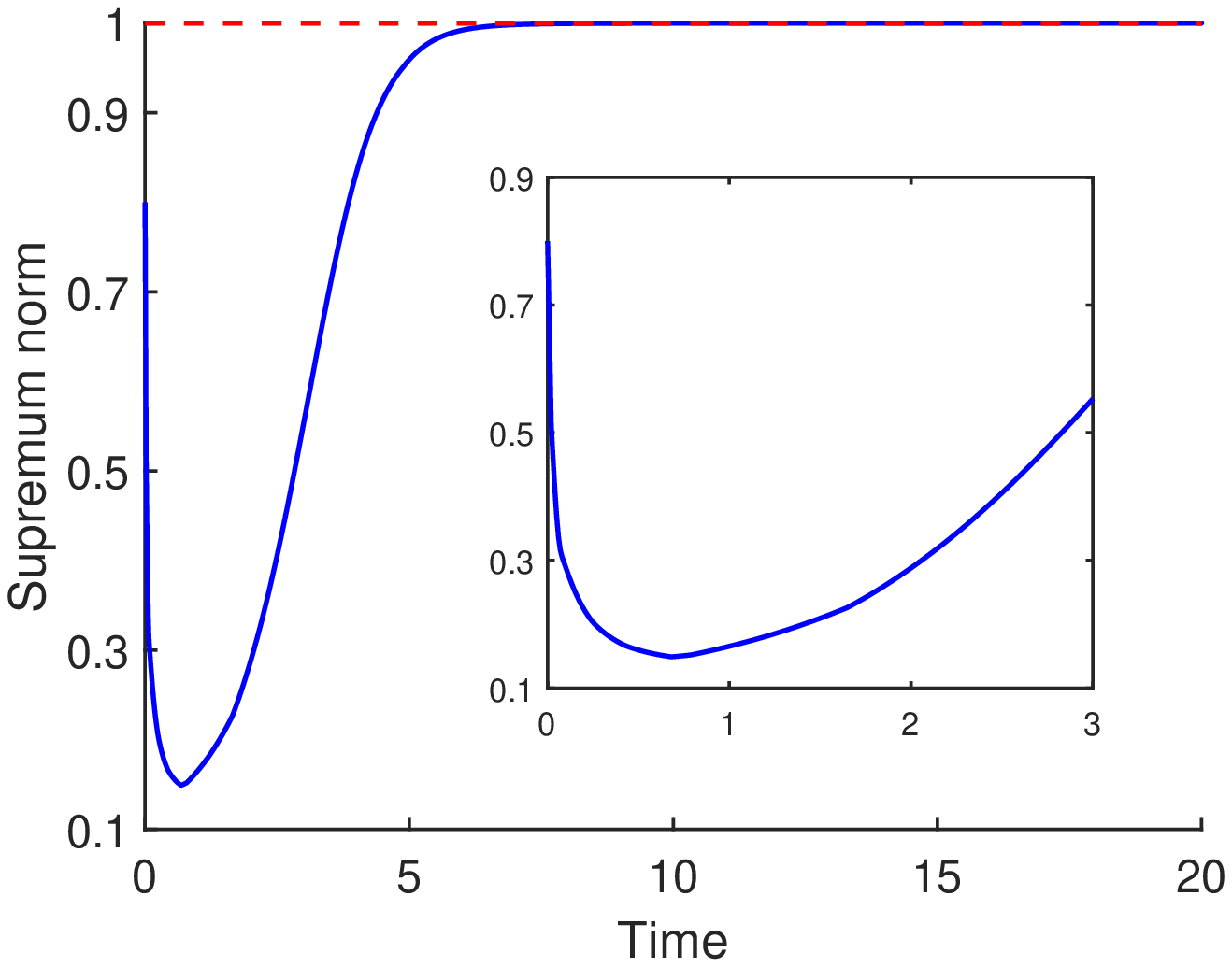}
\includegraphics[width=0.43\textwidth]{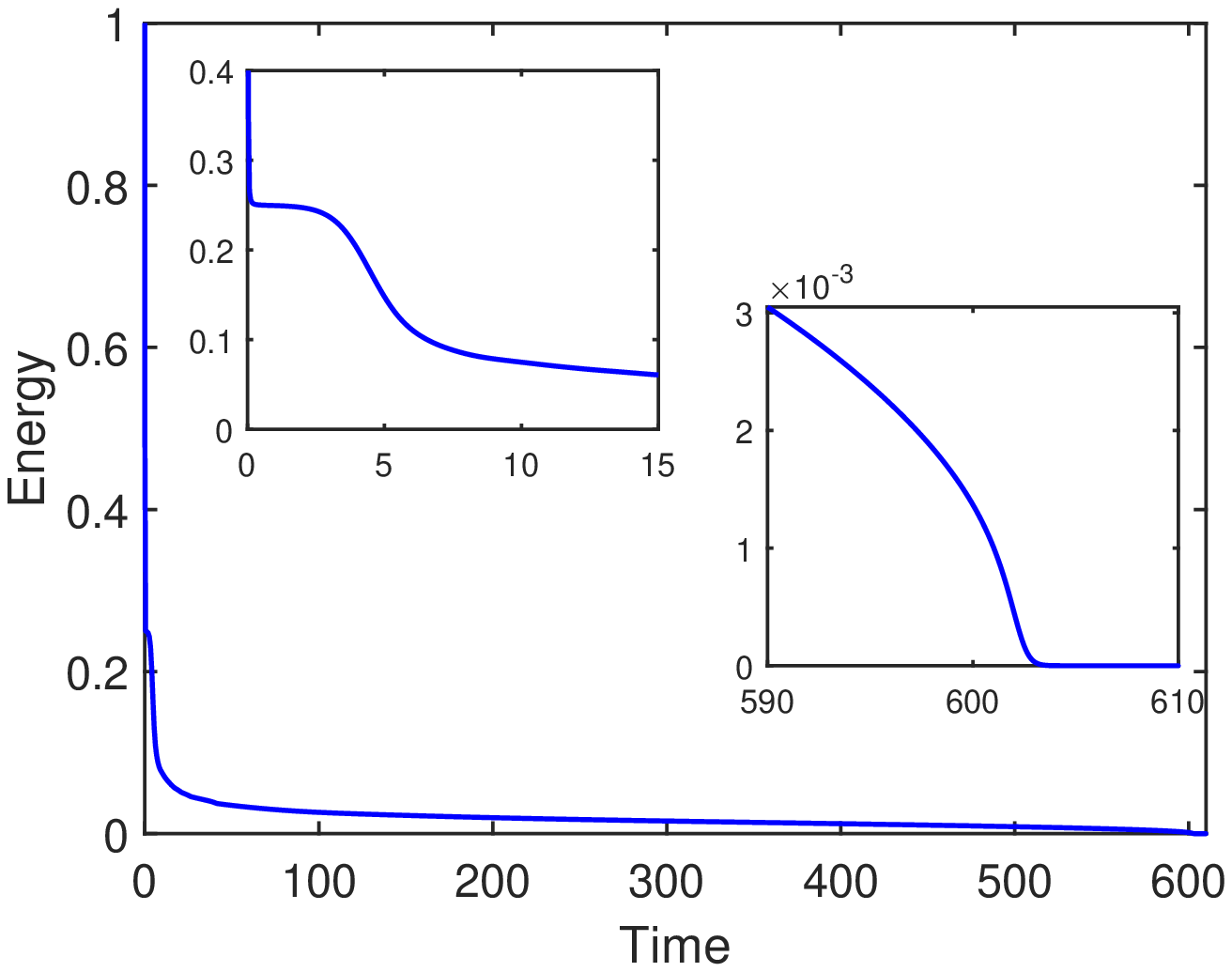}}
\caption{Evolutions of the supremum norm (left) and the energy (right)
for the coarsening dynamics of the double-well potential case.}
\label{fig_coarsen_poly2}
\end{figure}

\begin{figure}[!ht]
\centerline{
\includegraphics[width=0.34\textwidth]{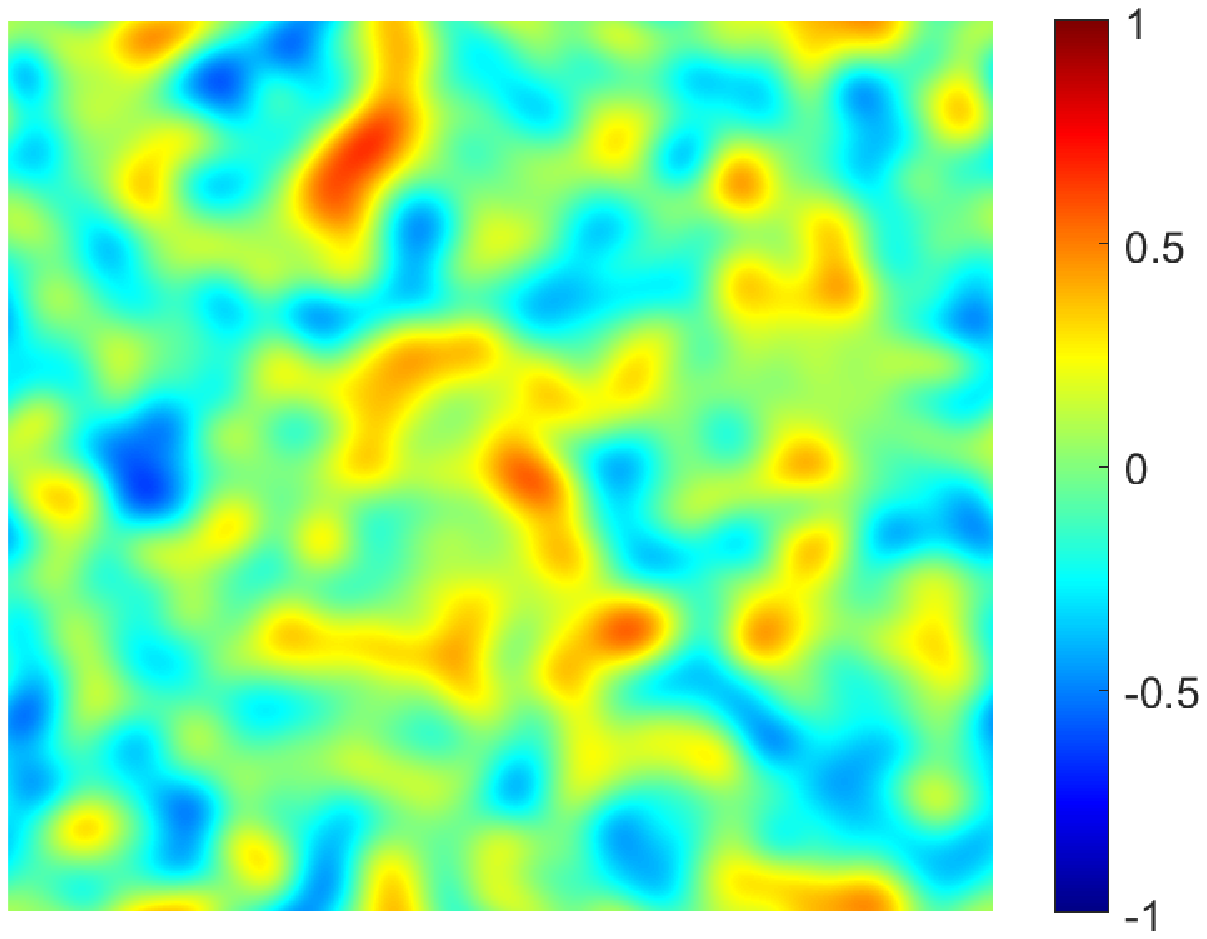}\hspace{-0.45cm}
\includegraphics[width=0.34\textwidth]{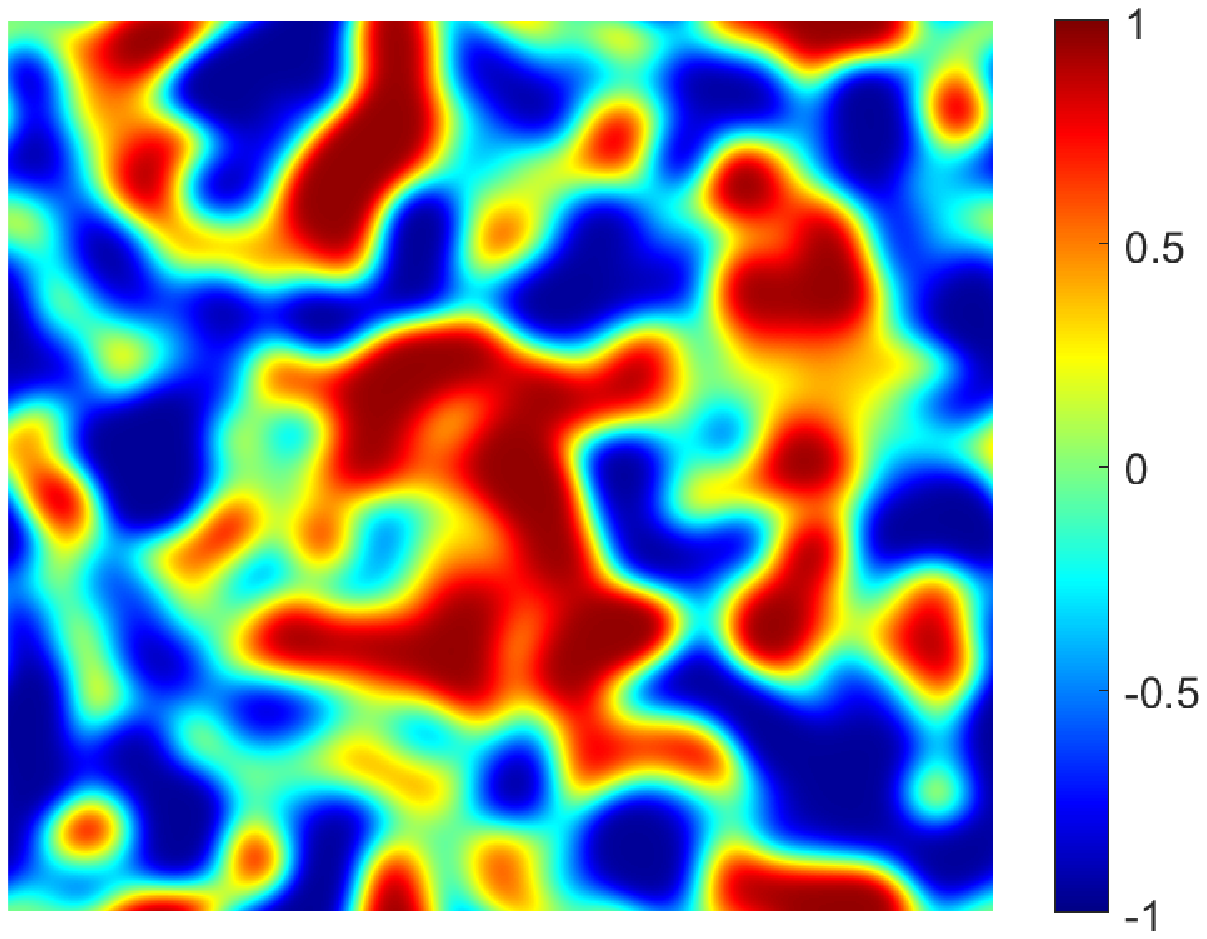}\hspace{-0.45cm}
\includegraphics[width=0.34\textwidth]{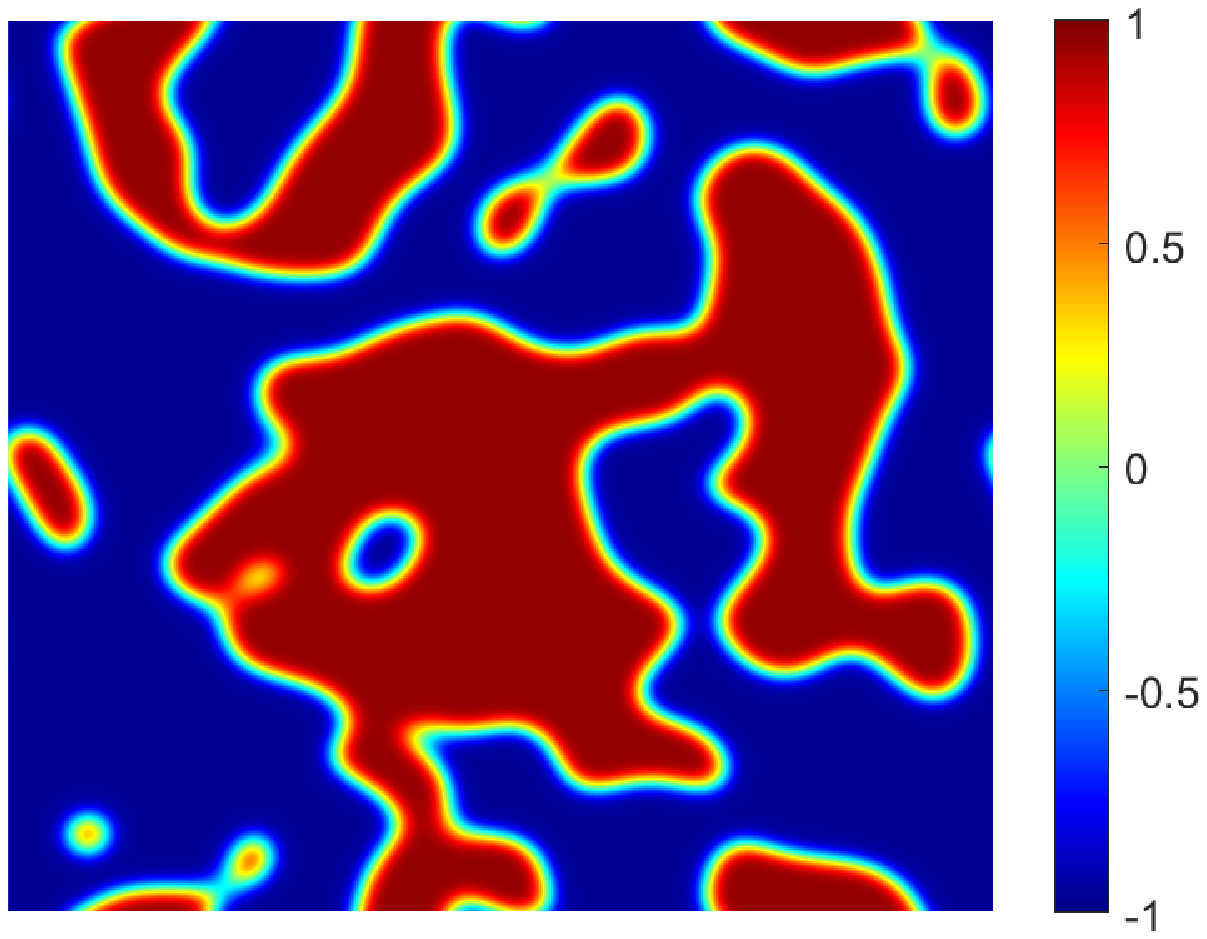}}\vspace{-0.2cm}
\centerline{
\includegraphics[width=0.34\textwidth]{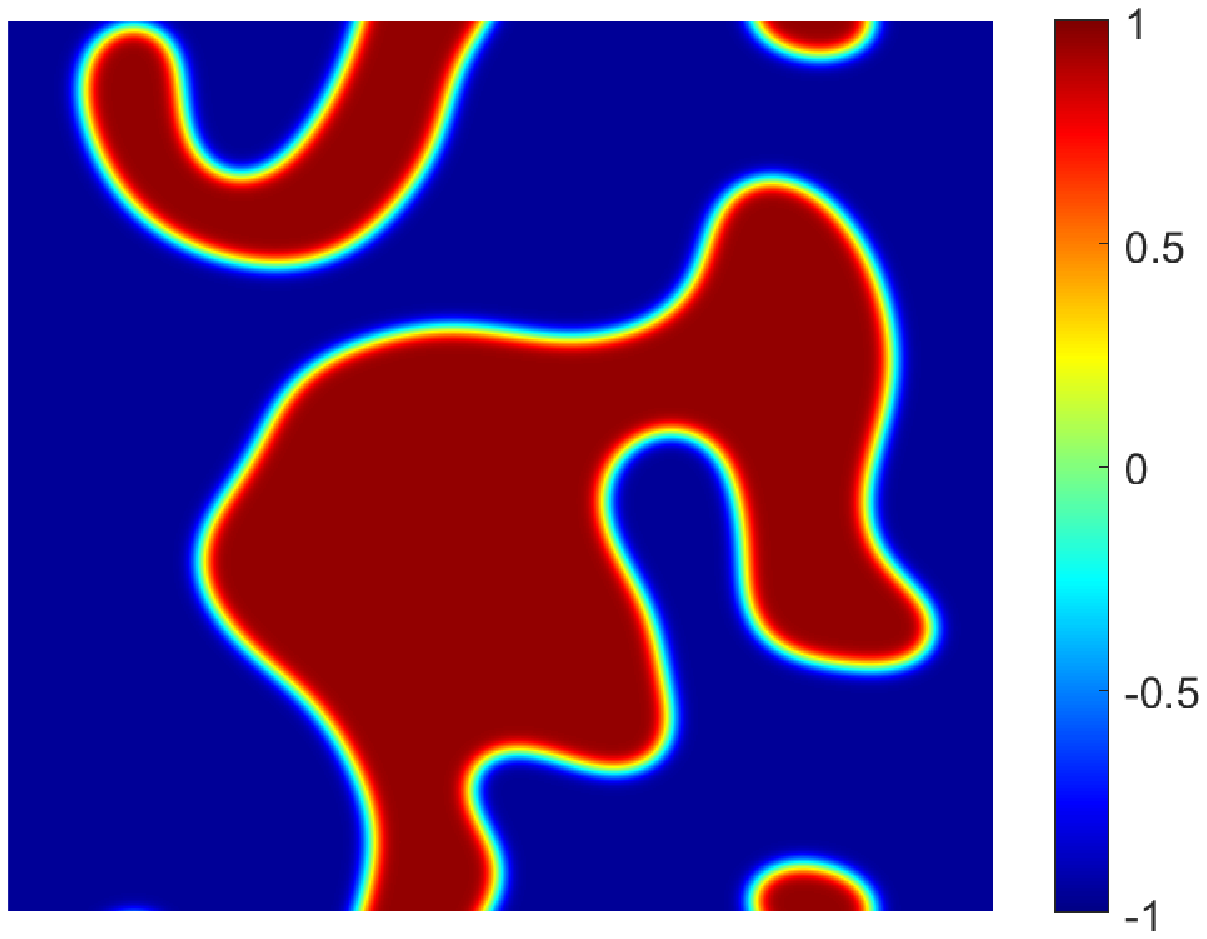}\hspace{-0.45cm}
\includegraphics[width=0.34\textwidth]{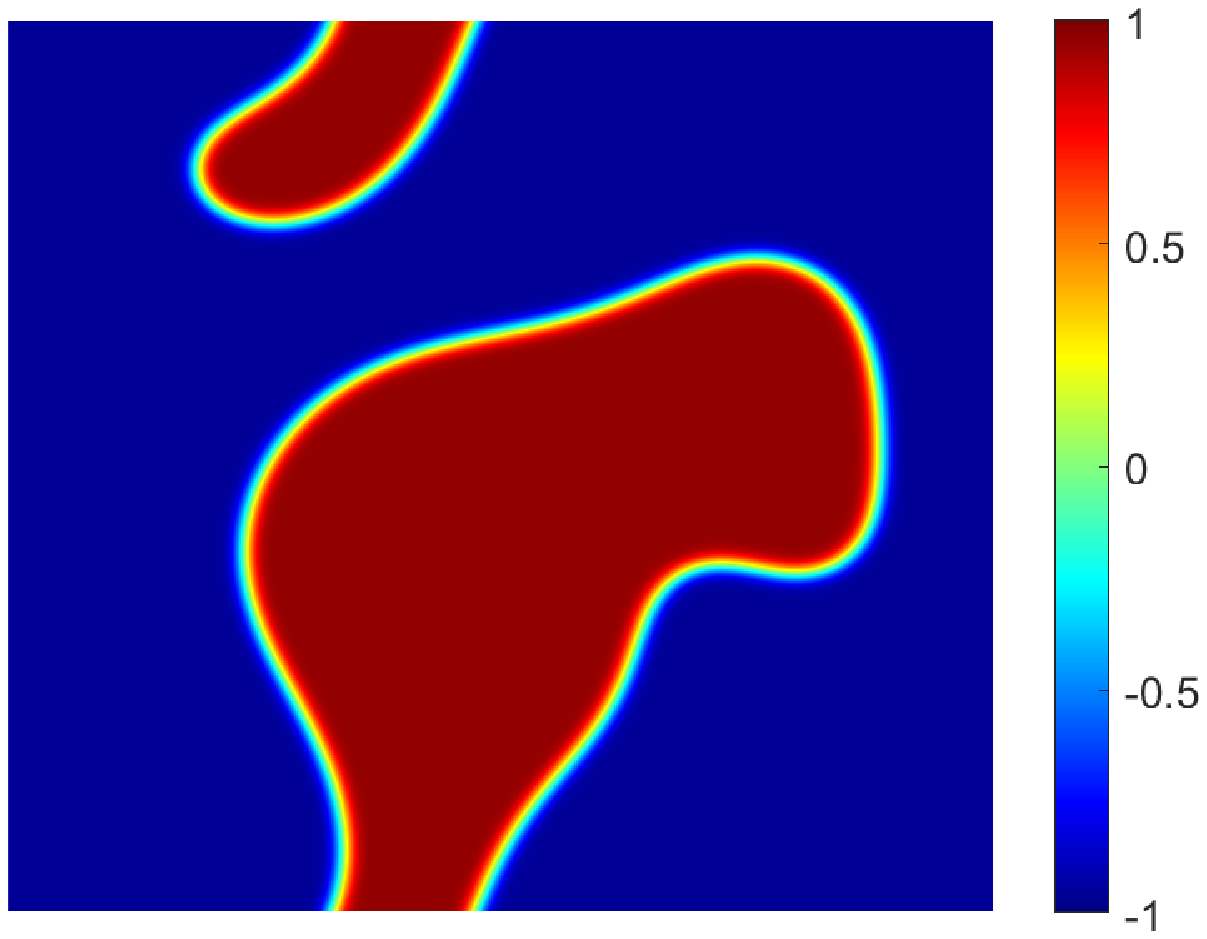}\hspace{-0.45cm}
\includegraphics[width=0.34\textwidth]{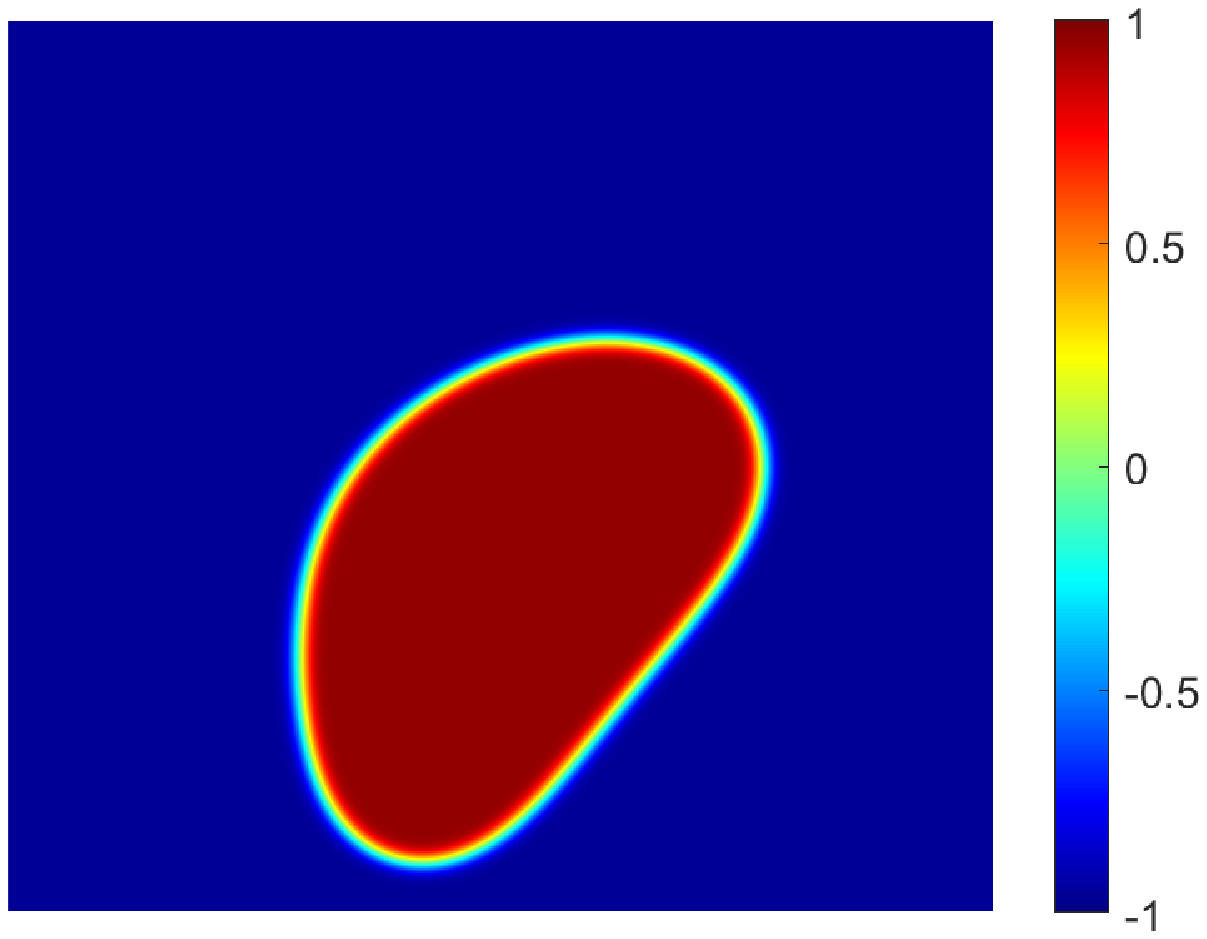}}\vspace{-0.2cm}
\caption{Simulated phase structures at $t=4$, $6$, $10$, $30$, $100$, and $300$, respectively
(left to right and top to bottom) by the sESAV2 scheme with  $\dt=0.01$ and $\kappa=8.02$  for the coarsening dynamics of the Flory--Huggins potential case.}
\label{fig_coarsen_log1}
\end{figure}

\begin{figure}[!ht]
\centerline{
\includegraphics[width=0.43\textwidth]{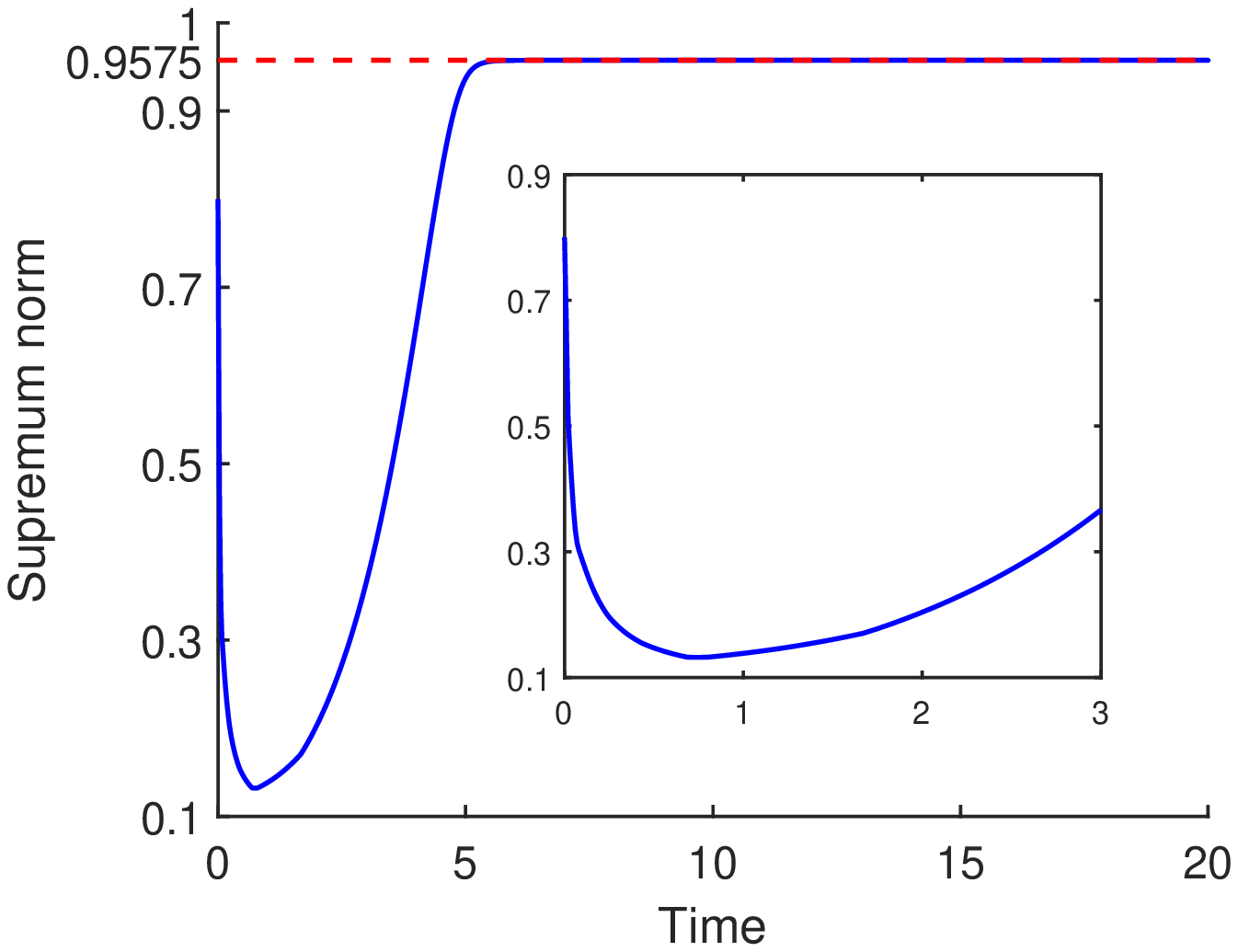}
\includegraphics[width=0.43\textwidth]{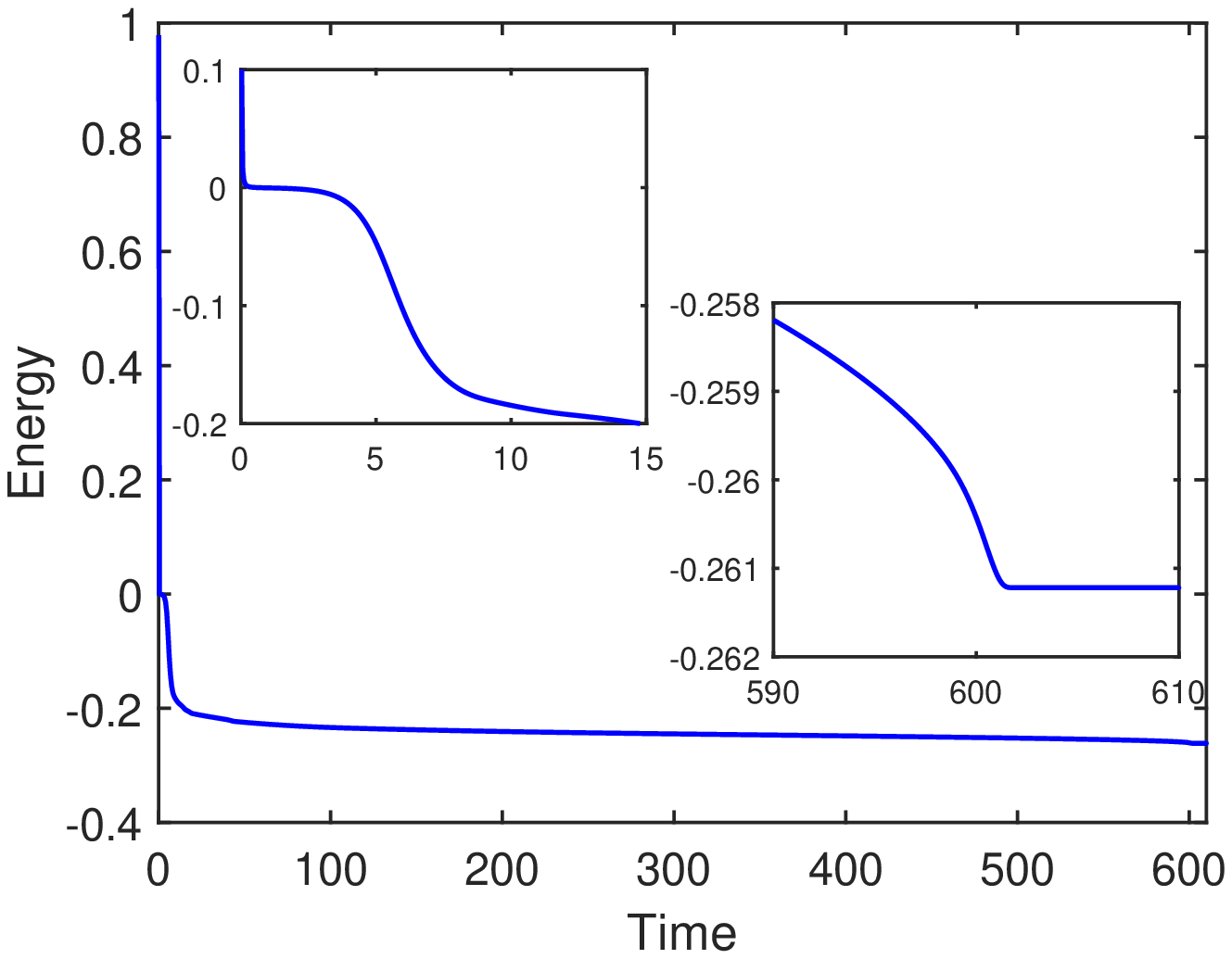}}
\caption{Evolutions of the supremum norm (left) and  the energy (right)
for the coarsening dynamics of the Flory--Huggins potential case.}
\label{fig_coarsen_log2}
\end{figure}

\section{Conclusion}
\label{sect_conclusion}

In this paper, we study MBP-preserving and energy dissipative schemes for the Allen--Cahn type equations
by combining the ESAV approach with stabilizing technique.
We present first- and second-order sESAV schemes
and prove their MBP preservation, energy dissipation, and error estimates.
The main results and observations include two aspects.
First, we choose the ESAV  approach rather than the classic SAV approach,
since the coefficient ($g(u^n,s^n)$ or $ g(\widehat{u}^{n+\frac{1}{2}},\widehat{s}^{n+\frac{1}{2}})$)
of the nonlinear term is positive automatically in the former one
while the sign of the corresponding coefficient is uncertain for the later one.
Second, to guarantee the MBP-preserving property,
we add the stabilization term as an extra artificial term,
that is, add and subtract a linear term in the scheme instead of a quadratic term in the energy functional;
they are equivalent mutually for the classic stabilization or convex splitting method,
but not for the SAV approach. Moreover, we find that
the MBP preservation and the energy dissipation of the sESAV schemes
can be established in parallel and independently,
unlike the purely stabilized semi-implicit scheme discussed in \cite{TaYa16}
where the MBP is needed first to bound the nonlinear term
in the proof of the stability with respect to the original energy.
Since the schemes we studied are all one-step methods,
adaptive time-stepping strategies (such as \cite{QiaoZhTa11})
can be inherently adopted to accelerate the computation.

Some generalizations can be carried out by replacing the Laplace operator in \eqref{AllenCahn} by some analogues,
for instance,
the nonlocal diffusion \cite{DuGuLeZh12} and the fractional Laplace operators \cite{SmakoKiMa93}
which also satisfy the semigroup property
with their discretizations satisfying the analogues of Lemma \ref{lem_lapdiff}. Furthermore, the proposed stabilizing approaches in this paper
can be naturally extended to  many other type of gradient flow problems, which can be handled by the existing SAV schemes.
For example, the fourth-order Cahn--Hilliard equation is the $H^{-1}$ gradient flow of the energy functional \eqref{energy},
and satisfies the same energy dissipation law as the Allen--Cahn equation.
The MBP is not valid anymore, but the solution is still $L^\infty$ stable.
In the similar spirit of this paper, it is interesting to develop the sESAV schemes for the Cahn--Hilliard equation,
and the discrete $L^\infty$ stability of the sESAV solution
can be established by combining the high-order consistency analysis and stability estimate,
as done in \cite{GuWaWi14,LiQiWa21}, which will also be one of our future works.


%
%



\end{document}